\documentclass[reqno]{amsart}

\usepackage[arrow, matrix, curve]{xy}

\usepackage{amssymb}

\usepackage{tikz}

\usepackage{hyperref}
 
\numberwithin{equation}{section}
\numberwithin{figure}{section}

\newcommand{\B}{\mathcal B}

\newcommand{\D}{\mathcal D}

\newcommand{\calL}{\mathcal L}

\newcommand{\calS}{\mathcal S}

\newcommand{\bC}{\mathbb C}
\newcommand{\bN}{\mathbb N}
\newcommand{\bQ}{\mathbb Q}
\newcommand{\bR}{\mathbb R}
\newcommand{\bZ}{\mathbb Z}

\newcommand{\ba}{\overline{a}}
\newcommand{\bb}{\overline{b}}

\newcommand{\bx}{\overline{x}}
\newcommand{\by}{\overline{y}}

\newcommand{\an}{\mathrm {an}}
\newcommand{\acl}{\mathrm {acl}}

\newcommand{\eq}{\mathrm {eq}}

\newcommand{\Th}{\mathrm {Th}}
\newcommand{\tp}{\mathrm {tp}}
\newcommand{\rad}{\mathrm {rad}}
\newcommand{\dist}{\mathrm {dist}}
\newcommand{\divi}{\mathrm {div}}

\newcommand{\ph}{\varphi}

\newcommand\vc{\operatorname{vc}}
\newcommand\VC{\operatorname{VC}}
\newcommand\IND{\operatorname{IN}}

\newcommand\RCVF{\operatorname{RCVF}}
\newcommand\ACVF{\operatorname{ACVF}}
\newcommand\rings{\operatorname{rings}}
\newcommand\pCF{\textup{$p$CF}}

\newcommand\MRk{\operatorname{MR}}
\newcommand\URk{\operatorname{U}}

\newcommand\breadth{\operatorname{breadth}}
\newcommand\width{\operatorname{width}}
\newcommand\height{\operatorname{height}}

\newcommand\Sh{\operatorname{Sh}}

\newcommand\alt{\operatorname{alt}}
\newcommand\ind{{\operatorname{ind}}}

\newcommand\sign{\operatorname{sign}}

\newcommand\restrict{\!\upharpoonright\!}

\newcommand{\abs}[1]{\lvert#1\rvert}

\DeclareMathAlphabet{\mathbf}{OML}{cmm}{b}{it}

\theoremstyle{plain}
\newtheorem{theorem}{Theorem}[section]

\newtheorem{lemma}[theorem]{Lemma}
\newtheorem{proposition}[theorem]{Proposition}
\newtheorem{corollary}[theorem]{Corollary}
\theoremstyle{definition}

\newtheorem{definition}[theorem]{Definition}

\theoremstyle{remark}
\newtheorem*{remark*}{Remark}
\newtheorem*{remarks*}{Remarks}

\newtheorem*{question}{Question}

\newtheorem{remark}[theorem]{Remark}
\newtheorem*{problem*}{Problem}
\newtheorem*{example*}{Example}
\newtheorem*{examples*}{Examples}
\newtheorem*{claim*}{Claim}

\newtheorem{example}[theorem]{Example}

\newcommand{\claim}[2][\!\!]{\medskip\noindent {\it Claim #1.} {#2}}

\begin{document}

\title[VC Density in some NIP Theories, I]{Vapnik-Chervonenkis Density in some Theories without the Independence Property, I}

\dedicatory{For Lou van den Dries, on his 60th birthday.}

\author[Aschenbrenner]{Matthias Aschenbrenner}
\address{Department of Mathematics \\
University of California, Los Angeles \\ 
Box 951555 \\
Los Angeles, CA 90095-1555, U.S.A.}
\email{matthias@math.ucla.edu}

\author[Dolich]{Alf Dolich}
\email{}
\address{Department of Mathematics\\
East Stroudsburg University\\
Science \& Technology Center\\
Room 118\\
East Stroudsburg, PA 18301, U.S.A.
}

\author[Haskell]{Deirdre Haskell}
\address{Department of Mathematics and Statistics\\
	  McMaster University \\
	  1280 Main St W \\
         Hamilton ON L8S 4K1, Canada} 
\email{haskell@math.mcmaster.ca}

\author[Macpherson]{Dugald Macpherson}
\address{School of Mathematics\\
University of Leeds\\
Leeds LS2 9JT, U.K.}
\email{h.d.macpherson@leeds.ac.uk}

\author[Starchenko]{Sergei Starchenko}
\address{Department of Mathematics \\
University of Notre Dame\\
255 Hurley Building\\
Notre Dame, IN 46556-4618, U.S.A.}
\email{starchenko.1@nd.edu}

\date{September 2011}

\begin{abstract}
We recast the problem of calculating  Vapnik-Chervonenkis~(VC)~density into one of counting types, and thereby calculate bounds (often optimal) on the VC~density for some weakly o-minimal, weakly quasi-o-minimal, and $P$-minimal  theories.
\end{abstract}

\maketitle

\tableofcontents

\section{Introduction}

\noindent
The notion of VC dimension, which arose in probability theory in the work of Vapnik and Chervonenkis \cite{vc}, was first drawn to the attention of model-theorists by Laskowski \cite{l}, who observed that a complete first-order theory does not have the independence property (as introduced by Shelah \cite{S1}) if and only if, in each model, each  definable family of sets has finite VC dimension. 
With this observation, Laskowski easily gave several examples of classes of sets with finite VC dimension, by
noting well-known examples of theories without the independence property. This line of thought was pursued by Karpinski and Macintyre \cite{km1},
who calculated explicit bounds on the VC dimension of definable families of sets in some o-minimal structures (with an eye towards applications to neural networks), which were polynomial in the number of parameter variables.
In a further paper \cite{km2}, they observe that their arguments also lead to a linear bound on the VC~\emph{density}
of definable families of sets in some o-minimal structures. They ask whether similar (linear) bounds hold for the $p$-adic numbers (whose theory also does not have the independence property). The bound in the o-minimal case in \cite{km2} was established independently, using a more combinatorial approach, 
by Wilkie (unpublished), and more recently, also by Johnson and Laskowski~\cite{JL}.

In this paper we give a sufficient criterion (Theorem~\ref{VCdensity, 2}) on a first-order theory for the VC~density of a definable family of sets to be bounded by a linear function in 
the number of parameter variables, and show that the criterion is satisfied by several theories of general interest, including  the theory of the $p$-adics and all weakly o-minimal theories. 
In a sequel to this paper \cite{ADHMS} we give  different arguments to get similar bounds in a variety of other examples where our criterion does not apply.
Before we state our main results, we introduce our setup and review some definitions and basic facts.
We hope that the present paper (unlike its sequel \cite{ADHMS}) can be read with only little technical knowledge of model theory beyond basic first order logic. The first few chapters of \cite{Hodges} or \cite{Marker} or similar texts should provide sufficient background for a prospective reader.

\subsection{VC dimension and VC density}
Let $X$ be an infinite set and $\calS$ be a non-empty collection of subsets of $X$. Given $A\subseteq X$, we say that a subset $B$ of $A$ is \emph{cut out  by $\calS$} if $B=S\cap A$ for some $S\in\calS$;
we let $\mathcal S\cap A:=\{S\cap A:S\in\mathcal S\}$ be the collection of subsets of $A$ cut out  by $\mathcal S$.
We say that \emph{$A$ is shattered by $\calS$} if every subset of $A$ is 
cut out  by some element of $\calS$. The collection $\calS$ is said to be a \emph{VC class} if there is a non-negative integer $n$ such 
that no subset of $X$ of size $n$ can be shattered by $\calS$. In this case, the \emph{VC dimension of $\calS$} is the largest $d\geq 0$ such that some set of size $d$ is shattered by
 $\calS$. We denote by $\pi_{\calS}(n)$
the maximum, as $A$ varies over subsets of $X$ of size $n$, of the numbers 
of subsets of $A$ that can be cut out  by $\calS$; that is,
$$\pi_{\calS}(n):=\max\left\{\abs{\mathcal S\cap A}:A\in{X\choose n}\right\}.$$
(Here and below, $X\choose n$ denotes the set of $n$-element subsets of $X$.) The function $n\mapsto\pi_{\mathcal S}(n)$ is called the \emph{shatter function} of $\mathcal S$.
Clearly $0\leq \pi_{\calS}(n)\leq 2^n$ for every $n$, and
if $\calS$ is not a VC class, then $\pi_{\calS}(n) = 2^n$ for every $n$. However,
if $\calS$ is a VC class, of 
VC dimension $d$ say,  then by a fundamental observation of Sauer~\cite{sauer} (independently made in \cite{S2} and, implicitly, in \cite{vc}), the function $n\mapsto \pi_{\calS}(n)$ is bounded above by a polynomial in $n$ of degree~$d$. (In fact, for $d,n\geq 1$ one has $\pi_{\calS}(n)\leq (en/d)^d$, where $e$ is the base of the natural logarithm.) Hence it makes sense to define the \emph{VC~density} of a VC~class $\calS$ as the infimum of all reals $r\geq 0$ such that $\pi_{\calS}(n)/n^r$ is bounded for all positive $n$.
It turns out that in many case, the VC~density (rather than the VC~dimension) is the decisive measure for the combinatorial complexity of a family of sets. For example, the VC~density of $\mathcal S$ governs the size of packings in $\mathcal S$ with respect to the Hamming metric (\cite{Haussler}, see also \cite[Lemma~2.1]{Matousek-discrepancy}), and is intimately related to the notions of entropic dimension \cite{Assouad} and discrepancy \cite{MWW}. We refer to the surveys \cite{Matousek-geometric set systems, FP} for uses of VC~density in combinatorics.

\subsection{VC dimension and VC density of formulas}
Let $\calL$ be a first-order language.
In an $\calL$-structure $\mathbf M$, a natural way to generate a collection of subsets of $M^m$ is to take the family of sets defined by a formula, as the 
parameters vary. 
Given a tuple $x=(x_1,\dots,x_m)$ of pairwise distinct variables we denote by $\abs{x}:=m$ the \emph{length} of~$x$. 
We often need to deal with $\mathcal L$-formulas whose free variables have been separated into \emph{object} and \emph{parameter} variables. We use the notation $\varphi(x;y)$ to indicate that the free variables of the $\mathcal L$-formula $\varphi$ are contained among the components of the tuples $x=(x_1,\dots,x_m)$ and $y=(y_1,\dots,y_n)$ of pairwise distinct variables (which we also assume to be disjoint). Here the $x_i$ are thought of as the object variables and the $y_j$ as the parameter variables. We refer to $\varphi(x;y)$ as a \emph{partitioned $\mathcal L$-formula.} 

In the rest of this introduction we let $\mathbf M$ be an infinite $\mathcal L$-structure.
Let $\ph(x;y)$ be a partitioned $\calL$-formula, $m=\abs{x}$, $n=\abs{y}$, and denote by 
$$\calS_\ph = \big\{ \ph^{\mathbf M}(M^{m};b) : b\in M^{n}\big\}$$ 
the family of subsets of $M^m$ defined by $\ph$ in $\mathbf M$ using parameters 
ranging over $M^n$. We call $\calS_\ph$ a definable family of sets (in $\mathbf M$).
We say that \emph{$\ph$ defines a VC class in $\mathbf M$} if $\calS_\ph$ is a VC class; in this case the \emph{VC dimension of $\ph$ in $\mathbf M$} is the VC dimension of the collection $\calS_\ph$ of subsets of $M^m$, and similarly one defines the \emph{VC density of $\ph$ in $\mathbf M$.} Since the shatter function $\pi_\ph=\pi_{\mathcal S_\ph}$ of $\mathcal S_\ph$ only depends on the elementary theory of~$\mathbf M$ (see Lemma~\ref{lem:shatter function and elementary equivalence} below), given a complete $\mathcal L$-theory $T$ with no finite models, we may also speak of the 
\emph{shatter function of $\ph$ in $T$} as well as
\emph{VC dimension of $\ph$ in $T$} and the \emph{VC density of $\ph$ in $T$.}

\subsection{NIP theories}
A partitioned $\calL$-formula $\ph(x;y)$ as above is said to have the \emph{independence property} for $\mathbf M$  if for every $t\in\bN$ there are $b_1,\ldots,b_t\in M^n$ such that for every $S \subseteq \{1,\ldots,t\}$ there is $a_S\in M^m$ such that for all $i\in \{1,\ldots,t\}$, $\mathbf M\models \ph(a_S;b_i) \Longleftrightarrow i\in S$. The structure $\mathbf M$ is said to have the independence property if some $\calL$-formula has the independence property for $\mathbf M$, and not to have the independence property (or to be  \emph{NIP} or \emph{dependent}) otherwise. 
By a classical result of Shelah \cite{S1} (with other proofs in \cite{Ku2, l,Poizat}), for $\mathbf M$ to be NIP it is actually sufficient that no formula $\varphi(x; y)$ with $\abs{x}=1$ has the independence property for $\mathbf M$. NIP is implied by (but not equivalent to) another prominent tameness condition on first-order structures called \emph{stability}: An $\calL$-formula $\ph(x;y)$ is said to be \emph{unstable} for $\mathbf M$ if for every $t\in\bN$ there are $a_1,\dots,a_t\in M^m$ and $b_1,\dots,b_t\in M^n$ such that $\mathbf M\models \ph(a_i;b_j) \Longleftrightarrow i\leq j$, for all $i,j\in\{1,\dots,t\}$.
The $\mathcal L$-structure $\mathbf M$ is called \emph{unstable} if some $\calL$-formula $\ph$ is unstable for $\mathbf M$; and ``stable'' (for formulas and structures) is synonymous with ``not unstable.''

Laskowski's observation \cite{l} is that an $\calL$-formula defines a VC class in $\mathbf M$ if and only if it does not have the independence property for $\mathbf M$. 
In fact, given a collection $\mathcal S$ of subsets of a set $X$,
 define the \emph{dual shatter function}\/  of $\mathcal S$ as the function $n\mapsto \pi^*_{\mathcal S}(n)$ whose value at $n$ is the maximum number of equivalence classes defined by an $n$-element subfamily $\mathcal T$ of $\mathcal S$, where two elements of $X$ are said to be equivalent with respect to $\mathcal T$ if they belong to the same sets of $\mathcal T$. Then a given partitioned $\mathcal L$-formula $\varphi(x;y)$ has the independence property precisely if $\pi^*_{\mathcal S_\varphi}(n)=2^n$ for every $n$. The dual shatter function of $\mathcal S_\varphi$ is really a shatter function in disguise: it agrees with the shatter function of $\mathcal S_{\varphi^*}$ where $\varphi^*(y;x):=\varphi(x;y)$ is the \emph{dual} of the partitioned formula $\varphi$. (See Section~\ref{sec:model-theoretic context}.)

A complete $\calL$-theory $T$ is said to have the independence property if some model of it does, and is said not to have the independence property (or to be NIP) otherwise.
Thus a complete $\calL$-theory  $T$ is NIP if and only if  every $\calL$-formula defines a VC class in every model of $T$. Many theories arising in mathematical practice turn out to be NIP:
By \cite{S1}, all stable theories (i.e., complete theories all of whose models are stable) are NIP; so, for example, 
algebraically closed (more generally, separably closed) fields,  differentially closed fields, modules, or free groups furnish examples of NIP structures. Furthermore, o-minimal (or more generally, weakly o-minimal) theories are NIP \cite{l,MMS}.
By \cite{gur} any ordered abelian group has NIP theory.
Certain important theories of henselian valued fields are NIP, for example, the completions of the theory of algebraically closed valued fields  and the theory of the field of $p$-adic numbers (and also their  rigid analytic  and $p$-adic subanalytic expansions, respectively). In fact, in the language of rings with a predicate for the valuation ring,  an unramified henselian valued field of characteristic $(0,p)$ is NIP if and only if its residue field is NIP \cite{Belair}. 
Similarly, henselian valued fields of characteristic $(0,0)$ and algebraically maximal Kaplansky fields of characteristic $(p,p$) are NIP iff their residue fields are NIP~\cite{Belair-Bousquet, Belair}.

On the other hand, each pseudofinite field (infinite model of the theory of all finite fields) is not NIP \cite{Duret}, since it defines the (Rado) random graph.

\subsection{Uniform bounds on VC density}
This paper is motivated by the following question: {\it Given a NIP theory $T$, can one find an upper bound, in terms of $n$ only, on the VC densities \textup{(}in $T$\textup{)} of all $\calL$-formulas $\varphi(x;y)$ with $\abs{y}=n$?} 
The intuition behind this question is, of course, that the complexity of a family $\mathcal S_\varphi$ of sets defined by a first-order formula $\varphi(x;y)$ in a NIP structure should be governed by the number $n$ of freely choosable parameters. 
Note that the minimum possible bound is $\abs{y}=n$: for if $\varphi(x;y)$, where $x$ is a single variable, is the formula $x=y_1\vee\cdots\vee x=y_n$, then the subsets of $M$ cut out  by $\mathcal S_\varphi$ are exactly the non-empty subsets of $M$ of cardinality at most~$n$, so $\varphi(x;y)$ has VC density $n$ (in any complete theory).  
We  note here in passing that the VC density of a formula $\varphi$ in a NIP theory may take fractional values, 
and that the shatter function of $\mathcal S_\varphi$, though not growing faster than polynomially, is not {\it asymptotic} to a real power function in general.
See Section~\ref{sec:calculations} below, where we explicitly compute the VC~density of certain incidence structures (related to the Sz\'emeredi-Trotter~Theorem) and of the edge relation in Spencer-Shelah random graphs, and investigate the asymptotics of a  
shatter function in the infinitary hypercube.

\medskip

In this paper we
employ VC~duality to translate the  problem of bounding the VC~density of a formula $\varphi$ into the task of counting $\varphi^*$-types over finite parameter sets, which then can be treated by model-theoretic machinery.
Viewing  VC~density as a bound on a number of types also illuminates the connection with a strengthening of the  NIP concept, which is that of \emph{dp-minimality.} (See Section~\ref{sec:relationship} below for a definition.) Dolich, Goodrick and Lippel \cite{dl} have observed that, if, in a theory, the dual VC~density of any $\mathcal L$-formula in a single object variable is less than $2$, then the theory in question is dp-minimal.  (No counterexample to the converse of this implication seems to be known.) 

\medskip

We now  state our main results.
First, an optimal bound on density is obtained for weakly o-minimal theories (see Theorem~\ref{weaklyomin} below). Recall that a complete theory $T$ in a language containing a binary relation symbol ``$<$'' which expands the theory of linearly ordered sets is called \emph{weakly o-minimal} if in every model of $T$, each partitioned $\mathcal L$-formula $\ph(x;y)$ with $\abs{x}=1$ defines a finite union of convex sets.
(See \cite{MMS} for more on this notion, which generalizes the probably more familiar concept of an o-minimal theory, cf.~\cite{vdDries-Tame}.)

\begin{theorem}\label{thm:weakly o-min}
Suppose $\mathcal L$ contains a binary relation symbol ``$<$'', interpreted in $\mathbf M$ as a linear ordering. If $T=\Th(\mathbf M)$ is weakly o-minimal, then every $\calL$-formula $\varphi(x;y)$ has VC density at most $n=\abs{y}$ in $T$ \textup{(}in fact, $\pi_{\varphi}(t)=O(t^n)$\textup{)}.
\end{theorem}

This bound is the same as  that obtained by Karpinski-Macintyre \cite{km1} for o-minimal expansions of the reals, or by Wilkie  and by Johnson-Laskowski \cite{JL} for all o-minimal structures.  The motivating example of a theory which is weakly o-minimal but not o-minimal  is the theory of real closed valued fields, that is,  real closed fields equipped with a predicate for a proper convex valuation ring.
In fact, the methods of Karpinski and Macintyre can also be adapted to give the correct density bounds for this and certain other weakly o-minimal expansions of real closed fields \cite{hm3}. Some interesting weakly o-minimal theories to which these methods do not readily adapt may be found in \cite{AvdD, Kuhlmann}.
Our approach to Theorem~\ref{thm:weakly o-min},  via definable types, was partly inspired by the use of Puiseux series in  \cite{BPR,PR}. 

\medskip

Let $\ACVF$ denote the theory of (non-trivially) valued algebraically closed fields, in the ring language expanded by a predicate for the valuation divisibility. This has completions 
$\ACVF_{(0,0)}$
(for residue characteristic $0$), $\ACVF_{(0,p)}$ (field characteristic $0$, residue characteristic $p$), and $\ACVF_{(p,p)}$ (field characteristic $p$). Because $\ACVF_{(0,0)}$ is interpretable in $\RCVF$, our methods give (non-optimal) density bounds  for $\ACVF_{(0,0)}$ (Corollary~\ref{acvf}). However, they give no information on density in the theories $\ACVF_{(0,p)}$ and $\ACVF_{(p,p)}$. The problems arise essentially because a definable set in $1$-space in $\ACVF$ is a finite union of `Swiss cheeses' but 
we have no way of choosing a particular Swiss cheese. This means that the definable types technique in our main tool (Theorem~\ref{VCdensity, 2}) breaks down. On the other hand, our methods do yield:

\begin{theorem}\label{thm:p-adic}
Suppose $\mathbf M=\bQ_p$ is the field of $p$-adic numbers,  construed as a first-order structure in Macintyre's language $\mathcal L_p$. Then the VC density of every $\calL_p$-formula $\ph(x;y)$ is at most $2\abs{y}-1$. 
\end{theorem}

The same result holds for the subanalytic expansions of $\bQ_p$ considered by Denef and van den Dries~\cite{dd}. (Theorem~\ref{Pmin} and Remark~\ref{analytic}.) Key tools available here, but not in the case of $\ACVF$, are cell decomposition and  the existence of definable Skolem functions. We do not know whether the bound in Theorem~\ref{thm:p-adic} is optimal.

\medskip

The investigation of the fine structure of type spaces over finite parameter sets in NIP theories is only just beginning, and the present paper can be seen as a first step in studying one particular measure (VC~density) for their complexity.
Applications of the results in this paper to transversals of definable families in NIP theories will appear in a separate manuscript, under preparation by the first- and last-named authors.

As remarked above, all stable theories are NIP, so it also makes sense to investigate VC~density in stable theories. In a sequel of the present paper \cite{ADHMS} we obtain bounds on VC~density in certain finite $\URk$-rank theories (including all complete theories of finite Morley rank expansions of infinite groups).

\medskip

We close off this introduction by pointing out that besides being of intrinsic interest, uniform bounds on VC~density of first-order formulas (as obtained in this paper) often also help to explain why certain well-known effective bounds on the complexity of geometric arrangements, used in computational geometry,  are polynomial in the number of objects involved. 
For example, the bound on the number of semialgebraically connected components of realizable sign conditions on  polynomials over real closed fields from \cite{BPR,PR} breaks up into a topological  and a combinatorial part, where the polynomial nature of the latter may be seen as a consequence of Theorem~\ref{thm:weakly o-min}:

\begin{example*}
Let $R$ be a real closed field, $\mathcal P=(P_1,\dots,P_s)$ be a tuple of polynomials from $R[X]=R[X_1,\dots,X_k]$, each of degree at most $d$. 
A sign condition for $\mathcal P$ is an $s$-tuple $\sigma\in\{-1,0,+1\}^s$, and we say that $\sigma$ is realized in a subset $V$ of $R^k$ if 
$$\sigma_{V} := \big\{ a\in V: (\sign P_1(a),\dots,\sign P_s(a)) = \sigma \big\}$$
is non-empty. Theorem~\ref{thm:weakly o-min} in the  semialgebraic case  yields: {\it if $V$ is an algebraic set defined by polynomials of degree at most $d$, then
the number of sign conditions for $\mathcal P$ realized in $V$ is at most $Cs^m$, where $m=\dim(V)$ and the constant $C=C(d,k)$ only depends on $d$ and $k$.} 

To see this  recall that by cell decomposition, $V$ is a finite union of semialgebraic subsets of $R^k$ each of which is semialgebraically homeomorphic to some $R^n$; moreover, this decomposition (and the resulting homeomorphism) can be chosen uniformly in the parameters:
Every zero set of polynomials from $R[X]$ of degree at most $d$ is the zero set of $M$ such polynomials, where 
$M={k+d\choose d}$ is the dimension of the $R$-linear subspace of $R[X]$  consisting of the polynomials of degree at most~$d$; thus we may take a semialgebraic (in fact, algebraic) family $(V_b)_{b\in R^{N}}$,  where $N=M^2$, whose fibers $V_b$ are the algebraic subsets of $R^k$ defined by polynomials of degree at most~$d$. Then
there are finitely many semialgebraic families $(V_b^{(i)})_{b\in R^N}$ of subsets of $R^k$ and for each $i$ there is a semialgebraic family $(F^{(i)}_b)_{b\in R^N}$ of maps
such that for each $b\in R^N$ we have
$V_b=\bigcup_i V_b^{(i)}$, and $F^{(i)}_b$ is a homeomorphism $R^{m^{(i)}}\to V_b^{(i)}$, for some $m^{(i)}$.

Fix some $i$ and write $m=m^{(i)}$. Let $\nu=(\nu_1,\dots,\nu_k)$ range over $\bN^k$, with $\abs{\nu}=\nu_1+\cdots+\nu_k$, and suppose
$y=(y_\nu)_{\abs{\nu}\leq d}$, so $y$ has length $M$.
Let $P(X;y)$ be the general polynomial in the indeterminates $X$ of degree at most $d$ with coefficient sequence $y$; so every $P_j$ is of the form $P_j=P(X;b_j)$ with $b_j\in R^M$.
Suppose also $x=(x_1,\dots,x_m)$, and let $z$ be a tuple of new variables of length~$N$,  let $z'$ be a single new variable, and let $\varphi^{(i)}(x;y,z,z')$ be a formula in the language of ordered rings which expresses that 
$P(F^{(i)}_z(x);y)$ and $z'$ have the same sign. 
So, e.g., for $a\in R^m$, $b\in R^N$ we have $R\models\varphi^{(i)}(a;b_j,b,1)$ iff $P_j(F^{(i)}_b(a))>0$.
In this way we see that the number of sign conditions for $\mathcal P$ realized in $V_b^{(i)}$ is bounded by $\pi^*_{\varphi^{(i)}}(3s)$ and thus is $O(s^m)$ by Theorem~\ref{thm:weakly o-min}, where the implicit constant only depends on $\varphi^{(i)}$ and hence on $d$ and $k$. This yields the claim highlighted above.
(Of course we have been very nonchalant with the constants. Indeed,  \cite{BPR} shows the more precise result that the sum of the number of semialgebraically connected components of the sets $\sigma_V$, where $\sigma$ ranges over all sign conditions for $\mathcal P$ realized in $V$, is bounded by $(O(d))^k{s\choose m}$.)
\end{example*}

A simpler example is the number of non-empty sets definable by equalities and inequalities of a finite collection of polynomials over an algebraically closed field:

\begin{example*}
Here we let $\nu=(\nu_1,\dots,\nu_m)$ range over $\bN^m$,  and suppose
$y=(y_\nu)_{\abs{\nu}\leq d}$. 
Let $\varphi(x;y)$ be the partitioned formula
$$\sum_{\abs{\nu}\leq d} y_\nu x^\nu=0$$ 
in the language $\mathcal L$ of rings, and fix an algebraically closed field $K$. Then $\mathcal S_\varphi=\mathcal S_\varphi^K$ is the collection of all zero sets (in $K^m$) of polynomials in $m$ indeterminates with coefficients in $K$ having degree at most $d$. Hence $\pi^*_{\mathcal S_\varphi}(t)$ is the maximum number of non-empty Boolean combinations of $t$ such hypersurfaces.
In the sequel  of our paper (see \cite[Theorem~1.1]{ADHMS}) we will show that the shatter function of any partitioned $\mathcal L$-formula with $m$ parameter variables (such as $\varphi^*$) is  $O(t^m)$ in $\Th(K)$; hence $\pi^*_{\varphi}(t)=\pi_{\varphi^*}(t)=O(t^m)$. (In fact, \cite{JS} proves that
 $\pi^*_{\varphi}(t)\leq\sum_{k=0}^m {t\choose k}d^k$ for every $t$, and this bound is asymptotically optimal.)
\end{example*}

\subsection{Organization of the paper}
In the preliminary Section~\ref{sec:preliminaries} we set the scene by recalling the definitions and basic facts concerning VC~dimension and VC~density in a general combinatorial setting. In Section~\ref{sec:model-theoretic context}  we then move to the model-theoretic context; in particular we introduce the VC~density function of a complete theory without finite models, and the (dual) VC~density of a finite set of formulas.
In Section~\ref{sec:calculations} we give some interesting examples of formulas in NIP theories for which we can explicitly compute their VC~density or determine the asymptotic behavior of their shatter function.
In Section~\ref{sec:VCm property} we introduce the $\VC{}d$ property (a refinement of Guingona's notion of uniform definability of types over finite sets) and get our main tool for counting types (Theorem~\ref{VCdensity, 2}) in place, which is then employed, in Section~\ref{sec:examples of VCm theories}, to prove Theorem~\ref{thm:weakly o-min}  from above. 
A strengthening of the $\VC{}d$ property is defined and established for the $p$-adics in Section~\ref{sec:strong VCm}, thus proving Theorem~\ref{thm:p-adic}. We refer to the introductions of each section for a more detailed description of their contents.

\subsection{Notations and conventions}
In this paper, $d$, $k$, $m$, $n$ range over the set $\bN:=\{0,1,2,\dots\}$ of natural numbers. We set $[n]:=\{1,\dots,n\}$. Given a set $X$,
we write $2^X$ for the power set of $X$, and we let $X\choose n$ denote the set of $n$-element subsets of~$X$ and ${X\choose \leq n}:={X\choose 0}\cup{X\choose 1}\cup\cdots\cup{X\choose n}$  the collection of subsets of $X$ of cardinality at most $n$.

\subsection{Acknowledgments}
Part of the work on this paper was done while some of the authors were participating in the thematic program on O-minimal Structures and Real Analytic Geometry at the Fields Institute in Toronto (Spring 2009), and in the 
Durham Symposium on
New Directions in the Model Theory of Fields (July 2009), organized by the London Mathematical Society  and funded by EPSRC~grant EP/F068751/1. The support of these institutions is gratefully acknowledged.
Aschen\-bren\-ner was partly supported by NSF grant DMS-0556197. He would also express his gratitude to Gerhard W\"oginger for suggesting the example in Section~\ref{sec:hypercube}, and to Andreas Baudisch and Humboldt-Universit\"at Berlin for their hospitality during Fall~2010. Haskell's research was supported by  NSERC grant~238875. Macpherson acknowledges support by EPSRC grant EP/F009712/1. Starchenko was partly supported by NSF grant DMS-0701364.

\section{VC Density}\label{sec:preliminaries}

\noindent
In this section we introduce various numerical parameters associated to abstract families of sets: VC~dimension, VC~density, and independence dimension,  and we recall the well-known phenomenon of ``VC~duality'' hinted at already in the introduction (which, in particular, allows us to relate VC~dimension with independence dimension). An important role in later sections is played by a new parameter associated to a set system defined here, which we call breadth, and which is the focus of the last part of this section.

\subsection{VC dimension and VC density}
A \emph{set system} 
is a pair $(X,\mathcal S)$ consisting of a set~$X$ and a collection $\mathcal S$ of subsets of $X$. We call $X$ the \emph{base set} of the set system~$(X,\mathcal S)$, and we sometimes also speak of a \emph{set system $\mathcal S$ on $X$.} 
Given a set system $(X,\mathcal S)$ and a
set $A\subseteq X$, we let $\mathcal S\cap A:=\{S\cap A:S\in\mathcal S\}$ and 
call $(A,\mathcal S\cap A)$ \emph{the set system on $A$ induced by $\mathcal S$.}
Let now $\mathcal S$ be a set system on an infinite set $X$.
The function $\pi_{\mathcal S}\colon\bN\to\bN$ given by
$$\pi_{\mathcal S}(n):=\max\left\{\abs{\mathcal S\cap A}:A\in{X\choose n}\right\}$$
is called the \emph{shatter function of $\mathcal S$.} 
We have $0\leq\pi_{\mathcal S}(n)\leq 2^n$ and $\pi_{\mathcal S}(n)\leq \pi_{\mathcal S}(n+1)$ for all $n$.  
Note that if $Y\supseteq X$ then $\pi_{\mathcal S}$ does not change if $\mathcal S$ is considered as a set system on~$Y$. (This justifies our choice of notation for the shatter function, suppressing the base set $X$ of our set system.)

One says that $A\subseteq X$ is \emph{shattered by $\mathcal S$} if $\mathcal S\cap A=2^A$. If $\mathcal S$ is non-empty, then we define the \emph{VC dimension  of $\mathcal S$,} denoted by $\VC(\mathcal S)$, as the supremum (in $\mathbb N\cup\{\infty\}$) of the sizes of all finite subsets of $X$ shattered by $\mathcal S$; so $\VC(\mathcal S)=\infty$ means that arbitrarily large finite subsets of $X$ can be shattered by $\mathcal S$. Equivalently, 
\begin{equation}\label{eq:VC}
\VC(\mathcal S)=\sup\big\{n:\pi_{\mathcal S}(n)=2^n\big\}.
\end{equation}
One says that $\mathcal S$ is a \emph{VC class} if $\VC(\mathcal S)<\infty$.
Note that some sources (e.g., \cite{l}) alternatively define the VC~dimension of $\mathcal S$ to be the minimum $n$ such that no set of size $n$ is shattered by~$\mathcal S$ (i.e., $\VC(\mathcal S)+1$, with $\VC(\mathcal S)$ as given by \eqref{eq:VC}).

\medskip

We have the following fundamental fact about set systems:

\begin{lemma}[Sauer-Shelah] \label{lem:SS}
If $\mathcal S$ has finite VC dimension $d$ \textup{(}so $\pi_{\mathcal S}(n)<2^n$ for $n> d$\textup{)}, then
$$\pi_{\mathcal S}(n)\leq {n\choose\leq d}:={n\choose 0}+\cdots+{n\choose d}\qquad\text{for every $n$.}$$
\end{lemma}

If $n\geq d$, then ${n\choose\leq d}$ is bounded above by $(en/d)^d$ (where $e$ is the base of the natural logarithm).
In particular, either $\pi_{\mathcal S}(n)=2^n$ for every $n$ (if $\mathcal S$ is not a VC class), or $\pi_{\mathcal S}(n)=O(n^{d})$. One may now define the \emph{VC density $\vc(\mathcal S)$ of $\mathcal S$} as the infimum of all
real numbers $r>0$ such that $\pi_{\mathcal S}(n)=O(n^r)$, if there is such an $r$, and $\vc(\mathcal S):=\infty$ otherwise. That is,
$$\vc(\mathcal S)=\limsup_{n\to\infty} \frac{\log \pi_{\mathcal S}(n)}{\log n}.$$
We also define $\VC(\emptyset):=\vc(\emptyset):=-1$. 
Then $\vc(\mathcal S)\leq\VC(\mathcal S)$ by Lemma~\ref{lem:SS}, and
$\vc(\mathcal S)<\infty$ iff $\VC(\mathcal S)<\infty$. The VC density of $\mathcal S$ is also known as the real density \cite{Assouad} or the VC exponent \cite{BG} of $\mathcal S$. It is related to the combinatorial dimension of $\mathcal S$ introduced by Blei \cite{Assouad-2} and to compression schemes for $\mathcal S$ \cite{JL}.

\begin{example*}
Suppose $\mathcal S={X\choose \leq d}$. Then the inequality in the statement of Lemma~\ref{lem:SS} is an equality, and
$\VC(\mathcal S)=\vc(\mathcal S)=d$.
\end{example*}

\begin{example*}
Suppose $X=\bR^d$, and $\mathcal S$ is the collection of all closed affine half-spaces  in $\bR^d$, i.e., sets of the form $\{x\in\bR^d: \langle x,a\rangle\geq\beta\}$ where $a\in\bR^d$, $\beta\in\bR$, and $\langle\ , \ \rangle$ denotes the usual inner product on $\bR^d$. Then $\VC(\mathcal S)=d+1$.
(The proof of this fact is based on Radon's Theorem on convex sets; see \cite[Corollaire~3.5]{Assouad}.)
Moreover, $\vc(\mathcal S)=d$; in fact, 
$\pi_{\mathcal S}(n)=2\sum_{i=0}^d (-1)^{d-i} {n \choose \leq i}$ for every~$n$; see \cite[Theorem~3.1]{Edelsbrunner}. 
\end{example*}

\begin{example*}
Suppose $X=\bR$, $k\geq 1$, and let $\mathcal S$ be the collection whose members are the unions of $k$ disjoint (open) intervals in $\bR$. Then $\VC(\mathcal S)=\vc(\mathcal S)=2k$, in fact, $\pi_{\mathcal S}(n)={n\choose \leq 2k}$ for each $n$. (See \cite[Exercise~11, Chapter~4]{Dudley-uniform}.)
\end{example*}

In all three examples, $\pi_{\mathcal S}$ is actually given by a polynomial of degree $d=\vc(\mathcal S)$. It is worth pointing out that for a VC~class $\mathcal S$, 
in general $\pi_{\mathcal S}$ is not even asymptotic to a real power function; see Section~\ref{sec:not growing like a power} below.

\medskip

Clearly, $\VC$ and $\vc$ are increasing: if $\mathcal S\subseteq\mathcal T\subseteq 2^X$, then $\pi_{\mathcal S}\leq\pi_{\mathcal T}$ and so
$\VC(\mathcal S)\leq\VC(\mathcal T)$ and $\vc(\mathcal S)\leq\vc(\mathcal T)$. 
If $X'$ is an infinite subset of $X$ then $\pi_{\mathcal S\cap X'}\leq \pi_{\mathcal S}$; more generally (see \cite[Proposition~2.2]{Assouad}):

\begin{lemma}\label{lem:inverse image}
Let $X'$ be an infinite set and $f\colon X'\to X$ be a map, and let $f^{-1}(\mathcal S):=\{f^{-1}(S):S\in\mathcal S\}$. Then $\pi_{f^{-1}(\mathcal S)}\leq\pi_{\mathcal S}$, with equality if $f$ is surjective. In particular, 
$\VC(f^{-1}(\mathcal S))\leq \VC(\mathcal S)$ and $\vc(f^{-1}(\mathcal S))\leq\vc(\mathcal S)$, with equality if $f$ is surjective.
\end{lemma}

It is easy to verify that  $\VC(\mathcal S)=0$ if and only if $\abs{\mathcal S}=1$, and $\vc(\mathcal S)=0$ if $\mathcal S$ is finite; in fact, the converse of the latter implication also holds: if $\vc(\mathcal S)<1$, then $\mathcal S$ is finite \cite[Proposition~2.19]{Assouad} (and hence actually $\vc(\mathcal S)=0$). It is also easy to verify (cf.~\cite[Proposition~2.4]{Assouad}) that if $\mathcal S_1$, $\mathcal S_2$ are subsets of $\mathcal S$ with $\mathcal S=\mathcal S_1\cup\mathcal S_2$, then
$\vc(\mathcal S) = \max\{\vc(\mathcal S_1),\vc(\mathcal S_2)\}$.
In particular, $\vc(\mathcal S)$ does not change if we alter finitely many sets from $\mathcal S$. 

\subsection{Independence dimension}\label{sec:IND}
Let $X$ be a set. 
Given subsets $A_1,\dots,A_n$ of $X$, we denote by $S(A_1,\dots,A_n)$ the set of atoms of the Boolean algebra of subsets of $X$ generated by $A_1,\dots,A_n$ (the ``non-empty fields in the Venn diagram of $A_1,\dots,A_n$''); that is, $S(A_1,\dots,A_n)$ is precisely the set of non-empty subsets of $X$ of the form
$$\bigcap_{i\in I} A_i \cap \bigcap_{i\in [n]\setminus I} X\setminus A_i\qquad\text{where $I\subseteq [n]=\{1,\dots,n\}$.}$$ 
Note that $S(A_1,\dots,A_n)$ does not depend on the particular order of the $A_i$, so sometimes we abuse notation and, e.g., write $S(A_i:i=1,\dots,n)$ instead of $S(A_1,\dots,A_n)$.
We have $0\leq \abs{S(A_1,\dots,A_n)}\leq 2^n$, and we say that the sequence $A_1,\dots,A_n$ is \emph{independent} (in $X$) if $\abs{S(A_1,\dots,A_n)}=2^n$, and call $A_1,\dots,A_n$ \emph{dependent} (in $X$) otherwise. 

Suppose now that $\mathcal S$ is a collection of subsets of $X$.
We define $\pi^*_{\mathcal S}\colon\bN\to\bN$ by
$$\pi^*_{\mathcal S}(n):=\max\big\{\abs{S(A_1,\dots,A_n)}:A_1,\dots,A_n\in\mathcal S \big\}.$$
Note that $0\leq \pi^*_{\mathcal S}(n)\leq 2^n$ for each $n$. 
We say that $\mathcal S$ is \emph{independent} (in $X$) if $\pi^*_{\mathcal S}(n)=2^n$ for every $n$, that is, if for every $n$ there is an independent sequence of elements of $\mathcal S$ of length $n$. Otherwise, we say that $\mathcal S$ is \emph{dependent} (in $X$). If $\mathcal S$ is dependent, we define the \emph{independence dimension $\IND(\mathcal S)$ of $\mathcal S$} as the largest $n$ such that $\pi^*_{\mathcal S}(n)=2^n$, and if $\mathcal S$ is independent, we set $\IND(\mathcal S)=\infty$.
If $\mathcal S$ is finite, then clearly $\IND(\mathcal S)\leq \abs{\mathcal S}$. 



\begin{example}\label{ex:dual VCdim at most 1}
$\IND(\mathcal S)\leq 1$ iff for all $S,S'\in\mathcal S$ one of the following relations holds:
$S\cap S'=\emptyset$, $S\subseteq S'$, $S'\subseteq S$, or $S\cup S'=X$.
\end{example}

The function $\pi^*_{\mathcal S}$ is called the \emph{dual shatter function of $\mathcal S$}, since (for infinite $\mathcal S$) one has
$\pi^*_{\mathcal S}=\pi_{\mathcal S^*}$ for a certain set system $\mathcal S^*$ on $X^*=\mathcal S$, called the \emph{dual}\/ of $\mathcal S$ (cf.~\cite[2.7--2.11]{Assouad} or \cite[Section~10.3]{Matousek-Book}). For the same reason, the independence dimension of $\mathcal S$ is sometimes also called the \emph{dual VC dimension}\/ of $\mathcal S$, denoted by $\VC^*(\mathcal S)$. 
The correspondence between $\mathcal S$ and $\mathcal S^*$ is explained in the following subsection. 

\subsection{VC duality}\label{sec:VC duality}
Let $X$ and $Y$ be infinite sets, and let $\Phi\subseteq X\times Y$. For $y\in Y$ we put
$$\Phi_y   := \{ x\in X: (x,y) \in\Phi \},$$
and we set $$\mathcal S_\Phi:=\{\Phi_y:y\in Y\}\subseteq 2^X.$$ 
We also write 
$$\Phi^*\subseteq Y\times X:=\big\{(y,x)\in Y\times X:(x,y)\in\Phi\big\}$$ 
for the dual of the binary relation $\Phi$.
In this way we obtain two set systems $(X,\mathcal S_{\Phi})$ and $(Y,\mathcal S_{\Phi^*})$. 
To simplify notation, we denote the shatter function of $\mathcal S_\Phi$ by $\pi_\Phi$, and its dual shatter function by $\pi^*_\Phi$; similarly for $\Phi^*$ in place of $\Phi$.
One verifies easily that given a finite set $A\subseteq X$, the assignment
$$A' \mapsto \bigcap_{x\in A'} \Phi^*_x \cap \bigcap_{x\in A\setminus A'} Y\setminus \Phi^*_x$$
defines a bijection $$\mathcal S_{\Phi}\cap A \to S(\Phi^*_x:x\in A).$$
This implies:

\begin{lemma}\label{lem:dual shatter}
$\pi_{\Phi} = \pi^*_{\Phi^*}$.
\end{lemma}

We set $\VC(\Phi):=\VC(\mathcal S_\Phi)$, and similarly with $\IND$ and $\vc$ in place of $\VC$.
By the previous lemma, $\VC(\Phi)=\IND(\Phi^*)$, hence
$\mathcal S_\Phi$ is a VC class iff $\mathcal S_{\Phi^*}$ is dependent.
Reversing the role of $\Phi$ and $\Phi^*$ also yields $ \pi_{\Phi^*}=\pi^*_{\Phi}$, hence $\VC(\Phi^*)=\IND(\Phi)$, and $\mathcal S_{\Phi^*}$ is a VC class iff $\mathcal S_{\Phi}$ is dependent. 
The following is also well-known (see, e.g., \cite[2.13~b)]{Assouad}):

\begin{lemma}\label{lem:dual shatter, 2}
$\VC(\Phi)<2^{1+\VC(\Phi^*)}$. \textup{(}In particular $\mathcal S_\Phi$ is a VC class iff $\mathcal S_{\Phi^*}$ is a VC class.\textup{)}
\end{lemma}

\begin{example}\label{ex:dual shatter, 2}
Suppose $\mathcal S_\Phi$ is finite (i.e., $\vc(\Phi)=0$). Then $\mathcal S_{\Phi^*}$ is also finite. (Take $y_1,\dots,y_N\in Y$, where $N=\abs{\mathcal S_\Phi}$, such that $\mathcal S_\Phi=\{\Phi_{y_1},\dots,\Phi_{y_N}\}$.
Let $X_i=\Phi_{y_i}$ and $Y_i=\{y\in Y: \Phi_y=\Phi_{y_i}\}$ for $i\in [N]$; thus $\Phi=X_1\times Y_1\cup\cdots\cup X_N\times Y_N$. Hence for each $x\in X$, $\Phi^*_x$ is a union of $Y_1,\dots,Y_N$, and so there are only finitely many choices for $\Phi^*_x$. Thus $\mathcal S_{\Phi^*}$ is also finite, of size at most $2^N$.)
\end{example}

Clearly every infinite set system $\mathcal S$ on $X$ is of the form $\mathcal S=\mathcal S_{\Phi}$ for some infinite set~$Y$ and some binary relation $\Phi \subseteq X\times Y$: just take $Y=\mathcal S$, $\Phi=\{(x,S):x\in S,\ S\in\mathcal S\}$. The resulting set system $\mathcal S_{\Phi^*}$ on $Y=\mathcal S$ is called the dual $\mathcal S^*$ of $\mathcal S$ in \cite[Section~10.3]{Matousek-Book}. By the above $\VC(\mathcal S^*)=\VC^*(\mathcal S)$.
If $\mathcal S$ is a dependent infinite set system on $X$, then by Lemmas~\ref{lem:SS} and \ref{lem:dual shatter},  there is a real number $r\geq 0$ such that $\pi^*_{\mathcal S}=O(n^r)$, and the infimum of all such $r$ is called the \emph{dual VC~density} of $\mathcal S$, denoted by $\vc^*(\mathcal S)$; note that $\vc(\mathcal S^*)=\vc^*(\mathcal S)$ and $\vc^*(\mathcal S)\leq\VC^*(\mathcal S)$.

\medskip

Given $\Phi\subseteq X\times Y$ we write $\neg\Phi$ for the relative complement $(X\times Y)\setminus \Phi$ of $\Phi$ in~$X\times Y$. We clearly have $\pi^*_{\neg\Phi}=\pi^*_{\Phi}$. It is also easy to show that given $\Phi,\Psi\subseteq X\times Y$ we have $\pi^*_{\Phi\cup\Psi} \leq \pi^*_{\Phi}\cdot\pi^*_{\Psi}$ and hence (using complementation) 
$\pi^*_{\Phi\cap\Psi} \leq \pi^*_{\Phi}\cdot\pi^*_{\Psi}$. By passing to duals and Lemma~\ref{lem:dual shatter},
this yields:

\begin{lemma} \label{lem:vc for boolean combinations}
Let $\Phi,\Psi\subseteq X\times Y$. Then
$$\vc(\neg\Phi)=\vc(\Phi), \quad
  \vc(\Phi\cup\Psi)\leq\vc(\Phi)+\vc(\Psi), \quad
  \vc(\Phi\cap\Psi)\leq\vc(\Phi)+\vc(\Psi).$$
\end{lemma}

VC dimension does not satisfy a similar subadditivity property for unions and intersections (cf.~\cite[Proposition~9.2.8]{Dudley}).
In this way, VC density is better behaved than VC dimension.


\medskip

An important class of relations $\Phi$ such that the associated set system $\mathcal S_\Phi$ is dependent are the stable ones. An \emph{$n$-ladder} for $\Phi$ is a $2n$-tuple $(a_1,\dots,a_n,b_1,\dots,b_n)$ where each $a_i\in X$ and each $b_j\in Y$, such that for all $i,j\in [n]$,
$$(a_i,b_j)\in \Phi \qquad\Longleftrightarrow\qquad i\leq j.$$
If there is an $n$ such that there is no $n$-ladder for $\Phi$, then $\Phi$ is called \emph{stable}, and $\Phi$ is said to be \emph{unstable} otherwise. 
If $\Phi$ is stable then the largest $n$ such that an $n$-ladder for $\Phi$ exists is called the \emph{ladder dimension} of $\Phi$; if $\Phi$ is unstable then we say that the ladder dimension of $\Phi$ is infinite. 
Clearly if $\Phi$ is stable then $\mathcal S_{\Phi}$ is a VC~class (with VC~dimension bounded by the ladder dimension).
It is well-known that $\Phi$ is stable iff $\Phi^*$ is stable (e.g, \cite[Exercise~II.2.8]{Shelah-book}), and that Boolean combinations of stable relations are stable. 

\subsection{Breadth}\label{sec:breadth}
In many cases of interest complicated set systems are generated by simpler collections of subsets, and then
the following lemma (essentially due to Dudley) can be used to show that the resulting set system is dependent. For this let $X$ be a set and~$\mathcal B$ be a collection of subsets of $X$.

\begin{lemma}
Let $N>0$ and suppose $\mathcal S$ is a set system on $X$ such that  each set in $\mathcal S$ is a Boolean combination of at most $N$ sets in $\mathcal B$. Then $\pi^*_{\mathcal S}(t)\leq \pi^*_{\mathcal B}(Nt)$ for each $t$. \textup{(}In particular, if $\mathcal B$ is dependent then so is $\mathcal S$.\textup{)}
\end{lemma}
\begin{proof}
Let $A_1,\dots,A_t\in\mathcal S$, and let each $A_i$ be a Boolean combination of the sets $B_{i1},\dots,B_{iN}\in\mathcal B$. Then the Boolean algebra of subsets of $X$ generated by the sets $A_i$ ($i\in [t]$) is contained in the Boolean algebra generated by the sets $B_{ij}$ ($i\in [t]$, $j\in [N]$), and every atom of the former Boolean algebra contains an atom of the latter.
\end{proof}

 Suppose there is a $d>0$ such that every non-empty intersection  $B_1\cap\cdots\cap B_{n}$ of $n>d$ sets from $\mathcal B$ equals an intersection of a subset consisting of $d$ of the $B_i$. We call the smallest such integer $d>0$ the \emph{breadth} of $\mathcal B$. 
This choice of terminology is motivated by lattice theory: Given a (meet-) semilattice $(L,{\wedge})$,
the smallest $d>0$ (if it exists) such that any meet $b_1\wedge\cdots\wedge b_n$ of $n>d$ elements of $L$ equals the meet of $d$ of the $b_i$ is called the \emph{breadth} of $L$; if there is no such $d$ we say that $L$ has infinite breadth. (See \cite[Section~II.5, Exercise~6, and Section~IV.10]{Birkhoff}.)
So if $\mathcal B$ is closed under (finite) intersection and only contains non-empty subsets of $X$, then the breadth of $\mathcal B$, viewed as a sub-semilattice of $(2^X,{\cap})$, agrees with the breadth of $\mathcal B$ as defined above.
Every set system of finite breadth is dependent:

\begin{lemma}\label{lem:breadth and IND}
$\breadth(\mathcal B)\geq \IND(\mathcal B)$. 
\end{lemma}
\begin{proof}
Suppose $d:=\breadth(\mathcal B)<n:=\IND(\mathcal B)$. Let $B_1,\dots,B_n\in\mathcal B$ such that $\abs{S(B_1,\dots,B_n)}=2^n$. Choose $I\subseteq [n]$ with $\abs{I}=d$ and $\bigcap_{i\in I} B_i=\bigcap_{i\in [n]} B_i$, and take $j\in [n]\setminus I$. Then $\bigcap_{i\in [n]\setminus\{ j \}} B_i=\bigcap_{i\in [n]} B_i$ and hence
$(X\setminus B_j)\cap\bigcap_{i\in [n]\setminus\{ j \}}  B_i = \emptyset$,
contradicting $\IND(\mathcal B)=n$.
\end{proof}

The previous two lemmas in combination with Lemma~\ref{lem:SS} immediately yield the following useful fact (cf.~\cite[Chapter~5, Lemma~2.6]{vdDries-Tame}):

\begin{corollary}\label{cor:vdDries}
Suppose $\mathcal B$ has breadth $d$, let $N>0$, and let $\mathcal S$ be a set system on $X$ with the property that each set in $\mathcal S$ is a Boolean combination of at most $N$ sets in $\mathcal B$. Then 
$$\pi^*_{\mathcal S}(t) \leq \sum_{i=0}^{d} {Nt\choose i}\qquad\text{for every $t$.}$$ 
In particular, $\pi^*_{\mathcal S}(t)=O(t^{d})$ and hence $\vc^*(\mathcal S)\leq d$. 
\end{corollary}

\begin{example}\label{ex:convex}
Let $<$ be a linear ordering on $X$. We first recall some terminology: A subset $S$ of $X$ is said to be {\em convex}\/ (with respect to $<$) if for all $s,s'\in S$ and $x\in X$ the implication $s<x<s'\Rightarrow x\in S$ holds. So $\emptyset$ and singleton subsets are convex, as are intervals in $X$. Here and in the rest of the paper,
an {\it interval}\/ in $X$ is a subset of the form $$(a,b):=\{x\in X:a<x<b\}$$ where $a$, $b$ are elements of  $X\cup\{\pm\infty\}$ with $a<b$. 
Other examples of a convex subset of $X$ are its initial segments: a subset $S$ of $X$ is an {\em initial segment}\/ of $X$ if for all $s\in S$ and $x\in X$, the implication $x<s\Rightarrow x\in S$ holds. 
Now let $\mathcal S$ be the family of unions of at most $N$ convex subsets of $X$, for some given $N\in\bN$, and let 
$\mathcal B$ be the collection of all initial segments of $X$. Then $\mathcal B$ has breadth $1$, and every set in $\mathcal S$ is a Boolean combination of at most $2N$ sets in $\mathcal B$. Thus $\pi^*_{\mathcal S}(t)=O(t)$ by Corollary~\ref{cor:vdDries}.
\end{example}

\begin{example}\label{ex:balls}
Let $K$ be a field and $v\colon K\to \Gamma_\infty:=\Gamma\cup\{\infty\}$ be a valuation on $K$. By an \emph{open ball} in $K$ we mean any subset of $K$ of the form $\{x\in K:v(x-a)>\gamma\}$ where $a\in K$, $\gamma\in\Gamma_\infty$; similarly 
a set of the form $\{x\in K:v(x-a)\geq\gamma\}$ is called a \emph{closed ball} in $K$. A \emph{ball} in $K$ is an open or a closed ball in $K$.
Any two given balls in $K$ are either disjoint, or one contains the other.
Hence the collection $\mathcal B$ of balls in a given valued field has breadth $1$. Thus if $\mathcal S$ is the family of all Boolean combinations of at most $N$ balls in $K$, for some $N\in\bN$, then $\pi^*_{\mathcal S}(t)=O(t)$.
\end{example}

The preceding examples can be subsumed  under the following general example (inspired by~\cite{Adler-VCmin}):

\begin{example}\label{ex:VC-codim}
A family $\mathcal B$ of subsets of $X$ is said to be \emph{directed}\/ if $\mathcal B$ has breadth~$1$; i.e., for all $B,B'\in\mathcal B$ with $B\cap B'\neq\emptyset$ one has $B\subseteq B'$ or $B'\subseteq B$. If  $\mathcal B\subseteq 2^X$ is directed and $\mathcal S$ is the family of Boolean combinations of at most $N$ sets in $\mathcal B$, for some
$N\in\bN$, then $\pi^*_{\mathcal S}(t)=O(t)$.
\end{example}

We also note:

\begin{example}\label{ex:cosets}
Let $G$ be a group and let $\mathcal H$ be a collection of subgroups of $G$ with breadth~$d$.  Let $\mathcal B=\{gH:g\in G,H\in\mathcal H\}$ be the set of all (left) cosets of subgroups from $\mathcal H$. Then $\mathcal B$ also has breadth $d$. 
This follows from the general fact that if $H_1,\dots,H_n$ are subgroups of $G$, $g_1,\dots,g_n\in G$,  then  the intersection $\bigcap_{i\in [n]} g_iH_i$ is either empty or a coset of $\bigcap_{i\in [n]} H_i$. (So  if $\mathcal S$ is a family of Boolean combinations of at most $N$ elements of $\mathcal B$, for some $N\in\bN$, then $\pi^*_{\mathcal S}(t)=O(t^d)$.)
\end{example}

In connection with the previous example it is worth recording:

\begin{lemma}[Poizat]
Let $G$ be a group and let $\mathcal H$ be a collection of subgroups of $G$. Then $\breadth(\mathcal H)=\IND(\mathcal H)$.
\end{lemma}
\begin{proof}
By Lemma~\ref{lem:breadth and IND} we already know that $\breadth(\mathcal H)\geq\IND(\mathcal H)$. Suppose this inequality is strict. Then there are $H_1,\dots,H_{n+1}\in\mathcal H$, where $n=\IND(\mathcal H)$, such that $\bigcap_{i\in [n+1]}\neq\bigcap_{i\in [n+1]\setminus\{j\}} H_i$ for each $j\in [n+1]$. So for each $j\in [n+1]$ we may take $g_j\in \left(\bigcap_{i\in [n+1]\setminus\{j\}} H_i\right)\setminus H_j$. Then for every subset $I$ of $[n+1]$ the element $g_I:=\prod_{i\in I} g_i$ (with $g_\emptyset=1$) is in $\bigcap_{i\in [n+1]\setminus I} H_i\cap\bigcap_{i\in I} (G\setminus H_i)$. This contradicts $\IND(\mathcal H)=n$.
\end{proof}

\begin{example*}
Let $\mathcal S$ be the collection of all subgroups of $(\bZ,{+})$. Then $\mathcal S$ has infinite breadth, hence infinite independence dimension by the previous lemma, and thus is not a VC~class by Lemma~\ref{lem:dual shatter, 2}. In particular, the collection of arithmetic progressions $a+b\bZ$ ($a,b\in\bZ$) in $\bZ$ is also not a VC~class.
\end{example*}

If our family $\mathcal B$ has finite breadth $d$, then the  Helly number of $\mathcal B$ is at most $d$.
The \emph{Helly number}\/ of $\mathcal B$ is defined as the smallest $d>0$ such that
every finite subfamily $\{B_1,\dots,B_n\}$ of $\mathcal B$ with $n>d$ which is $d$-consistent, is consistent, that is to say: if for every $I\in {[n]\choose d}$ we have $\bigcap_{i\in I} B_i\neq\emptyset$, then $\bigcap_{i\in [n]} B_i\neq \emptyset$. Note however that conversely, the breadth may be infinite yet the Helly number finite, even in the case of cosets: the collection of arithmetic progressions in $\bZ$ is independent, but has Helly number~$2$. Also, not every VC~class has finite Helly number: the family whose members are the subsets of $\bR$ with two connected components, though a VC~class (of VC~dimension $4$), has infinite Helly number.
(For each $n$  the elements $[0,i)\cup (i+1,n]$, $i=0,\dots,n-1$ of this family form an $n-1$-consistent subfamily which is inconsistent.)

\medskip

The following example is a prototype for finite-breadth families when we have a dimension function at our disposal:

\begin{example}\label{ex:dimension}
Define the \emph{height} of $\mathcal B$ to be the largest $d$ (if it exists) such that there are $B_1,\dots,B_d\in\mathcal B$ with
$$B_1 \supsetneq B_1\cap B_2 \supsetneq \cdots \supsetneq B_1\cap\cdots\cap B_d\neq\emptyset.$$
So $\mathcal B$ has height $0$ iff $\mathcal B$ does not contain a non-empty set, and
$\mathcal B$ has height $1$ iff $\mathcal B$ does contain a non-empty set, but any two distinct elements of $\mathcal B$ are disjoint.
Clearly if $\mathcal B$ has height $d>0$, then the breadth of $\mathcal B$ is at most~$d$. If $\mathcal B$ has height $d>1$ and in addition $\mathcal B$ has a largest element (with respect to inclusion) then the breadth of $\mathcal B$ is smaller than~$d$: to see this
let $B_1,\dots,B_{d}\in\mathcal B$ with $\bigcap_{i\in[d]} B_i\neq\emptyset$ be given; if $B_1$ is the largest  element $B$ of $\mathcal B$ then clearly $\bigcap_{i\in[d]} B_i=\bigcap_{i\in[d]\setminus\{1\}} B_i$, and otherwise we have a chain
$$B \supsetneq B_1 \supseteq B_1\cap B_2 \supseteq \cdots \supseteq B_1\cap\cdots\cap B_d\neq\emptyset,$$
hence
$\bigcap_{i\in [j]} B_i=\bigcap_{i\in [j+1]} B_i$ and so $\bigcap_{i\in [d]} B_i = \bigcap_{i\in [d]\setminus\{j+1\}} B_i$, for some $j\in [d-1]$.

\end{example}

The following observation (the proof of which we leave to the reader) allows us to produce new finite-breadth set systems from old ones:

\begin{lemma}
Let $\mathcal B$, $\mathcal B'$ be set systems on $X$ and $X'$, respectively, and consider the set system
$$\mathcal B\boxtimes\mathcal B' := \{B\times B': B\in\mathcal B,\ B'\in\mathcal B'\}$$
on $X\times X'$. Then
$$\breadth(\mathcal B\boxtimes\mathcal B') \leq \breadth(\mathcal B) + \breadth(\mathcal B'),$$
and this inequality is an equality if both $\mathcal B$ and $\mathcal B'$ have breadth larger than~$1$ and contain a largest element  \textup{(}with respect to inclusion\textup{)}.
\end{lemma}

This lemma  immediately yields:

\begin{corollary}
Let $\mathcal B$, $\mathcal B'$ be set systems on $X$. Then the set system
$$\mathcal B\sqcap\mathcal B' := \{B\cap B': B\in\mathcal B,\ B'\in\mathcal B'\}$$
on $X$ has breadth at most $\breadth(\mathcal B) + \breadth(\mathcal B')$.
\end{corollary}

\begin{example*}
Suppose $<$ is a linear ordering of $X$ and $\mathcal B$ is the collection of convex subsets of $X$. Every element of $\mathcal B$ can be expressed as an intersection of an initial segment of $(X,{<})$ with a \emph{final segment} of $(X,{<})$ (i.e., an initial segment of the linearly ordered set $(X,{>})$). Hence $\breadth(\mathcal B)=2$.
\end{example*}

If $\mathcal B$ is a sublattice of $(2^X,{\cap},{\cup})$ which does not contain $\emptyset$ and $X$, then $\mathcal B$ and the set system $\neg\mathcal B:=\{X\setminus B:B\in\mathcal B\}$ have the same breadth; this is an immediate consequence of the following lemma:

\begin{lemma}\label{lem:breadth duality}
Suppose $\mathcal B$ is closed under \textup{(}finite\textup{)} intersections and unions, and $\mathcal B$ does not contain the empty set.
Then for each $d$ the following are equivalent:
\begin{enumerate}
\item For all $B_1,\dots,B_{d+1}\in\mathcal B$ there is some $i\in [d+1]$ such that $\bigcap_{j\neq i} B_j\subseteq B_i$;
\item for all $B_1,\dots,B_{d+1}\in\mathcal B$ there is some $i\in [d+1]$ such that $B_i\subseteq \bigcup_{j\neq i} B_j$.
\end{enumerate}  
\end{lemma}
\begin{proof}
To see (1)~$\Rightarrow$~(2)  apply (1) to $B_i'=\bigcup_{j\neq i} B_j$ ($i\in [d+1]$) in place of the $B_i$, and for the converse implication apply (2) to $B_i''=\bigcap_{j\neq i} B_j$  ($i\in [d+1]$).
\end{proof}

We finish our discussion of breadth by a surprising connection between breadth and stability. We will not use this observation later in the paper, but we include it here since it shows, under the assumption of stability, the ubiquity of set systems of \emph{infinite} breadth. The breadth of a relation between two sets is by definition the breadth of 
the associated set system, cf.~Section~\ref{sec:VC duality}.

\begin{proposition}\label{prop:BB}
Let  $X$, $Y$ be infinite sets and $\Phi\subseteq X\times Y$ be a relation. If $\vc(\Phi)>0$ 
then $\Phi$ is unstable, or at least one of $\Phi$ or $\neg\Phi$ has infinite breadth.
\end{proposition}

At the root of Proposition~\ref{prop:BB} is a theorem of Balogh and Bollob\'as \cite{BB}, which we explain first.
For this we need some additional terminology:
Let $(X,\mathcal S)$ and $(X',\mathcal S')$ be set systems. We say that \emph{$(X,\mathcal S)$ contains $(X',\mathcal S')$} as a trace if there exists an injective map $f\colon X'\to X$ such that $f(\mathcal S')\subseteq \mathcal S\cap f(X')$.
For example, if $(X,\mathcal S)$ is a set system and $A\subseteq X$ then $(X,\mathcal S)$ trivially contains $(A,\mathcal S\cap A)$. Also, if $(X,\mathcal S)$ contains $(X',\mathcal S')$, and $(X',\mathcal S')$ contains $(X'',\mathcal S'')$, then $(X,\mathcal S)$ contains $(X'',\mathcal S'')$.

For $k\geq 2$ consider now the following set systems on $[k]$:
\begin{align*}
\mathcal C_k		&= \big\{ [i] : i\in [k] \big\}				&& \text{(the \emph{$k$-chain})} \\
\mathcal S_k		&= \big\{ \{i\} : i\in [k] \big\}				&& \text{(the \emph{$k$-star})} \\
\mathcal T_k		&= \big\{ [k]\setminus\{i\} : i\in [k] \big\}		&& \text{(the \emph{$k$-costar})}.
\end{align*}
Balogh and Bollob\'as \cite[Theorem~1]{BB} showed that these set systems are unavoidable among sufficiently large set systems. More precisely: {\it
for all integers $k,l,m\geq 2$ there is some $N=N(k,l,m)$ such that every set system $\mathcal S$ on a finite base set with $\abs{\mathcal S}\geq N$ contains the $k$-chain, the $l$-star, or the $m$-costar.} (Note that there is no condition on the size of the base set in this statement.)

\begin{proof}[Proof of Proposition~\ref{prop:BB}]
Let $\mathcal S=\mathcal S_\Phi$. We first observe, for $k\geq 2$:
\begin{enumerate}
\item $\mathcal S$ contains $\mathcal C_k$  iff there is a $k$-ladder for $\Phi$;
\item if $\breadth(\Phi^*)\geq k$ then $\mathcal S$ contains $\mathcal T_k$; and
\item if $\mathcal S$ contains $\mathcal T_{k+1}$ then $\breadth(\Phi^*)\geq k$.
\end{enumerate}
Part (1) is obvious. For (2) note that $\breadth(\Phi^*)\geq k$ iff there exist elements $x_1,\dots,x_{k}$ of $X$ such that $\bigcap_{j\in [k]}  \Phi^*_{x_j} \neq \emptyset$ and for each $i\in [k]$,
$$(Y\setminus\Phi^*_{x_i})\cap \bigcap_{j\in [k]\setminus \{i\}} \Phi^*_{x_j} \neq\emptyset,$$
and for such choice of $x_i$, setting $X'=\{x_1,\dots,x_k\}$ we have $X'\setminus\{x_i\} \in \mathcal S\cap X'$ for each~$i$. 
Similarly, for (3), if $X'=\{x_1,\dots,x_{k+1}\}\in {X\choose k+1}$ such that $X'\setminus\{x_i\}\in\mathcal S\cap X'$ for each $i\in [k+1]$, then for each such $i$ we have
$$(Y\setminus\Phi^*_{x_i})\cap \bigcap_{j\in [k+1]\setminus \{i\}} \Phi^*_{x_j} \neq\emptyset;$$
in particular, taking $i=k+1$ we see that $\bigcap_{j\in [k]}\Phi^*_{x_j} \neq\emptyset$, and
for each $i\in [k]$ we have $(Y\setminus\Phi^*_{x_i})\cap\bigcap_{j\in [k]\setminus\{i\}}\Phi^*_{x_j} \neq\emptyset$, hence $\breadth(\Phi^*)\geq k$.
Also note that (2) and (3) are true with $\mathcal T_k$, $\mathcal T_{k+1}$ and $\Phi^*$ replaced by $\mathcal S_k$, $\mathcal S_{k+1}$ and~$\neg\Phi^*$, respectively.

Suppose now that $\vc(\Phi)>0$, i.e., $\mathcal S$ is infinite. Then $\mathcal S^*=\mathcal S_{\Phi^*}$ is also infinite (see Example~\ref{ex:dual shatter, 2}). Then we have $\vc(\mathcal S^*)\geq 1$, hence there are arbitrarily large $n$ and $B\in {Y\choose n}$ such that $\abs{\mathcal S^*\cap B} \geq n^{1/2}$. In particular, for each $N$ there is a finite subset $B_N$ of $Y$ with $\abs{\mathcal S^*\cap B_N}\geq N$. Now suppose $\Phi$ is stable; then $\Phi^*$ is also stable. Let $k_0\geq 2$ be larger than the ladder dimension of $\Phi^*$. Then if $k\geq 2$ and $N\geq N(k_0,k,k)$ then $\mathcal S^*\cap B_N$ (and hence $\mathcal S^*$) contains the $k$-star or the $k$-costar. Thus by observation (3) above, at least one of $\Phi$ or $\neg\Phi$ has infinite breadth.
\end{proof}

Of course, the converse of the implication in this proposition also holds:
if $\vc(\Phi)=0$ then $\mathcal S_\Phi$ is finite, hence trivially $\Phi$ is stable, and both $\Phi$ and $\neg\Phi$  have finite breadth, since $\mathcal S_{\neg\Phi}$ is finite as well.

\begin{example*} 
Let $<$ be a linear ordering of $X$. Suppose $\mathcal B$ is the collection of initial segments of $(X,{<})$, as in Example~\ref{ex:convex}. Then $\neg\mathcal B=\{X\setminus B:B\in\mathcal B\}$ consists of final segments of $(X,{<})$. Hence $\mathcal B$ and $\neg\mathcal B$ both have finite breadth (indeed, breadth~$1$). Proposition~\ref{prop:BB} shows that phenomena such as these are confined to unstable contexts  (for infinite set systems).
\end{example*}

Using Lemma~\ref{lem:breadth duality}, Proposition~\ref{prop:BB} also implies:

\begin{corollary}\label{cor:BB}
Suppose $X$ and $Y$ are infinite sets and $\Phi\subseteq X\times Y$ such that $\mathcal S_\Phi$ is an infinite sublattice of $(2^X,{\cap},{\cup})$ of finite breadth, with $\emptyset, X \notin\mathcal S_\Phi$. Then $\Phi$ is unstable.
\end{corollary}

\section{The Model-Theoretic Context}\label{sec:model-theoretic context}

\noindent
Throughout this section we  fix a first-order language $\mathcal L$, and we let $\varphi(x;y)$ be  a partitioned $\mathcal L$-formula (as defined in the introduction), with object variables $x=(x_1,\dots,x_m)$ and parameter variables $y=(y_1,\dots,y_n)$. The formula $\varphi$ gives rise, in a given $\mathcal L$-structure, to a set system. The associated parameters introduced in the previous section (shatter function, VC density etc.) are elementary invariants of the structure in question. In Section~\ref{sec:VC density of a theory} below we also introduce the VC~density function $\vc^T$ of a complete first-order $\mathcal L$-theory~$T$ with no finite models: if $\vc^T(n)$ is finite then  $\vc^T(n)$ is a uniform bound on the VC~density of all partitioned formulas in $T$ having $n$ parameter variables. In Section~\ref{sec:computing vc(1)}  we illustrate this concept by computing $\vc^T(1)$ for various~$T$. In Section~\ref{sec:sets of formulas} we then extend the definition of dual VC~density to finite sets of formulas; this is convenient for later sections, but, as we see in Section~\ref{sec:coding}, does not add much extra generality. There is some indication that computing VC~density is easier when only parameters coming from initial segments of indiscernible sequences are considered;
although we will not pursue these issues in the rest of the paper, we think that
the relationship between  quantities like $\VC$ or $\vc$ and their ``indiscernible'' counterparts deserves further investigation; we explore some connections in the last subsection.

\subsection{VC density of definable families}
Given an $\mathcal L$-structure $\mathbf M$ and a tuple $b\in M^n$, we denote the subset of $M^m$ defined by the $\mathcal L$-formula $\varphi(x;b)$ with parameters $b$ in $\mathbf M$ by
$$\varphi^{\mathbf M}(M^m;b) := \{ a\in M^m: \mathbf M\models\varphi(a;b) \}.$$
A subset of $M^m$ is called \emph{definable} (in $\mathbf M$) if it is of the form $\varphi^{\mathbf M}(M^m;b)$, for some $\ph$ and $b$.
We also denote by
$$\mathcal S^{\mathbf M}_{\varphi}:=\{ \varphi^{\mathbf M}(M^m;b) : b\in M^n \}$$
the family of subsets of $M^m$ defined by $\varphi$ in $\mathbf M$,
and we call $(M^m,\mathcal S^{\mathbf M}_{\varphi})$
the \emph{set system associated with $\varphi$ in $\mathbf M$}.
More generally, to a given collection $\Phi(x)=\{\varphi_i(x;y_i)\}_{i\in I}$ of partitioned $\mathcal L$-formulas in the tuple of object variables $x$ (and in various tuples of parameter variables $y_i$) we may associate the set system
$$\mathcal S_{\Phi}^{\mathbf M} := \{ \varphi_i^{\mathbf M}(M^m;b) : i\in I,\ b\in M^{\abs{y_i}} \}$$
on $M^m$ defined by the instances of the formulas $\varphi_i$.
If the $\mathcal L$-structure $\mathbf M$ is understood from the context, we drop the superscript $\mathbf M$ in our notation. 

Suppose now $\mathbf M$ is an infinite $\mathcal L$-structure. 
As usual, say that $\varphi$ is invariant under an extension $\mathbf M\subseteq \mathbf N$ of $\mathcal L$-structures if $\mathbf M\models\varphi(a;b)\Longleftrightarrow \mathbf N\models\varphi(a;b)$ for all $a\in M^m$ and $b\in M^n$.
The following is obvious:

\begin{lemma}\label{lem:invariant under extensions}
Suppose $\mathbf N$ is an $\mathcal L$-structure with
$\mathbf M\subseteq \mathbf N$ and $\varphi$ is invariant under $\mathbf M\subseteq \mathbf N$. Then $\mathcal S^{\mathbf M}_\varphi \subseteq M^m\cap \mathcal S^{\mathbf N}_{\varphi}$, hence
$\pi^{\mathbf M}_\varphi\leq \pi^{\mathbf N}_\varphi$
and therefore
$\VC(\mathcal S^{\mathbf M}_\varphi)\leq \VC(\mathcal S^{\mathbf N}_\varphi)$ and
$\vc(\mathcal S^{\mathbf M}_\varphi)\leq \vc(\mathcal S^{\mathbf N}_\varphi)$.
\end{lemma}

For each $s,t\in\bN$, consider the 
$\mathcal L$-sentence 
\begin{multline*}
\pi^{s,t}_\varphi :=
\forall x^{(1)}\cdots\forall x^{(t)} \forall y^{(1)} \dots \forall y^{(s+1)} \\
\left(
\bigvee_{1\leq i<j\leq t} \bigwedge_{k=1}^m x^{(i)}_k=x^{(j)}_k \vee \bigvee_{1\leq k < l\leq s+1}
\bigwedge_{1\leq i\leq t} \varphi(x^{(i)};y^{(k)}) \leftrightarrow \varphi(x^{(i)};y^{(l)})\right),
\end{multline*}
where $x^{(i)}=(x^{(i)}_1,\dots,x^{(i)}_m)$ and $y^{(j)}=(y^{(j)}_1,\dots,y^{(j)}_n)$ are tuples of new variables.
Then, with $\pi^{\mathbf M}_\varphi:=\pi_{\mathcal S^{\mathbf M}_{\varphi}}$ denoting the shatter function of $\mathcal S^{\mathbf M}_{\varphi}$, we obviously have:

\begin{lemma}\label{lem:shatter function and elementary equivalence}
For each $s,t\in\bN$,
$$\mathbf M\models\pi^{s,t}_\varphi \qquad\Longleftrightarrow\qquad 
\pi^{\mathbf M}_\varphi(t) \leq s.$$
In particular, if $\mathbf N$ is an $\mathcal L$-structure with $\mathbf M\equiv \mathbf N$, then $\pi^{\mathbf M}_\varphi = \pi^{\mathbf N}_\varphi$.
\end{lemma}

From now until the end of this section we fix a complete $\mathcal L$-theory $T$ with only infinite models, and let $\mathbf M$ range over models of $T$. By the previous lemma we may set 
$$\pi_\varphi:=\pi^{\mathbf M}_\varphi,\quad 
\VC(\varphi):=\VC(\mathcal S^{\mathbf M}_{\varphi}), \quad
\vc(\varphi):=\vc(\mathcal S^{\mathbf M}_{\varphi}),$$
where $\mathbf M$ is an arbitrarily chosen model of $T$. 
We call $\pi_\varphi$ the \emph{shatter function of $\varphi$} (in~$T$), and we call $\VC(\varphi)$ and $\vc(\varphi)$ the \emph{VC dimension of $\varphi$} and \emph{VC density of~$\varphi$} (in~$T$), respectively.
If we want to stress the dependence of $\pi_\varphi$ on $T$ we write $\pi_\varphi^T$, and similarly for $\VC$ and $\vc$.

\medskip

Note that the definition of $\pi_\varphi$ only depends on the set system $\mathcal S_\varphi$ and not on the particular representing formula $\varphi$. In particular, $\pi_\varphi$ remains unchanged under $\emptyset$-definable reparameterizations:

\begin{lemma}\label{lem:reparametrization}
Let $\gamma(z;y)$ be an $\mathcal L$-formula, where $z=(z_1,\dots,z_l)$, which defines the graph of a map $g\colon M^l\to M^n$. Let $\sigma(x;z):=\exists y (\gamma(z;y)\wedge\varphi(x;y) )$, so 
$$\mathcal S_{\sigma}=\big\{\varphi^{\mathbf M}(M^m;g(c)):c\in M^l\big\}.$$ 
Then $\pi_\sigma\leq\pi_\varphi$, with equality if $g$ is surjective.
\end{lemma}

The \emph{dual} of the partitioned $\mathcal L$-formula $\varphi(x;y)$ is
$\varphi^*(y;x):=\varphi(x;y)$; that is, $\varphi^*(y;x)$ is syntactically the same $\mathcal L$-formula $\varphi$, only with the role of the object and parameter variables interchanged. We call $\VC^*(\varphi):=\VC(\varphi^*)$ and $\vc^*(\varphi):=\vc(\varphi^*)$ the \emph{dual VC dimension} and \emph{dual VC density} of $\varphi$, respectively.
By Lemma~\ref{lem:dual shatter} we have $\pi^*_{\varphi}=\pi_{\varphi^*}$ and hence $\VC^*(\varphi)=\IND(\varphi)$ and
$$\vc^*(\varphi) = \inf \big\{r\in \bR^{>0} : \pi^*_\varphi(n)=O(n^r) \big\}.$$
If any of the quantities $\VC(\varphi)$, $\vc(\varphi)$, $\VC^*(\varphi)$, $\vc^*(\varphi)$ is finite, then so are all the others, and in this case we say that $\varphi$ is \emph{dependent}\/ or that $\varphi$ \emph{defines a VC class.}
Note that for every partitioned $\mathcal L$-formula  $\varphi(x;y)$ we have $\vc(\varphi)\geq 0$, with equality if $\mathcal S_\varphi$ is finite. If~$\mathcal S_\varphi$ is infinite then $\vc(\varphi)\geq 1$.
(See the remarks following Lemma~\ref{lem:inverse image}.)

\medskip

Letting $\Phi := \varphi(M^m;M^n)$ and $X:=M^m$, $Y:=M^n$,  in the notation introduced in the previous subsection we have
$\mathcal S_\Phi =  \mathcal S_\varphi$ and $\mathcal S_{\Phi^*} = \mathcal S_{\varphi^*}$. Hence
Lemma~\ref{lem:vc for boolean combinations} yields:

\begin{corollary}
We have $\vc(\neg\varphi)=\vc(\varphi)$, and if
$\psi(x;z)$ is another partitioned $\mathcal L$-formula, then
$\vc(\varphi\wedge\psi), \vc(\varphi\vee\psi)\leq \vc(\varphi)+\vc(\psi)$.
\end{corollary}

From Lemma~\ref{lem:inverse image} one also obtains the invariance of $\vc$ under inverse images of surjective $\emptyset$-definable maps:

\begin{corollary}\label{cor:inverse image}
Let $\delta(v;x)$ be an $\mathcal L$-formula, where $v=(v_1,\dots,v_k)$, which defines the graph of
a map $f\colon M^k\to M^m$, and let $\rho(v;y):=\exists x (\delta\wedge\varphi)$, so $\mathcal S_{\rho}=f^{-1}(\mathcal S_{\varphi})$. Then $\pi_{\rho}\leq\pi_{\varphi}$, with equality if $f$ is surjective.
\end{corollary}

The theory $T$ is NIP iff every partitioned $\mathcal L$-formula defines a VC class. 
The theorem of Shelah \cite{S1} already mentioned in the introduction shows that in order for every partitioned $\mathcal L$-formula $\varphi(x;y)$ to define a VC~class, it is enough that this holds for all such $\varphi(x;y)$ with a single parameter variable (i.e., $\abs{y}=1$).
Hence if for each partitioned $\mathcal L$-formula $\varphi(x;y)$ with $\abs{x}=1$ the set system $\mathcal S_\varphi$ has finite breadth then $T$ is NIP, by Lemma~\ref{lem:breadth and IND}.
The theory $T$ is said to be stable if for every partitioned $\mathcal L$-formula $\varphi(x;y)$ the associated relation $\Phi=\varphi(M^m;M^n)$ is stable (in the sense of Section~\ref{sec:VC duality}); if $T$ is stable then for each $\varphi(x;y)$ with $\mathcal S_\varphi$ infinite, at least one of $\mathcal S_\varphi$ or $\mathcal S_{\neg\varphi}$ has infinite breadth, by Proposition~\ref{prop:BB}.
(Corollary~\ref{cor:BB} of the same proposition also yields that if $T$ is stable then all finite-breadth sublattices $\mathcal S$ of the lattice of all subsets of $M^m$ which have the form $\mathcal S=\mathcal S_\varphi$ for some
$\mathcal L$-formula $\varphi(x;y)$ with $\abs{x}=m$ are finite.)

\subsection{VC density of a theory}\label{sec:VC density of a theory}
We define the \emph{VC density of $T$} to be the function 
$$\vc=\vc^T\colon\bN\to\bR^{\geq 0}\cup\{\infty\}$$ 
given by
$$\vc(n):=\sup\big\{ \vc(\varphi): \text{$\varphi(x;y)$ is an
$\mathcal L$-formula with $\abs{y}=n$} \big\}.$$
Note that we could have also defined $\vc^T$ as
$$\vc(m)=\sup\big\{ \vc^*(\varphi): \text{$\varphi(x;y)$ is an
$\mathcal L$-formula with $\abs{x}=m$} \big\}.$$
In the introduction we already observed that $\vc(m)\geq m$ for every $m$.
If $\mathcal L'$ is an expansion of $\mathcal L$ and $T'\supseteq T$ a complete $\mathcal L'$-theory, then $\vc^{T}\leq\vc^{T'}$, with equality if $T'$ is an expansion of $T$ by definitions. Moreover, $\vc$ does not change under expansions by constants:

\begin{lemma}\label{lem:expansions by constants}
Let $\mathcal L'=\mathcal L\cup\{c_i:i\in I\}$ where the $c_i$ are new constant symbols, and let $T'\supseteq T$ be a complete $\mathcal L'$-theory. Then $\vc^T=\vc^{T'}$.
\end{lemma}
\begin{proof}
Let $\mathbf M'\models T'$ and $C:=\{c_i^{\mathbf M'}:i\in I\}\subseteq M'$. 
Let $\varphi(x;y,z)$ be an $\mathcal L$-formula with $\abs{x}=m$, and let $c\in C^{\abs{z}}$. Then $\pi^*_{\varphi(x;y,c)}(t)\leq\pi^*_{\varphi(x;y,z)}(t)$ for every $t$, hence $\vc^*(\varphi(x;y,c))\leq\vc^*(\varphi(x;y,z))\leq\vc^{T}(m)$ and thus $\vc^{T'}(m)\leq\vc^T(m)$.
\end{proof}

It is clear that $\vc(n)\leq\vc(n+1)$ for every $n$, by viewing a formula with $n$ parameter variables as one with $n+1$ parameters; perhaps less obviously:

\begin{lemma}\label{lem:lower estimate, 1}
$\vc(n)+1\leq\vc(n+1)$ for every $n$.
\end{lemma}
\begin{proof}
By the preceding lemma  we may assume that $\mathcal L$ contains a constant symbol~$0$.
Let $\varphi(x;y)$ be a partitioned $\mathcal L$-formula with $\abs{x}=m$, $\abs{y}=n$. We construct a formula $\psi(x,x_{m+1};y,y_{n+1})$ with $\pi_\varphi(t)\cdot t\leq\pi_\psi(2t)$ for every $t$ (hence $\vc(\varphi)+1\leq\vc(\psi)$), which then shows the lemma. We set
$$\psi:=(x_{m+1}=0\wedge\varphi(x;y)) \vee (x_{m+1}=y_{n+1}).$$
Then for $b\in M^n$, $c\in M$ we have
$$  \psi(M^{m+1};b,c) = (\varphi(M^m;b)\times \{0\})\cup (M^m\times\{c\}).$$
Let $A\subseteq M^m$ with $\abs{A}=t$ and $\pi_\varphi(t)=\abs{A\cap\mathcal S_\varphi}$. Choose 
pairwise distinct elements $a_1,\dots,a_t\in M\setminus\{0\}$ and an arbitrary element $a'\in M^m$, and set
$$A':=\big(A\times\{0\}\big)\cup \big\{(a',a_1),\dots,(a',a_t)\big\}.$$
Then $\abs{A'}=2t$, and for $b\in M^n$ and $j=1,\dots,t$ we have
$$A'\cap\psi(M^{m+1};b,a_{j}) = \big((A\cap\varphi(M^m;b))\times\{0\}\big) \cup \{(a',a_j)\}.$$
Take $b_1,\dots,b_k\in M^n$, $k=\pi_\varphi(t)$, such that the sets $A\cap\varphi(M^{m};b_i)$, $i=1,\dots,k$, are pairwise distinct. Then the sets
$A'\cap\psi(M^{m+1};b_i,a_j)$ (where $i=1,\dots,k$, $j=1,\dots,t$)
are also pairwise distinct. Hence $\pi_\psi(2t)\geq\abs{A'\cap\mathcal S_\psi}\geq k\cdot t=\pi_\varphi(t)\cdot t$ as claimed.
\end{proof}

In this paper we prove, for many (unstable) NIP theories $T$ of interest, that $\vc^T(m)<\infty$ for every $m$, and in fact, in these cases we establish that $\vc^T(m)$ is bounded by a linear function of $m$.
Note, however, that $T$ NIP does not imply that $\vc^T(m)<\infty$ for all $m$: it is easy to see that for every $T$ (whether NIP or not) we have $\vc^{T^\eq}(1)=\infty$, whereas $T$ is NIP iff $T^\eq$ is NIP. (We thank Martin Ziegler for pointing this out.)

\medskip

By Laskowski's proof \cite{l} of Shelah's theorem \cite{S1}, the VC~dimension $\VC(\varphi)$ of an $\mathcal L$-formula $\varphi(x;y)$ is bounded in terms of the VC~dimensions of certain $\mathcal L$-formulas with a single parameter variable (which, however, are astronomical, involving iterated Ramsey numbers).
This together with the examples below raises the following question, the answer to which we don't know:

\begin{question}\mbox{}
If $\vc^T(1)<\infty$, is $\vc^T(m)<\infty$ for every $m$? 
\end{question}

Provided the answer to this question is positive, one may then also ask how $\vc(m)$ depends on $m$ and $\vc(1)$; e.g.: {\it is there a function $\beta\colon \bN\times\bR^{\geq 0}\to\bR^{\geq 0}$, independent of $T$, with the property that if $\vc^T(1)<\infty$, then $\vc^T(m)\leq \beta\big(m,\vc^T(1)\big)$ for every~$m$?}
(In all examples which we considered where $\vc^T$ is known to be real-valued, $\vc^T$ grows at worst linearly.) 

\subsection{Computing $\vc^T(1)$}\label{sec:computing vc(1)}
In concrete cases it is often easy to see that $\vc^T(1)=1$:

\begin{example}\label{ex:strongly minimal}
Suppose that $\mathbf M$ is strongly minimal. The collection $\mathcal B={M\choose 1}$ of one-element subsets of $M$ has breadth $1$; so $\vc^T(1)=1$. (Corollary~\ref{cor:vdDries}.)
\end{example}

\begin{example}\label{ex:weakly o-minimal}
Suppose that $\mathcal L$ contains a binary relation symbol ``$<$'', 
$\mathbf M=(M,{<},\dots)$ is an expansion of a linearly ordered set $(M,{<})$, and $T=\Th(\mathbf M)$ is weakly o-minimal. Then for every partitioned $\mathcal L$-formula $\varphi(x;y)$ with $\abs{x}=1$ there exists an integer $N\geq 0$ such that for every $b\in M^m$, the set $\varphi^{\mathbf M}(M;b)$ is a finite union of at most $N$ convex subsets of $M$. Hence
$\vc^T(1)=1$ by Example~\ref{ex:convex}.
\end{example}

\begin{example}\label{ex:C-minimal}
Suppose that $\mathcal L_{\divi}$ is the expansion of the language $\{0,1,{+},{-},{\times}\}$ of rings by a binary relation symbol ``$|$''. In a field  $K$ equipped with a valuation
$v\colon K \to \Gamma \cup \{\infty\}$, we interpret $|$ by putting $a|b :\Longleftrightarrow v(a)\leq v(b)$, for all $a,b\in K$. 
Suppose $T$ is a complete theory of valued fields in an expansion of $\mathcal L_{\divi}$, and $T$ is $C$-minimal, i.e., 
for every $\mathbf K\models T$, every definable subset of $K$ is a finite Boolean combination of balls in $K$. Then
for every partitioned $\mathcal L$-formula $\varphi(x;y)$ with $\abs{x}=1$ there exists an integer $N\geq 0$ such that for every $b\in K^m$, the set $\varphi^{\mathbf K}(K;b)$ is a Boolean combination of at most $N$ balls in $K$.
Thus $\vc^T(1)=1$ by Example~\ref{ex:balls}.
\end{example}

The definition of $C$-minimality used in the previous example agrees (for expansions of valued fields) with the one in \cite{HK}; this definition is slightly more restrictive than the original one, introduced in \cite{hm2,MS}.
Every completion of the $\mathcal L_{\divi}$-theory ACVF of non-trivially valued algebraically closed fields is $C$-minimal (essentially by A.~Robinson's quantifier elimination in ACVF; see \cite{Holly}). Conversely, every valued field with $C$-minimal elementary theory is algebraically closed \cite{hm2}.  
Moreover, the rigid analytic expansions of ACVF introduced by Lipshitz \cite{Lipshitz} are $C$-minimal \cite{lr}.

\begin{example}\label{ex:modules, 1}
Let $R$ be a ring and suppose $\mathcal L=\mathcal L_R$ is the language of $R$-modules. (In this paper, ``$R$-module'' always means ``left $R$-module.'') Suppose $M$ is an $R$-module, construed as an $\mathcal L_R$-structure in the natural way. By the Baur-Monk~Theorem, every $\mathcal L_R$-formula is equivalent in $T=\Th(M)$ to a Boolean combination of positive primitive (p.p.) $\mathcal L_R$-formulas; given a p.p.~$\mathcal L_R$-formula $\varphi(x;y)$ and $b\in M^{\abs{y}}$, the set $\varphi(M^{\abs{x}};b)$ is a coset of $\varphi(M^{\abs{x}};0)$. 
Suppose $M$ is \emph{p.p.-uniserial,} i.e., the subgroups of $M$ definable by p.p.~$\mathcal L_R$-formulas form a chain. 
By Example~\ref{ex:cosets}, if $M$ is infinite, then we have $\vc^T(1)=1$.
(In \cite{ADHMS} 
this will be extended to $\vc^T(m)=m$ for every $m$.)
Examples for p.p.-uniserial abelian groups (viewed as $\bZ$-modules) include $\bQ^{(\alpha)}$, $\bZ_{(p)}^{(\alpha)}$, $\bZ(p^n)^{(\alpha)}$ and $\bZ(p^\infty)^{(\alpha)}$, where $p$ is a prime and $\alpha$ is a cardinal, possibly infinite. 
Here
$$\bZ_{(p)}=\big\{a/b:a,b\in\bZ,\ b\neq 0,\ p\nmid b\big\},$$ viewed as a subgroup of the additive group of $\bQ$,
$\bZ(p^n)$ denotes the cyclic group $\bZ/p^n\bZ$ of order $p^n$, and $\bZ(p^\infty)$ denotes the Pr\"ufer $p$-group (the group of $p^n$th roots of unity, for varying $n$, written additively). Given an $R$-module $M$ and an index set $I$, $M^{(I)}$ denotes, as usual, the $R$-submodule of the direct product $M^I$ consisting of all sequences with cofinitely many zero entries.
\end{example}

Examples~\ref{ex:strongly minimal}--\ref{ex:modules, 1} may be generalized as follows:

\begin{example}\label{ex:VC-minimal}
A family $\Phi(x)=\{\varphi_i(x;y_i)\}_{i\in I}$ of $\mathcal L$-formulas in the object variables $x$ (and in various tuples of parameter variables $y_i$) is said to have \emph{dual VC~dimension $d$}\/ if the set system $\mathcal S=\mathcal S_\Phi$ defined by the instances of the formulas $\varphi_i$ has dual VC~dimension~$d$. 
If $\Phi$ has dual VC~dimension at most~$1$, then we say that $\Phi$ is \emph{VC-minimal}; cf.~Example~\ref{ex:dual VCdim at most 1}. We also say that $\Phi$ is directed if $\mathcal S$ is directed in the sense of Example~\ref{ex:VC-codim}.

The $\mathcal L$-theory $T$ is  \emph{VC-minimal} if there is a VC-minimal family of $\mathcal L$-formulas~$\Phi(x)$ with $\abs{x}=1$ such that in every $\mathbf M\models T$ every definable (possibly with parameters) subset of $M$ is a Boolean combination of finitely many sets in~$\mathcal S_\Phi$. (This definition was introduced in \cite{Adler-VCmin}.) If $T$ is a VC-minimal $\mathcal L$-theory, then for every $\mathcal L$-formula $\varphi(x;y)$ with $\abs{x}=1$ there exists some $N\in\bN$ such that in every $\mathbf M\models T$ every instance $\varphi(x;b)$ ($b\in M^{\abs{y}}$) of $\varphi$ defines a subset of $M$ which is a Boolean combination of at most $N$ sets in $\mathcal S_\Phi$, by compactness.

One says that the VC-minimal theory $T$ is \emph{directed} if one can additionally choose $\Phi(x)$ to be directed; in that case we have $\vc^T(1)=1$ by Example~\ref{ex:VC-codim}.
By \cite[Proposition~6]{Adler-VCmin}, if $\Phi(x)$ is VC-minimal and $\mathcal S_\Phi$ contains some $\emptyset$-definable set other than $\emptyset$ or $M^{\abs{x}}$, then there is a directed set $\Psi(x)$ of $\mathcal L$-formulas such that $\mathcal S_\Phi=\mathcal S_\Psi$ and $\mathcal S_{\neg\Phi}=\mathcal S_{\neg\Psi}$.
By Lemma~\ref{lem:expansions by constants} this yields in fact
$\vc^T(1)=1$ for every complete VC-minimal~$T$ (directed or not) without finite models.
\end{example}

Example~\ref{ex:modules, 1} can also be generalized in a different direction:

\begin{example}\label{ex:one-based}
Suppose $\mathcal L$ is a language expanding the language $\{1,{\cdot}\}$ of groups, and  $T$ is a complete $\mathcal L$-theory containing the theory of infinite groups. 
Suppose for every $\mathbf G\models T$, every definable subset of $G$ is a Boolean combination of cosets of $\acl^\eq(\emptyset)$-definable subgroups of $G$.
(This condition holds, in particular, if $T$ satisfies the model-theoretic condition known as $1$-basedness, cf.~\cite{hp}.) By Example~\ref{ex:cosets}, if the collection of  $\acl^\eq(\emptyset)$-definable subgroups of $G$ has breadth at most $d$ (in particular, by Example~\ref{ex:dimension}, if it has height at most $d$), then   we have $\vc^T(1)\leq d$. 
\end{example}

Here is a particular instantiation of the previous example:

\begin{example}\label{ex:modules, 2}
Let $R$ be a ring, $M$ an $R$-module, and $T=\Th(M)$ in the language~$\mathcal L_R$, as in Example~\ref{ex:modules, 1}. 
We have $M^{\aleph_0}\equiv M^{(\aleph_0)}$ (see, e.g., \cite[Lemma~A.1.6]{Hodges} or \cite[Corollary~2.24]{Prest}). Set $T^{\aleph_0}:=\Th(M^{\aleph_0})=\Th(M^{(\aleph_0)})$. It is well-known that $T=T^{\aleph_0}$ iff
the class of models of $T$ is closed under direct products, iff
for all p.p.~$\mathcal L_R$-formulas $\varphi(x)$, $\psi(x)$, either $\varphi(M^{\abs{x}})\subseteq\psi(M^{\abs{x}})$ or the index 
$$\operatorname{Inv}(M,\varphi,\psi):=\big[\varphi(M^{\abs{x}}):(\varphi\wedge\psi)(M^{\abs{x}})\big]$$ is infinite.
(See, e.g., \cite[Lemma~A.1.7]{Hodges}.)
So if $T=T^{\aleph_0}$ and the Morley rank $\MRk(T)$ of $T$ is finite then
the length $n$ of every sequence
$$M\supsetneq \varphi_1(M)\supsetneq \varphi_1(M)\cap \varphi_2(M)\supsetneq\cdots \supsetneq \varphi_1(M)\cap\cdots\cap\varphi_n(M),$$
where each $\varphi_i(x)$ is a p.p.~$\mathcal L_R$-formula with $\abs{x}=1$, 
is bounded by $d=\MRk(T)$;
so by Examples~\ref{ex:cosets} and \ref{ex:dimension} we see that $\vc^T(1)\leq d$. (Note that this bound is far from optimal: e.g., for $R=\bZ$, $M=\bZ(p^d)^{(\aleph_0)}$ we have $\MRk(T)=d$, yet $\vc^T(1)=1$ by Example~\ref{ex:modules, 1}.)  
In \cite{ADHMS} we will extend this to $\vc^T(m)\leq md$ for every $m$.
\end{example}

\subsection{Dual VC density of sets of formulas}\label{sec:sets of formulas}
It is convenient to extend the definition of dual VC density to finite sets of formulas.
Let $\Delta=\Delta(x;y)$ be a finite set of partitioned $\mathcal L$-formulas $\varphi=\varphi(x;y)$ with the object variables $x$ and parameter variables~$y$.  We set 
$\neg\Delta:=\{\neg\varphi:\varphi\in\Delta\}$, and for $B\subseteq M^{\abs{y}}$ we let
$$\Delta(x;B) := \big\{ \varphi(x;b) : \varphi\in\Delta,\ b\in B \big\}.$$
Given a finite set $B\subseteq M^{\abs{y}}$, 
we call a consistent subset of $\Delta(x;B)\cup\neg\Delta(x;B)$ a $\Delta(x;B)$-type. Note that  our parameter sets are subsets of $M^{\abs{y}}$, and not of $M$, as is more common in model theory. (This is simply a matter of convenience, in order to be compatible with VC~duality.)
Given a  $\Delta(x;B)$-type $p$ we denote by $p^{\mathbf M}\subseteq M^{\abs{x}}$ its set of realizations in~$\mathbf M$. 
Since we are only dealing with finite sets $\Delta$ and finite parameter sets $B\subseteq M^{\abs{y}}$, all $\Delta(x;B)$-types  have realizations in $\mathbf M$ itself (rather than in an elementary extension).
Given another finite set $\Delta'(x;y')$ of partitioned $\mathcal L$-formulas and a finite $B'\subseteq M^{\abs{y'}}$,
we say that a $\Delta(x;B)$-type 
$p$ is equivalent to a $\Delta'(x;B')$-type $q$ if $p^{\mathbf M}=q^{\mathbf M}$.

Let now $B\subseteq M^{\abs{y}}$ be finite.
Given $a\in M^{\abs{x}}$   we denote the $\Delta(x;B)$-type of $a$ by
\begin{align*}
\operatorname{tp}^\Delta(a/B) := & \{\ \ \varphi(x;b): b\in B,\ \varphi\in\Delta,\ \mathbf M\models\varphi(a;b) \} \cup {} \\
& \{ \neg\varphi(x;b): b\in B,\ \varphi\in\Delta,\ \mathbf M\not\models\varphi(a;b) \}.
\end{align*}
We write $S^\Delta(B)$ for the set of complete $\Delta(x;B)$-types (in~$\mathbf M$), that is, the set of (in~$\mathbf M$) maximally consistent subsets of $\Delta(x;B)\cup\neg\Delta(x;B)$; equivalently,
$$S^\Delta(B) = \big\{ \operatorname{tp}^\Delta(a/B) : a\in M^{\abs{x}}\big\}.$$
If $\Delta=\{\varphi\}$ is a singleton, we also write $S^\varphi(B)$ instead of $S^\Delta(B)$. 
The elements of $S^\Delta(B)$ are syntactical objects (sets of formulas), but associating to a type $p\in S^\Delta(B)$ its set $p^{\mathbf M}$ of realizations in $\mathbf M$ gives a bijection from
$S^\Delta(B)$ onto the set 
$$S\big(\varphi^{\mathbf M}(M^{\abs{x}};b):b\in B,\varphi\in\Delta\big)$$ 
of atoms of the Boolean algebra generated by the subsets $\varphi^{\mathbf M}(M^{\abs{x}};b)$ of $M^{\abs{x}}$. (See Section~\ref{sec:IND}.)
Hence for every partitioned $\mathcal L$-formula $\varphi(x;y)$ we have
$$\pi^*_\varphi(t) = \max\big\{ \abs{S^\varphi(B)}: B\subseteq M^{\abs{y}},\ \abs{B}=t\big\}.$$
In the general case, for every $t\in\bN$ we also set
$$\pi^*_\Delta(t) :=\max\big\{ \abs{S^\Delta(B)}:B\subseteq M^{\abs{y}},\ \abs{B}=t\big\},$$
so $0\leq \pi^*_\Delta(t) \leq 2^{\abs{\Delta}t}$.
Similarly as in Lemma~\ref{lem:shatter function and elementary equivalence} one shows that if we pass from $\mathbf M$ to an elementarily equivalent $\mathcal L$-structure then $\pi^*_\Delta$ does not change (justifying our notation, which suppresses $\mathbf M$).

Let $\Delta_0(x;y)$  be a finite set of partitioned $\mathcal L$-formulas with $\Delta_0\subseteq\Delta$, and  $B\subseteq M^{\abs{y}}$ be finite. Then there is a natural restriction map $S^{\Delta}(B)\to S^{\Delta_0}(B)$, written as $p\mapsto p\restrict\Delta_0$. This map is onto: given $p\in S^{\Delta_0}(B)$ let $a\in p^{\mathbf M}$ be arbitrary; then $q:=\operatorname{tp}^\Delta(a/B)\in S^{\Delta}(B)$ satisfies $q\restrict\Delta_0=p$. In particular, $\abs{S^{\Delta_0}(B)}\leq \abs{S^{\Delta}(B)}$.
Note also that if $\Delta\neq\emptyset$, then the restriction maps $p\mapsto p\restrict\varphi$, where $\varphi\in\Delta$, combine to an injective map $S^\Delta(B)\to \prod_{\varphi\in\Delta} S^\varphi(B)$; in particular, $\abs{S^\Delta(B)}\leq \prod_{\varphi\in\Delta} \abs{S^\varphi(B)}$. This shows:

\begin{lemma}
If all $\varphi\in\Delta$ are dependent, then there exists a real number~$r$ with $0\leq r\leq \sum_{\varphi\in\Delta} \vc^*(\varphi)$ and
\begin{equation}\label{eq:density of Delta}
\abs{S^\Delta(B)} = O(\abs{B}^r) \qquad\text{for all finite $B\subseteq M^{\abs{y}}$.}
\end{equation}
\end{lemma}

We define the \emph{dual VC density of $\Delta$} as the infimum $\vc^*(\Delta)$ of all real numbers $r\geq 0$ such that \eqref{eq:density of Delta} holds; that is, 
$$\vc^*(\Delta)=\inf\big\{r\geq 0:\pi^*_\Delta(t)=O(t^r)\big\}.$$ 
We have
$$\max_{\varphi\in\Delta}\vc^*(\varphi)\leq\vc^*(\Delta)\leq \sum_{\varphi\in\Delta} \vc^*(\varphi).$$ 
Clearly
$\vc^*(\Delta)$ agrees with $\vc^*(\varphi)$ as defined previously if $\Delta=\{\varphi\}$ is a singleton.
Moreover, $\vc^*(\Delta)=0$ iff $\vc^*(\varphi)=0$ for every $\varphi\in\Delta$, and if $\vc^*(\Delta)<1$ then $\vc^*(\Delta)=0$. 
(See the remarks following Lemma~\ref{lem:inverse image}.)
Note that in computing $\vc^*(\Delta)$ there is no harm in assuming that $\Delta$ is \emph{closed under negation,} i.e., with every $\varphi\in\Delta$ the set $\Delta$ also contains a formula equivalent (in $\mathbf M$) to $\neg\varphi$. (Passing from $\Delta$ to $\Delta\cup \neg\Delta$ does not change $S^\Delta(B)$.)

\begin{example*}
Suppose $\Delta(x;y)=\{x_1=y,\dots,x_m=y\}$ where $\abs{x}=m$ and $\abs{y}=1$. Then for finite $B\subseteq M$ we have 
$\abs{S^\Delta(B)}=(\abs{B}+1)^m$, hence $\vc^*(\Delta)=\sum_{\varphi\in\Delta}\vc^*(\varphi)=m$.
\end{example*}

We finish this subsection with an easy result about interpretations (related to Lem\-ma~\ref{lem:reparametrization} and Corollary~\ref{cor:inverse image}).

\begin{lemma} \label{interp}
Let $\mathbf M'$ be an infinite structure in a language $\mathcal L'$ and
$\pi\colon X\to M'$ an interpretation of $\mathbf M'$ in $\mathbf M$ without parameters, where $X\subseteq M^r$ is $\emptyset$-definable.
Then for any finite set $\Delta'(x;y)$ of $\mathcal L'$-formulas there exists a finite set $\Delta(\bx;\by)$ of $\mathcal L$-formulas such that $\abs{\Delta}=\abs{\Delta'}$, $\abs{\bx}=r\abs{x}$, and $\pi^*_{\Delta'}\leq\pi^*_{\Delta}$.
\end{lemma}

\begin{proof}
Let $m:=\abs{x}$ and $n:=\abs{y}$. Let
$B'\subseteq (M')^n$ be finite. Choose $B\subseteq X^n$ with $\abs{B}=\abs{B'}$ such that each
$b=(b_1,\ldots,b_n)\in B'$ has the form $(\pi(\bb_1),\ldots,\pi(\bb_n))$ for some
$(\bb_1,\ldots,\bb_n)\in B$.
For each $\mathcal L'$-formula
$\ph(x;y)$ choose an $\mathcal L$-formula $\psi_\ph(\bx;\by)$, where $\bx=(\bx_1,\dots,\bx_m)$, $\by=(\by_1,\dots,\by_n)$ and
$\abs{\bx_1}=\dots=\abs{\bx_m}=\abs{\by_1}=\dots=\abs{\by_n}=r$, such that 
$\psi_\ph(M^{(m+n)r})\subseteq X^{m+n}$ and
for any
$\ba_1,\ldots,\ba_m,\bb_1,\ldots\,\bb_n\in X$, 
$$\mathbf M\models \psi_\ph(\ba_1,\ldots\ba_m;\bb_1,\ldots,\bb_n)\ \Longleftrightarrow\ \mathbf M'\models \ph\big(\pi(\ba_1),\ldots,\pi(\ba_m);\pi(\bb_1),\ldots,\pi(\bb_n)\big).$$
Let a finite set $\Delta'(x;y)$ of $\mathcal L'$-formulas be given. Set $\Delta:=\{\psi_\ph:\ph \in \Delta'\}$. 
Then $S^\Delta(B)\subseteq X^m$, and
$(\ba_1,\ldots\ba_m)\mapsto (\pi(\ba_1),\ldots,\pi(\ba_m))$ yields a surjective map $S^\Delta(B)\to S^{\Delta'}(B')$, hence 
$\abs{S^{\Delta'}(B')}\leq \abs{S^\Delta(B)}$ as required.
\end{proof}

By Lemmas~\ref{lem:expansions by constants} and \ref{interp}:

\begin{corollary}\label{cor:interp}
Let $\mathbf M'$ be an infinite structure in a language $\mathcal L'$, interpretable in $\mathbf M$ \textup{(}possibly with parameters\textup{)} on a definable subset of $M^r$. Then, writing $T=\Th(\mathbf M)$ and $T'=\Th(\mathbf M')$, we have $\vc^{T'}(m)\leq \vc^T(rm)$ for every $m$.
\end{corollary}

So for example if $G$ is a group (considered as a structure in the usual first-order language of group theory) and $H$ is a definable normal subgroup of $G$, then $\vc^{\Th(G/H)}\leq\vc^{\Th(G)}$ if $H$ has infinite index in $G$, and $\vc^{\Th(H)}\leq\vc^{\Th(G)}$ if $H$ is infinite. 

\subsection{Coding finite sets of formulas}\label{sec:coding}
We let $\mathcal L$, $\mathbf M$ and $\Delta$ be as in the previous subsection, and $T=\Th(\mathbf M)$.
The following useful lemma, essentially due to Shelah \cite[Lemma~II.2.1]{Shelah-book}, shows that counting $\Delta(x;B)$-types where $\abs{\Delta}>1$ is not really more general than counting $\Delta(x;B)$-types where $\Delta$ is a singleton:

\begin{lemma}\label{lem:encoding finite sets of formulas}
Let $d=\abs{\Delta}$ and  $y'=(y_1,\dots,y_{2d},z,z_1,\dots,z_{2d})$ with $\abs{y}=\abs{y_i}=\abs{z_i}=\abs{z}$ for every $i=1,\dots,2d$.
There is an $\mathcal L$-formula $\psi_\Delta(x;y')$
with the following properties:
\begin{enumerate}
\item 
for every finite $B\subseteq M^{\abs{y}}$ with $\abs{B}\geq 2$ there is some $B'\subseteq M^{\abs{y'}}$ with 
$\abs{B'}=2d\abs{B}$ such that
every $p\in S^\Delta(B)$ is equivalent to some $q\in S^{\psi_\Delta}(B')$;
\item for every finite $B'\subseteq M^{\abs{y'}}$ there is some $B\subseteq M^{\abs{y}}$ with $\abs{B}\leq 2d\abs{B'}$ 
such that every $q\in S^{\psi_\Delta}(B')$ is equivalent to some \textup{(}possibly incomplete\textup{)} $\Delta(x;B)$-type~$p_0$.
\end{enumerate}
In particular, we have $\pi^*_\Delta(t)\leq \pi^*_{\psi_\Delta}(2dt)$ for $t>1$ and $\pi^*_{\psi_\Delta}(t)\leq\pi^*_\Delta(2dt)$ for $t\geq 0$. Thus $\vc^*(\Delta)=\vc^*(\psi_\Delta)\leq \vc^T(m)$ where $m=\abs{x}$.
\end{lemma}
\begin{proof}
Write  $\Delta=\{\varphi_1,\dots,\varphi_d\}$ and define $\psi_\Delta$ as follows:
\begin{multline*}
\psi_\Delta= \bigwedge_{k=1}^d \big(z=z_k \to \varphi_k(x;y_k)\big) \wedge
\bigwedge_{k=d+1}^{2d} \big(z=z_k \to \neg\varphi_{k-d}(x;y_k)\big) \wedge \\
\bigvee_{k=1}^{2d} z=z_k \wedge \bigwedge_{1\leq k<l\leq 2d} \neg (z=z_k\wedge z=z_l).
\end{multline*}
For (1), suppose $B\subseteq M^{\abs{y}}$ is finite, and $b_0\neq b_1$ are distinct elements of $B$. 
For $b\in B$ and $k\in [d]$ set
$$\begin{array}{rclllllllllr}
b_0^{(k)} := 	\big(b_0, 	& b_0, 	&\dots, 	&  b,		& \dots,	&   b_0,	 & b_1, 	& b_0, 	& \dots,  & b_1,  & \dots,  & b_0\big) \\
			\ y_1 							& y_2 							& \dots		&  y_{d+k}	& \dots		& y_{2d} & z 		& z_1 & \dots & z_{d+k} & \dots & z_{2d}\ \
\end{array} $$
and
$$\begin{array}{rclllllllllr}
b_1^{(k)} := 	\big(b_0, 	& b_0, 	&\dots, 	&  b,		& \dots,	&   b_0,	 & b_1, 	& b_0, 	& \dots,  & b_1,  & \dots,  & b_0\big) \\
			\ y_1 							& y_2 							& \dots		&  y_k	& \dots		& y_{2d} & z 		& z_1 & \dots & z_k & \dots & z_{2d}\ \
\end{array}, $$
and put 
$$B':=\big\{b_0^{(k)}, b_1^{(k)} :b\in B,\ k\in [d]\big\}\subseteq (M^{\abs{y}})^{4d+1}.$$
Then $\abs{B'}=2d\abs{B}$, and for every $b\in B$, $k\in [d]$ we have
$$\psi_\Delta(M^{\abs{x}};b_0^{(k)}) = \neg\varphi_k(M^{\abs{x}};b), \qquad
\psi_\Delta(M^{\abs{x}};b_1^{(k)}) = \varphi_k(M^{\abs{x}};b).$$
Given $p\in S^\Delta(B)$ we set
\begin{align*}
q&:=\{ \neg\psi_\Delta(x;b_0^{(k)}), \ \ \psi_\Delta(x;b_1^{(k)}) : \varphi_k(x;b)\in p \}\ \cup 
\\ & \quad\ \ 
\{ \ \ \psi_\Delta(x;b_0^{(k)}), \neg\psi_\Delta(x;b_1^{(k)}) : \varphi_k(x;b)\notin p \}.
\end{align*}
Then clearly $q\in S^{\psi_\Delta}(B')$, and $q$ is equivalent to $p$. The map $p\mapsto q\colon S^{\Delta}(B)\to S^{\psi_\Delta}(B')$ is injective, hence $\abs{S^\Delta(B)}\leq \abs{S^{\psi_\Delta}(B')}\leq  \pi^*_{\psi_\Delta}(2d\abs{B})$.

For (2) note that if $b_1,\dots,b_{2d},c,c_1,\dots,c_{2d}\in M^{\abs{y}}$ then the formula $$\psi_\Delta(x;b_1,\dots,b_{2d},c,c_1,\dots,c_{2d})$$ defines $\varphi_k(M^{\abs{x}};b_k)$, $\neg\varphi_k(M^{\abs{x}};b_{k+d})$, or $\emptyset$ (since the $c_i$'s are not necessarily distinct). Let $B'\subseteq M^{\abs{y'}}$ be finite, and $q\in S^{\Delta}(B')$.
Set 
$$B:=\big\{b\in M^{\abs{y}}: \text{$b=b_i$ for some $(b_1,\dots,b_{2d},c,c_1,\dots,c_{2d})\in B'$} \big\}$$ 
and let $p_0$ be the set of formulas which have the form
$\varphi_k(x;b)$ where
$k\in[d]$, $b=b_k$ for some $\psi_\Delta(x;b_1,\dots,b_{2d},c,c_1,\dots,c_{2d})\in q$ with $c=c_k$, or the form $\neg\varphi_k(x;b)$ with $k\in[d]$, $b=b_{d+k}$ for some $\psi_\Delta(x;b_1,\dots,b_{2d},c,c_1,\dots,c_{2d})\in q$ with $c=c_{k+d}$. Then $\abs{B}\leq 2d\abs{B'}$, and $p_0$ is a $\Delta(x;B)$-type equivalent to $q$.
For each $q$ choose an extension $p$ of $p_0$ to a complete $\Delta(x;B)$-type. Then the map $q\mapsto p\colon S^{\psi_\Delta}(B')\to S^\Delta(B)$ is injective, so  $\abs{S^{\psi_\Delta}(B')}\leq \abs{S^{\Delta}(B)}\leq  \pi^*_{\Delta}(2d\abs{B'})$.
\end{proof}

In the rest of this subsection we give some applications of this lemma. We first note:

\begin{corollary}\label{cor:encoding finite sets of formulas}
Let $\Phi$ be a set of $\mathcal L$-formulas with the tuple of object variables $x$ and varying parameter variables such that every $\mathcal L$-formula $\varphi(x;y)$ is equivalent in $T$ to a Boolean combination of formulas in $\Phi$. Then 
$$\vc^T(m) = \sup\big\{ \vc^*(\Delta):  \text{$\Delta\subseteq\Phi$ finite }\big\}\qquad\text{where $m=\abs{x}$.}$$
\end{corollary}
\begin{proof}
The inequality ``$\leq$'' is a consequence of the hypothesis: for each $\mathcal L$-formula $\varphi(x;y)$ 
there is a finite subset $\Delta=\Delta(x;y)$ of $\Phi$ such that $\abs{S^\varphi(B)}\leq\abs{S^\Delta(B)}$
for each finite $B\subseteq M^{\abs{y}}$.
The reverse inequality follows from the previous lemma. 
\end{proof}

Let  $\mathbf M^*\succcurlyeq \mathbf M$ be a monster model of $T$. Consider the expansion $\mathcal L^{\Sh}$ of $\mathcal L$ by a new predicate symbol $R_{\psi,c}(x)$ for every $\mathcal L$-formula $\psi(x;z)$ and every $c\in (M^*)^{\abs{z}}$.
The \emph{Shelah expansion} of $\mathbf M$ is the expansion of $\mathbf M$ to an $\mathcal L^{\Sh}$-structure $\mathbf M^{\Sh}$ where each predicate symbol $R_{\psi,c}(x)$ as before is interpreted by $M^{\abs{x}}\cap \psi^{\mathbf M^*}((M^*)^{\abs{x}};c)$. Shelah showed \cite{S-dependent} (with another proof given in \cite{CS}) that if $T$ is NIP then $T^{\Sh}=\Th(\mathbf M^{\Sh})$ admits quantifier elimination and is also NIP. This provides an interesting way of constructing new NIP theories from old ones.
The previous lemma and its Corollary~\ref{cor:encoding finite sets of formulas} allows us to prove that $T$ and $T^{\Sh}$ share the same VC~density function:

\begin{corollary}\label{cor:Shelah expansion} $\vc^{T^{\Sh}}=\vc^T$.
\end{corollary}
\begin{proof}
Fix some $m$ and assume $\abs{x}=m$. The inequality $\vc^{T^{\Sh}}(m)\geq\vc^T(m)$ being obvious, we only need to show that $\vc^{T^{\Sh}}(m)\leq\vc^T(m)$. 
Let $\Delta=\Delta(x;y)$ be a finite set of atomic $\mathcal L^{\Sh}$-formulas; by Corollary~\ref{cor:encoding finite sets of formulas} and Shelah's theorem mentioned above, it suffices to show that $\vc^*(\Delta)\leq\vc^T(m)$.
Take a finite set $\Psi=\Psi(x;y,z)$ of partitioned $\mathcal L$-formulas and some $c\in M^{\abs{z}}$ such that $\Delta=\{R_{\psi,c}(x;y):\psi\in\Psi\}$. Let $B\subseteq M^{\abs{y}}$ be finite, $B^*:=B\times\{c\}$, and let $p\in S^\Delta(B)$.
Let $a$ be an arbitrary realization of $p$ (in $\mathbf M^{\Sh}$), and define $p^*:=\tp^\Psi(a/B^*)$ (in $\mathbf M^*$). Then for $\psi\in\Psi$ and $b\in B$ we have
\begin{align*}
\psi(x;b,c)\in p^*	&\quad\Longleftrightarrow\quad \mathbf M^*\models\psi(a;b,c) \\
					&\quad\Longleftrightarrow\quad \mathbf M^{\Sh}\models R_{\psi,c}(a;b) \\
					&\quad\Longleftrightarrow\quad R_{\psi,c}(x;b)\in p.
\end{align*}
In particular, the map $p\mapsto p^*\colon S^\Delta(B)\to S^\Psi(B^*)$ is injective, so
$\vc^*(\Delta)\leq\vc^*(\Psi)\leq\vc^T(m)$ by Lemma~\ref{lem:encoding finite sets of formulas}.
\end{proof}

It is well-known (see, e.g., \cite[Theorem~4.7]{Wierzejewski}) that the direct product of two NIP structures is again NIP.
As a consequence of the last lemma we can also now estimate the VC~density of a direct product in terms of the VC~densities of its factors. We refer to \cite[Section~9.1]{Hodges} for the definition of the product of two $\mathcal L$-structures, and to \cite[Corollary~9.6.4]{Hodges} for the Feferman-Vaught Theorem used in the proof below.

\begin{lemma}\label{lem:product}
Let $\mathbf M'$ be another infinite $\mathcal L$-structure, $T'=\Th(\mathbf M')$, and let $T^\times=\Th(\mathbf M\times\mathbf M')$ be the $\mathcal L$-theory of the direct product of $\mathbf M$ and $\mathbf M'$. Then 
$$\vc^{T^\times}\leq\vc^T+\vc^{T'}.$$
\end{lemma}
\begin{proof}
Given $n$-tuples $a=(a_1,\dots,a_n)\in M^n$ and $a'=(a_1',\dots,a_n')\in (M')^n$ we denote by $a\times a'$ the $n$-tuple $((a_1,a_1'),\dots,(a_n,a_n'))$ of elements of $M\times M'$; every element of $(M\times M')^n$ has the form $a\times a'$ for some $a\in M^n$, $a'\in (M')^n$.

Let $\varphi(x;y)$ be an $\mathcal L$-formula. By the Feferman-Vaught Theorem there exist finitely many pairs of $\mathcal L$-formulas $(\theta_i(x;y),\theta_i'(x;y))$, $i\in [n]=\{1,\dots,n\}$, such that for all $a\in M^{\abs{x}}$, $a'\in (M')^{\abs{x}}$ and $b\in M^{\abs{y}}$, $b'\in (M')^{\abs{y}}$,
$$\mathbf M\times\mathbf M'\models \varphi(a\times a';b\times b')\quad\Longleftrightarrow\quad 
\text{for some $i\in [n]$, $\mathbf M\models\theta_i(a;b)$ and $\mathbf M'\models\theta_i'(a';b')$.}
$$
Set $\Theta:=\{\theta_1,\dots,\theta_n\}$, $\Theta':=\{\theta_1',\dots,\theta_n'\}$.
Let $C$ be a finite set of tuples from $(M\times M')^{\abs{y}}$. 
Take $B\subseteq M^{\abs{y}}$, $B'\subseteq (M')^{\abs{y}}$ with $\abs{B},\abs{B'}\leq \abs{C}$ such that each $c\in C$ is of the form $c=b\times b'$ for a unique pair $(b,b')\in B\times B'$.
For every $p\in S^\varphi(C)$ choose a realization $a_p\times a_p'\in (M\times M')^{\abs{x}}$ of $p$ in $\mathbf M\times \mathbf M'$, and put
$$q := \tp^{\Theta}(a_p/B),\qquad q' := \tp^{\Theta'}(a_p'/B').$$
Then for all $(b,b')\in B\times B'$ we have
\begin{align*}
\varphi(x;b\times b')\in p	& \quad\Longleftrightarrow \quad \mathbf M\times\mathbf M'\models \varphi(a_p\times a_p';b\times b') \\
							& \quad\Longleftrightarrow \quad \text{$\mathbf M\models\theta_i(a_p;b)$ and $\mathbf M'\models\theta_i'(a_p';b')$, for some $i\in [n]$} \\
							& \quad\Longleftrightarrow \quad \text{$\theta_i(x;b)\in q$ and $\theta'_i(x;b')\in q'$, for some $i\in [n]$.}
\end{align*}
Hence the map $p\mapsto (q,q')$ is an injection
$S^\varphi(C)\to S^\Theta(B)\times S^{\Theta'}(B')$. In particular we obtain
$\pi^*_\varphi(t) \leq \pi^*_\Theta(t)\cdot \pi^*_{\Theta'}(t)$ for every $t$ and hence
$\vc^*(\varphi)\leq \vc^*(\Theta)+\vc^*(\Theta')$; here $\pi^*_\varphi$ is computed in $\mathbf M\times \mathbf M'$ and
$\pi^*_\Theta$, $\pi^*_{\Theta'}$ in $\mathbf M$ and $\mathbf M'$, respectively, and similarly for $\vc^*$. By Lemma~\ref{lem:encoding finite sets of formulas} therefore $\vc^{T^\times}(m)\leq \vc^T(m)+\vc^{T'}(m)$ where $m=\abs{x}$.
\end{proof}

\begin{remark*}
In a similar way one shows that if $\mathbf M'$ is a finite $\mathcal L$-structure and $T^\times=\Th(\mathbf M\times\mathbf M')$, then $\vc^{T^\times}\leq \vc^T$.
\end{remark*}

We finish this subsection by noting a further restriction on the growth of $\vc$ (cf.~also Lemma~\ref{lem:lower estimate, 1}):

\begin{lemma}\label{lem:lower estimate, 2}
$d\vc(m)\leq \vc(dm)$ for all $d,m>0$.
\end{lemma}
\begin{proof}
Let $\Delta(x;y)$ be a finite set of $\mathcal L$-formulas with $\abs{x}=m$. Let $x_1,\dots,x_d$ be new $m$-tuples of variables and set
$$\Delta'(x_1,\dots,x_d;y) := \big\{ \varphi(x_i;y): \varphi(x;y)\in\Delta,\ i=1,\dots,d\big\}.$$
Let $B\subseteq M^{\abs{y}}$, $\abs{B}=t\in\bN$, such that $r:=\pi^*_\Delta(t)=\abs{S^\Delta(B)}$. Let $a_1,\dots,a_r\in M^{m}$ be realizations of the types in $S^\Delta(B)$. For each $\mathbf i=(i_1,\dots,i_d)\in [r]^d$ let $a_{\mathbf i}:=(a_{i_1},\dots,a_{i_d})\in (M^m)^d=M^{dm}$. Then the $a_{\mathbf i}$ realize pairwise distinct $\Delta'(x_1,\dots,x_d;B)$-types. This yields $(\pi^*_\Delta(t))^d=\abs{S^\Delta(B)}^d \leq \abs{S^{\Delta'}(B)}\leq\pi^*_{\Delta'}(t)$.
Since $t$ was arbitrary, we obtain $d\vc^*(\Delta)\leq\vc^*(\Delta')$. Hence $d\vc(m)\leq \vc(dm)$ by Lemma~\ref{lem:encoding finite sets of formulas}.
\end{proof}

\subsection{VC density and indiscernible sequences}
In this subsection we assume that $\mathbf M$ is sufficiently saturated. Recall that $\pi_\varphi(t)$ is the maximum  size of $\mathcal S_\varphi\cap A$ as $A$ ranges over $t$-element subsets of $M^{m}$, and $\pi^*_\varphi(t)$ is the maximum size of $S^\varphi(B)$ as $B$ ranges over all $t$-element subsets of $M^{n}$; here, as above $m=\abs{x}$, $n=\abs{y}$.
These definitions may naturally be relativized to parameters coming from indiscernible sequences. More precisely:

\begin{definition}
For every $t$ let $\pi_{\varphi,\ind}(t)$  be the maximum of $\abs{\mathcal S_\varphi\cap A}$ as $A$ ranges  over all sets of the form $A=\{a_0,\dots,a_{t-1}\}$ for some indiscernible sequence $(a_i)_{i\in\bN}$ in $M^{m}$, and let $\pi^*_{\varphi,\ind}(t)$  be the maximum of $\abs{S^\varphi(B)}$ where $B=\{b_0,\dots,b_{t-1}\}$ for some indiscernible sequence $(b_i)_{i\in\bN}$ in $M^{n}$. We call $\pi_{\varphi,\ind}$ the \emph{indiscernible shatter function} of $\varphi$ and $\pi^*_{\varphi,\ind}$ the \emph{dual indiscernible shatter function} of $\varphi$.
\end{definition}

The indiscernible shatter functions give rise to corresponding notions of \emph{indiscernible VC dimension $\VC_\ind(\varphi)$} and \emph{indiscernible VC density $\vc_\ind(\varphi)$} of $\varphi$ (and their duals $\VC_\ind^*(\varphi)$ and $\VC_\ind^*(\varphi)$) in a natural way; for example, $\vc^*_\ind(\varphi)$ is the infimum of all $r>0$ having the property that there is some $C>0$ such that for all $t$ and indiscernible sequences $(b_i)_{i\in\bN}$ we have $\abs{S^\varphi(B)}\leq Ct^r$, where $B=\{b_0,\dots,b_{t-1}\}$; if there is no such~$r$ then $\vc^*_\ind(\varphi)=\infty$.

As in the classical case (cf.~Lemma~\ref{lem:dual shatter}) we see that $\pi^*_{\varphi,\ind}=\pi_{\varphi^*,\ind}$ and hence $\VC_\ind(\varphi^*)=\VC^*_\ind(\varphi)$ and $\vc_\ind(\varphi^*)=\vc^*_\ind(\varphi)$.
Directly from the definition we have $\pi_{\varphi,\ind}\leq\pi_\varphi$ and hence
$\VC_\ind(\varphi)\leq\VC(\varphi)$ and $\vc_\ind(\varphi)\leq\vc(\varphi)$.
In particular $\VC_\ind(\varphi)$ and $\vc_\ind(\varphi)$ are finite if $\varphi$ defines a VC~class.
Conversely,  if $\VC_\ind(\varphi)$ is finite, then so is $\VC(\varphi)$. (This follows by saturation of $\mathbf M$ and extraction of an indiscernible sequence; see proof of Proposition~4 in \cite{Adler}.)
Hence if one of the quantities $\VC(\varphi)$, $\vc(\varphi)$, $\VC_\ind(\varphi)$, or $\vc_\ind(\varphi)$ is finite, then so are all the others.

\medskip

Another numerical parameter associated to $\varphi$ and defined via indiscernible sequences is the \emph{alternation number $\alt(\varphi)$} of $\varphi$ (in $\mathbf M$). This is the largest $d$ (if it  exists) such that for some indiscernible sequence $(a_i)_{i\in\bN}$ in $M^{m}$ and some $b\in M^{n}$ we have 
$$a_i\in \varphi(M^{m};b)\quad\Longleftrightarrow\quad a_{i+1}\notin \varphi(M^{m};b)\qquad\text{for all $i<d-1$.}$$ 
If there is no such $d$ we set $\alt(\varphi)=\infty$.
It is well-known (and essentially due to Poizat) that $\alt(\varphi)\leq 2\VC_\ind(\varphi)+1$ (see, e.g., \cite[Proposition~3]{Adler}) and that if $\alt(\varphi)$ is finite then $\varphi$ defines a VC~class \cite[Proposition~4]{Adler}.
Moreover:

\begin{lemma}\label{lem:alt}
$\vc_\ind(\varphi)\leq\alt(\varphi)-1$.
\end{lemma}
\begin{proof}
Since this is trivial if $\varphi$ has infinite alternation number, we assume that $d:=\alt(\varphi)<\infty$.
Let $(a_i)_{i\in\bN}$ be an indiscernible sequence in $M^{m}$ and $A=\{a_0,\dots,a_{t-1}\}$. Then for each $b\in M^{n}$, there are less than $d$ indices $i<t-1$ such that $\varphi(a_i;b)$ and $\varphi(a_{i+1};b)$ have different truth value in $\mathbf M$,
and the set $A\cap\varphi(M^{m};b)$ is uniquely determined by knowledge of these indices $i$.
Thus $\abs{A\cap\mathcal S_\varphi}\leq 2\sum_{i=0}^{d-1}{t\choose i}=O(t^{d-1})$ and hence $\vc_\ind(\varphi)\leq d-1$ as required.
\end{proof}

\begin{example*}
Suppose $\mathcal S_\varphi \subseteq {M^{m}\choose d}$ where $d>0$. Then $\alt(\varphi) \leq 2d+1$ and $\vc_\ind(\varphi)\leq\vc(\varphi)\leq d$, and all these inequalities are equalities if  $\mathcal S_\varphi = {M^{m}\choose d}$.
\end{example*}

The previous example shows that the inequality in Lemma~\ref{lem:alt},  in general, is strict.
The inequality $\VC_\ind(\varphi)\leq\VC(\varphi)$ may be strict if there are no non-trivial indiscernible sequences:

\begin{example*}
Suppose $\mathcal L=\{A,S,P\}$ where $A$ and $S$ are unary relation symbols and $P$ is a binary relation symbol, and suppose $\mathbf M$ is an $\mathcal L$-structure, with the interpretations of $A$,  $S$ and $P$ in $\mathbf M$ denoted by the same symbols, such that
\begin{enumerate}
\item $\abs{A}=d$ and $\abs{S}=2^d$;
\item for $s\in S$, $P(x,s)$ defines a subset of $A$ so that when $s$ runs through $S$ we obtain all subsets of $A$;
\item for $s\notin S$, $P(x,s)$ defines the empty set.
\end{enumerate}
Then $\VC(P)=d$ and $\VC_\ind(P)=1$ (as well as $\vc(P)=\vc_\ind(P)=0$).
\end{example*}

The inequality $\vc_\ind(\varphi)\leq\vc(\varphi)$ may also be strict, as Lemma~\ref{lem:alt for SzTr} in the next section shows. 
We do not know the answer to the following question:

\begin{question}
Is $\vc_\ind(\varphi)$ always integral-valued? 
\end{question}

(After a first version of this manuscript had been completed, Guingona and Hill~\cite{GH} showed that this question indeed has a positive answer.)

\medskip



We finish this section with a connection between $\vc^*_\ind$ and the Helly number.
We already remarked (see Section~\ref{sec:breadth}) that if $\mathbf M=(M,{<})$ is a dense linearly ordered set and $\varphi(x;y_1,z_1,y_2,z_2)=(y_1<x<z_1 \vee y_2<x<z_2)$ then the set system $\mathcal S_\varphi$ has infinite Helly number: that is, for each $d$ there is a finite subfamily of $\mathcal S_\varphi$ which is $d$-consistent yet inconsistent. In contrast to this, we have:

\begin{lemma}\label{lem:Helly indisc}
Put $d=\lfloor \vc^*_\ind(\varphi)\rfloor+1$. Then for every indiscernible sequence $(b_i)_{i\in\bN}$ in $M^{\abs{y}}$ the set system $\mathcal S=\{\varphi(M^{m};b_i):i\in\bN\}$ has Helly number at most~$d$.
\end{lemma}
\begin{proof}
Suppose for a contradiction that $(b_i)_{i\in\bN}$ is an indiscernible sequence such that $\mathcal S=\{\varphi(M^{m};b_i):i\in\bN\}$ has Helly number larger than $d$. Then
some finite  subfamily $\mathcal S_0$ of $\mathcal S$ is $d$-consistent but not consistent. By  indiscernibility of $(b_i)$, \emph{every} finite subfamily of $\mathcal S$ of size at least $\abs{\mathcal S_0}$ has this property.
In particular, we can take $D\in \mathbb N$ maximal such that the 
set $\{ \varphi(M^{m};b_i): i < D\}$ is consistent. Obviously $D\geq d$. Since $(b_i)$
is indiscernible, we obtain that for any $I_0\in {\bN\choose D}$ the set 
$\{ \varphi(M^{m};b_i) : i\in I_0\}$ is consistent, but for any $D'>D$ and any    
$I_1\in {\bN\choose D'}$ the set 
$\{ \varphi(M^{m};b_i) : i\in I_1\}$ is inconsistent.
Let $t>D$ be arbitrary, and set $B_t=\{ b_i : i<t \}$. 
For $I\in { t \choose D }$ let $q_I(x)$ be the unique  $\varphi$-type over $B_t$
with $\varphi(x;b_i)\in q_I$ for $i\in I$ and   $\neg \varphi(x;b_i)\in q_I$ for 
$i\not\in I$. Since $\abs{I}=D$ every $q_I$ is consistent.
Thus $\abs{S^\varphi(B_t)}\geq {t\choose D}=\Theta(t^D)$. Since $D\geq d$, this contradicts  
$\vc_\ind^*(\varphi)<d$.
\end{proof}

\begin{remark*}
Note that in the context of the previous lemma, we cannot achieve the stronger conclusion that $\mathcal S$ has breadth at most $d$: for the formula $\varphi(x;y)=x\neq y$ and any indiscernible sequence $(b_i)$, the set system $\mathcal S$ always has infinite breadth.
\end{remark*}

By Lemma~\ref{lem:Helly indisc} and extraction of an indiscernible sequence (using that $\mathbf M$ is assumed to be sufficiently saturated) we obtain a consequence which does not mention indiscernibles:

\begin{corollary}
Suppose the set system $\mathcal S_\varphi$ is $d$-consistent, where $d=\lfloor \vc^*(\varphi)\rfloor+1$. Then there is an infinite subset of $\mathcal S_\varphi$ which is consistent.
\end{corollary}

This is a weak version of a theorem of Matou{\v{s}}ek \cite{Matousek-fractional}, according to which, if  $\mathcal S_\varphi$ is $d$-consistent, where $d>\vc^*(\varphi)$, then one may write $\mathcal S_\varphi=\mathcal S_1\cup\cdots\cup \mathcal S_N$ (for some $N\in\bN$) where each $\mathcal S_i$ is consistent.

\section{Some VC Density Calculations}\label{sec:calculations}

\noindent
In this section we give an example of a formula in the language of rings which, in every infinite field, defines a set system with fractional VC density, depending on the characteristic of the field.  The construction of this formula (which is inspired by an example by Assouad \cite{Assouad}, who in turn credits Frankl) proceeds in two steps: we first associate to a given partitioned formula $\varphi$ a bigraph (= bipartite graph with a fixed ordering of the bipartition of the vertex set), and then we realize the set of edges of this bigraph as a definable family $\mathcal S_{\widehat\varphi}$.
For our example we choose $\varphi$ so as to encode point-line incidences in the affine plane; the calculation of $\vc(\widehat\varphi)$ in characteristic zero uses an analogue of the Sz\'emeredi-Trotter Theorem due to T\'oth.
We also discuss the question whether VC~density in NIP theories can take irrational values, and give examples of formulas in NIP theories whose shatter function is not asymptotic to a real power function.

Throughout this section $\mathcal L$ is a first-order language and $\mathbf M$ is an $\mathcal L$-structure.

\subsection{Associating a bigraph to a partitioned formula}\label{sec:bigraphs}
We follow \cite{LS} and make a distinction between bipartite graphs and bigraphs. A \emph{bipartite graph} is a graph $(V,E)$ whose  set $V$ of vertices can be partitioned into two classes such that all edges connect vertices in different classes.
By a \emph{bigraph} we mean a triple $G=(X,Y,\Phi)$ where $X$ and $Y$ are (not necessarily disjoint) sets and $\Phi\subseteq X\times Y$. 
Thus a bipartite graph can be viewed as a bigraph if we fix a partition and specify which bipartition class is first and second. Conversely, if $G=(X,Y,\Phi)$ is a bigraph then we obtain a bipartite graph $(V(G),E(G))$ (the bipartite graph associated to $G$) by letting $V(G)$ be the disjoint union of the sets $X$ and $Y$, and $E(G)=\Phi$; by abuse of language we call
$V(G)$ the set of \emph{vertices} of $G$ and $E(G)$  the set of \emph{edges} of $G$. We also say that $G$ is a bigraph \emph{on $V=V(G)$.} (What we call a bigraph $G=(X,Y,\Phi)$ is sometimes called an incidence structure, and  $(V(G),E(G))$ is called its Levi graph or incidence graph.)
A  bigraph is said to be finite if its set of vertices is finite.
It is easy to see that a finite bigraph $G$ can have at most $\frac{1}{4}\abs{V(G)}^2$ edges.

A bigraph $G'=(X',Y',\Phi')$ is a \emph{sub-bigraph} of
$G=(X,Y,\Phi)$ if $X\subseteq X'$, $Y\subseteq Y'$, and $\Phi'\subseteq \Phi$. We say that a bigraph $G$ contains a given bigraph $G'$ (as a sub-bigraph) if $G'$ is isomorphic to a sub-bigraph of $G$.
Given a bigraph $G=(X,Y,\Phi)$ and a subset $V$ of its vertex set $V(G)$, we denote by 
$$G\restrict V:=\big(X\cap V,Y\cap V,\Phi\cap (V\times V)\big)$$ the sub-bigraph of $G$  \emph{induced on $V$.} 
The \emph{complement} of a  bigraph $G=(X,Y,\Phi)$ is the bigraph $\neg G:=(X,Y,\neg \Phi)$, and its \emph{dual} is $G^*:=(Y,X,\Phi^*)$
where $\neg \Phi$ and $\Phi^*$ are as in Section~\ref{sec:VC duality}. 

\medskip

Let $\varphi(x;y)$ be a partitioned $\mathcal L$-formula, where $\abs{x}=m$, $\abs{y}=n$. We may associate a bigraph $G_\varphi=(X,Y,\Phi)$ to $\varphi$ and $\mathbf M$, where 
$X=M^{m}$, $Y=M^{n}$, and
$$\Phi = \varphi(M^m;M^n)=\big\{(a,b)\in M^{m}\times M^{n} : \mathbf M\models\varphi(a;b) \big\}.
$$
Note that $G_{\neg\varphi}=\neg G_\varphi$ and $G_{\varphi^*}=(G_\varphi)^*$.
If we want to stress the dependence of $G_\varphi$ on $\mathbf M$, then we write $G^{\mathbf M}_\varphi$ instead of $G_\varphi$. If $\varphi$ is invariant under the extension  $\mathbf M\subseteq\mathbf N$ of $\mathcal L$-structures, then $G^{\mathbf N}_\varphi\restrict V=G^{\mathbf M}_\varphi$ where $V=V(G^{\mathbf M}_\varphi)$.

\medskip

From now on until the end of this subsection we assume that $M$ is infinite and $m=n$. 
The collection 
$$E(G_\varphi)=\big\{ (a,b) : (a,b)\in\varphi(M^m;M^m) \big\}\subseteq M^m\times M^m$$
of edges of $G_\varphi$ then maps naturally onto the definable family $$\mathcal S_{\widehat{\varphi}} = \big\{ \{a,b\} : (a,b)\in\varphi(M^m;M^m)\big\}\subseteq {M^m\choose \leq 2}$$
of subsets of $M^{m}$ by a map whose fibers have at most $2$ elements;
here $\widehat\varphi(v;x,y)$ is the partitioned $\mathcal L$-formula with object variables $v=(v_1,\dots,v_m)$ and parameter variables $(x,y)$  given by
$$\widehat{\varphi}(v; x,y) := 
\varphi(x;y) \wedge ( v=x \vee v=y ).$$
Note that $\VC(\widehat\varphi)\leq 2$. Also, $\mathcal S_{\widehat{\varphi^*}}=\mathcal S_{\widehat\varphi}$ and hence 
$\widehat{\varphi^*}$ and $\widehat\varphi$ have the same VC dimension and VC density.
A bound on the number of subsets of a given finite set which are cut out  by $\mathcal S_{\widehat{\varphi}}$ may be computed as follows:

\begin{lemma}\label{lem:upper bound}
Let $A\subseteq M^{m}$ be finite. 
Then
$$
\abs{A_0}+\textstyle\frac{1}{2}\abs{E(G_\varphi\restrict V)}\leq
\abs{A\cap \mathcal S_{\widehat{\varphi}}} \leq 1+\abs{A_0}+
\abs{E(G_\varphi\restrict V)}$$
where 
\begin{enumerate}
\item $A_0$ is the set of all $a\in A$ such that $\mathbf M\models\varphi(a;b)$ or $\mathbf M\models\varphi(b;a)$ for some $b\in M^m$, but there is no $b\in A$ with $\mathbf M\models\varphi(a;b)$ or $\mathbf M\models\varphi(b;a)$, and
\item
$V\subseteq V(G_\varphi)$ is the disjoint union of $A$ considered as a subset of $X$ and $A$ considered as a subset of $Y$.
\end{enumerate}
\end{lemma}
\begin{proof}
Each set $S\in A\cap \mathcal S_{\widehat{\varphi}}$ is of one of the following types: $S=\emptyset$; $S=\{a\}$ where $a\in A_0$; or $S=\{a,b\}$ where $a,b\in A$ with $\mathbf M\models\varphi(a;b)$ or $\mathbf M\models\varphi(b;a)$. Each set of the last two types actually occurs in $A\cap \mathcal S_{\widehat{\varphi}}$, whereas $S=\emptyset$ only occurs iff there is some edge $(a,b)$ of $G_\varphi$ with $a,b\notin A$.
\end{proof}

Hence if we set
$$\Pi_\varphi(t) := \max\big\{ \abs{E(G_{\varphi}\restrict V)} : V\subseteq V(G_\varphi),\ \abs{V}=t \big\}\in\bN,$$
then the lemma shows that
\begin{equation}\label{eq:pi Pi inequality}
\textstyle\frac{1}{2}\Pi_\varphi(t)\leq\pi_{\widehat\varphi}(t)\leq 1+t+\Pi_\varphi(2t)\qquad\text{for every $t$.}
\end{equation}
This observation opens up a road to computing (upper or lower) bounds on the VC density of the formula $\widehat\varphi$: {\it find a bound on the number of edges of the subgraph of $G_\varphi$ induced on finite subsets of its vertex set, in terms of the number of vertices.}\/ In the following we give some applications of this approach.

\medskip

For positive integers $r$ and $s$ we denote by $K_{r,s}:=\big([r],[s],[r]\times[s]\big)$ the complete bigraph with the vertex set $[r]\cup [s]$.
The following is a fundamental fact about finite bigraphs:

\begin{theorem}[K\H{o}v\'ari, S\'os and Tur\'an \cite{KST}]\label{thm:KST}
Let $r\leq s$ be positive integers.
There exists a real number $C=C(r,s)$ such that every finite bigraph $G$ which does not contain $K_{r,s}$ as a sub-bigraph 
has at most $C\,\abs{V(G)}^{2-1/r}$ edges.
\end{theorem}

(In fact, a more precise bound is also available, in terms of the sizes of the vertex sets $X$ and $Y$, but we won't need this.)

\begin{corollary}\label{cor:KST}
Let $r\leq s$ be positive integers.
There is a real number $C_1=C_1(r,s)$ with the following property:
if $\varphi(x;y)$ is an $\mathcal L$-formula such that $G_\varphi$ does not contain $K_{r,s}$ as a subgraph, then
$\pi_{\widehat{\varphi}}(t)\leq C_1\,t^{2-1/r}$ for every $t$; in particular,
$\vc(\widehat{\varphi})\leq 2-\frac{1}{r}$.
\end{corollary}
\begin{proof}
If $V\subseteq V(G_\varphi)$ is finite, and the  bigraph $G_\varphi\restrict V$ does not contain $K_{r,s}$,
then $\abs{E(G_\varphi\restrict V)} \leq C\,\abs{V}^{2-1/r}$  by Theorem~\ref{thm:KST}, where $C=C(r,s)>0$ is as in that theorem. Thus
$\pi_{\widehat\varphi}(t)\leq  1+t+\Pi_\varphi(2t)\leq 2(1+2^{1-1/r}C)\,t^{2-1/r}$ by \eqref{eq:pi Pi inequality}.
\end{proof}

Given integers $r,s\geq 1$, the bigraph $G_\varphi$ contains $K_{r,s}$ if and only if there are pairwise distinct $a_1,\dots,a_r\in M^m$ and pairwise distinct $b_1,\dots,b_s\in M^m$ such that $\mathbf M\models\varphi(a_i;b_j)$ for all $i\in [r]$, $j\in [s]$.
It is interesting to note that if $G_\varphi$ does not contain $K_{r,s}$ as a sub-bigraph, for some $r,s\geq 1$, then the  bigraph $G_{\neg\varphi}$ associated to $\neg\varphi$ does contain $K_{t,t}$, for every $t\geq 1$:
by an analogue of Ramsey's Theorem for bigraphs due to Erd\H{o}s and Rado \cite{Erdos-Rado}, for every $t$ there exists an $n$ such that for all bigraphs $G$ with $\abs{V(G)}\geq n$, one of $G$, $\neg G$ contains $K_{t,t}$ as a sub-bigraph.
Hence in this case the VC~density of the formula $\widehat{\neg\varphi}$ associated to $\neg\varphi$ equals $2$, by \eqref{eq:pi Pi inequality}.

\subsection{Point-line incidences}\label{sec:point-line}
Let $K$ be an infinite field, construed as a first-order structure in the language of rings as usual.
The partitioned formula
$$\varphi(x_1,x_2; y_1, y_2) :=  x_2=y_1x_1 + y_2
$$
gives rise to the bigraph
$G_\varphi=(X,Y,\Phi)$ where $X=Y=K^2$ and
$$\Phi = \big\{ ((\eta,\xi),(a,b))\in K^2\times K^2 : \eta=a\xi+b \big\}.$$
We may think of $V(G_\varphi)=X\cup Y$ as the disjoint union of the set $X$ of points $p=(\eta,\xi)\in K^2$ in the affine plane $\mathbb A^2(K)$ over $K$ and the set $Y$ of non-vertical lines $\ell$ in $\mathbb A^2(K)$; thus $E(G_\varphi)$ is the set of point-line incidences $(p,\ell)$ where $p\in\mathbb A^2(K)$ and $\ell\subseteq\mathbb A^2(K)$ is a non-vertical line containing $p$.
The bigraph $G_\varphi$ does not contain $K_{2,2}$ as a subgraph. (Two distinct points in $\mathbb A^2(K)$ lie on a unique line.)
Hence by Corollary~\ref{cor:KST}:

\begin{corollary} \label{cor:point-line incidences}
There is a real number $C_1>0$ \textup{(}independent of $K$\textup{)} such that
$\pi_{\widehat{\varphi}}(t)\leq C_1\,t^{3/2}$ for every $t$; in particular,
$\vc(\widehat{\varphi})\leq \frac{3}{2}$.
\end{corollary}

Note that this bound is better than what we get from the general estimate $\vc\leq\VC$, since $\VC(\widehat\varphi)=2$.
Also, if $K=\bR$, then for our original formula $\varphi$ we have
$\pi_\varphi(t)=1+t+{t\choose 2}$ for every $t$.
In particular $\VC(\varphi)=\vc(\varphi)=2$.

\medskip

A lower bound on $\vc(\widehat\varphi)$ is given by:

\begin{lemma}\label{lem:lower bound}
Suppose $K$ has characteristic $0$. Then
$\vc(\widehat\varphi)\geq\frac{4}{3}$.
\end{lemma}
\begin{proof}
This is due to Erd\H{o}s, with the following simpler argument by Elekes \cite{Elekes}: let $k$ be a positive integer, $t=4k^3$, and consider the subsets
\begin{align*}
P 	&:= \big\{(\eta,\xi):\eta=0,1,\dots,k-1,\xi=0,1,\dots,4k^2-1\big\}\\
L	&:= \big\{(a,b)     :a=0,1,\dots,2k-1, b=0,1,\dots,2k^2-1\big\}
\end{align*}
of $\mathbb Z^2$, and set $V:=P\cup L\subseteq V(G_\varphi)$.
Then for each $i=0,1,\dots,k-1$,
each line $\eta=a\xi+b$ with $(a,b)\in L$ contains a point $(\eta,\xi)\in P$ with $\xi=i$, so 
$$\abs{E(G_\varphi\restrict V)}\geq k\cdot\abs{L}=4k^4=\textstyle\frac{1}{4^{1/3}}t^{4/3}=\frac{1}{4}\abs{V}^{4/3}$$
and hence $\vc(\widehat\varphi)\geq\frac{4}{3}$ by~\eqref{eq:pi Pi inequality}.
\end{proof}

The precise value of $\vc(\widehat\varphi)$ depends on the characteristic of $K$:

\begin{proposition}\label{prop:fractional vc}
\mbox{}
\begin{enumerate}
\item Suppose $K$ has characteristic $0$. Then $\vc(\widehat\varphi)=\frac{4}{3}$.
\item Suppose $K$ has positive characteristic. Then $\vc(\widehat\varphi)=\frac{3}{2}$.
\end{enumerate}
\end{proposition}

In the proof of this proposition  we use the following  generalization of a famous theorem of Sz\'emeredi and Trotter \cite{ST} (although a weaker version of this theorem from  \cite{SolTao}, with a somewhat simpler proof, would also suffice for our purposes):

\begin{theorem}[T\'oth \cite{Toth}]\label{thm:Toth}
There exists a real number $C$ such that for all $m,n>0$ there are at most
$C(m^{2/3}n^{2/3}+m+n)$ incidences among $m$ points and $n$ lines in the affine plane over~$\bC$.
\end{theorem}

\begin{proof}[Proof of Proposition~\ref{prop:fractional vc}]
The lower bound  $\vc(\widehat\varphi)\geq \frac{4}{3}$ in (1) was shown in the previous lemma. From  Theorem~\ref{thm:Toth} and Lemma~\ref{lem:upper bound}  we  obtain $\vc^{\mathbb C}(\widehat\varphi)\leq \frac{4}{3}$.
If $K$ is any field of characteristic $0$ with algebraic closure $K^{\operatorname{alg}}$, then $\pi_{\widehat\varphi}^{K}\leq
\pi_{\widehat\varphi}^{K^{\operatorname{alg}}}=\pi_{\widehat\varphi}^{\mathbb C}$ by Lemmas~\ref{lem:invariant under extensions} and \ref{lem:shatter function and elementary equivalence}, showing part (1) of Proposition~\ref{prop:fractional vc}.

The upper bound  $\vc(\widehat\varphi)\leq \frac{3}{2}$ in (2) 
is a consequence of Corollary~\ref{cor:point-line incidences}. For the lower bound we use  the following observation:
if $F$ is a finite subfield of $K$, say $\abs{F}=q$, then 
$\abs{V(G^F_{\varphi})}=2q^2$ and $\abs{E(G^F_{\varphi})}=q^3$, hence
$$\abs{E(G^K_\varphi\restrict V)}=\abs{E(G^F_\varphi)} = \frac{1}{\sqrt{8}} \abs{V}^{3/2}\qquad\text{where $V=V(G^F_\varphi)$.}$$
Together with \eqref{eq:pi Pi inequality} this yields the inequality $\vc(\widehat\varphi)\geq \frac{3}{2}$ in (2). 
\end{proof}

Proposition~\ref{prop:fractional vc} shows in particular that there is no hope for a ``\L{}os~Theorem'' for VC~density: if $\mathbf M$ is a non-principal ultraproduct of a family $(\mathbf M_i)_{i\in I}$ of infinite $\mathcal L$-structures, then one may have $\vc^{\mathbf M}(\varphi)\neq\vc^{\mathbf M_i}(\varphi)$ for all $i\in I$.

\medskip

It is interesting to contrast Proposition~\ref{prop:fractional vc} with the outcome of only considering parameters from an indiscernible sequence:

\begin{lemma}\label{lem:alt for SzTr}
The formula $\widehat\varphi$ has alternation number $2$,  hence
$\vc_\ind(\widehat\varphi)=1$.
\end{lemma}
\begin{proof}
It suffices to show $\alt(\widehat\varphi)=2$, since then Lemma~\ref{lem:alt} yields $\vc_\ind(\widehat\varphi)=1$.
Suppose for a contradiction that $(a_i)_{i\in\bN}$ is an indiscernible sequence in $K^2$ and $b=(p,\ell)\in K^2\times K^2$ witnessing that $\alt(\widehat\varphi)\geq 3$. 
We think of the elements  of $K^2$ both as points in the affine space $\mathbb A^2(K)$ over $K$ and as non-vertical lines in $\mathbb A^2(K)$, and let $i$, $j$ range over $\{0,1,2,3\}$.
The $a_i$ are pairwise distinct, $p\in\ell$, and $a_i=p$, $a_j=\ell$ for some $i\neq j$; hence $a_i\in a_j$ for some $i\neq j$. If $a_i\in a_j$ where $i<j$, then $a_i\in a_j$ for \emph{all} $i<j$ (by indiscernibility) and hence $a_0,a_1\in a_2\cap a_3$, and this forces $a_0=a_1$ or $a_2=a_3$, in both cases a contradiction. 
Similarly the assumption that $a_i\in a_j$ with $i>j$ leads to a contradiction.
\end{proof}

Many other results in the combinatorial literature lead to non-trivial (upper and lower) bounds on $\vc(\widehat\varphi)$  if $\varphi$ encodes the incidence of points on various geometric objects; see \cite[Chapter~4]{Matousek-Book} or \cite{PS-Survey}. For example,
let $\mathbf R=(\mathbb R,0,1,{+},{-},{\times},{<})$ be  the ordered field of real numbers. 
Let 
$$\varphi(x_1,x_2;y_1,y_2) := (x_1-y_1)^2+(x_2-y_2)^2=1,$$ 
so $\mathcal S^{\mathbf R}_\varphi$ is the collection of circles with radius $1$ in the plane, and $E(G^{\mathbf R}_\varphi)$ is the set of incidences between points in $\bR^2$ and circles of radius $1$. Then an analogue of the Sz\'emeredi-Trotter Theorem \cite{SST} (or a more general result due to Pach and Sharir \cite{PS} on families of simple plane curves) and \eqref{eq:pi Pi inequality} yields $\vc^{\mathbf R}(\widehat\varphi)\leq\frac{4}{3}$. (However, it is unknown whether this bound is sharp, cf.~\cite[Section~2]{PS-Survey}.) 

\subsection{Irrational VC~density}
In \cite{Assouad-2} it is shown 
that for every real number $r\geq 1$ there exists a set system $\mathcal S\subseteq {\bN\choose \lceil r\rceil}$ with $\vc(\mathcal S)=r$.
We do not know the answer to the following question (though we suspect the answer to be negative):

\begin{question}
Is the  VC density of a formula in a NIP theory always rational?
\end{question}

Let $\mathcal L_{\operatorname{gr}}=\{E\}$ be the language with a single binary relation symbol $E$. The $\mathcal L_{\operatorname{gr}}$-structures are nothing but the (directed) graphs (with $E$ interpreted as the edge relation). Given a graph $G$ we denote by $V(G)$ its set of vertices and by $E(G)$ its set of edges.
Spencer and Shelah \cite{SS} established a $0$-$1$-law for $\mathcal L_{\operatorname{gr}}$-sentences about random (symmetric, loopless) graphs with $n$ vertices and edge probability $n^{-\alpha}$, where $\alpha$ is an irrational number between $0$ and $1$. We denote the resulting complete  $\mathcal L_{\operatorname{gr}}$-theory by $T_\alpha$.
It was shown by Baldwin and Shelah \cite{BS} that $T_\alpha$ is stable. (This can also be checked by simply verifying that $T_\alpha$ is superflat in the sense of \cite{PZ}; cf.~Section~\ref{sec:not growing like a power} below.) In particular, $T_\alpha$ is NIP, so it makes sense to investigate VC~density of formulas in $T_\alpha$. We consider the following $\mathcal L_{\operatorname{gr}}$-formula:
$$\varphi(x;y) = E(x,y) \vee x=y.$$
For any graph $G$ and vertex $v$ of $G$, the formula $\varphi(x;v)$ defines the (closed) neighborhood of $v$, i.e., the set consisting of $v$ together with all vertices adjacent to it. It is tempting to guess that $\vc(\varphi)=1/\alpha$. (This would give rise to a negative answer of the question posed above.) However, it turns out that $\vc(\varphi$) is an integer:

\begin{lemma}\label{lem:vc random graphs}
$\vc(\varphi)=\lfloor 1/\alpha \rfloor$.
\end{lemma}

Before we give the proof, we recall some basic facts about the theory $T_\alpha$; our main reference is \cite{Spencer}. We let $G$ be a model of $T_\alpha$.

A \emph{rooted graph} is a pair $(R,H)$ where $H$ is a finite graph and $R$ a proper subset of its set of vertices; the elements of $R$ will be called \emph{roots.} We consider each finite non-empty graph $H$ as a rooted graph by identifying it with $(\emptyset, H)$.
Given a rooted graph $(R,H)$, a rooted graph $(R,H')$, where $H'$ is subgraph of $H$ whose vertex set properly contains~$R$, is called a rooted subgraph of $(R,H)$.

A \emph{weak embedding} of a rooted graph $(R,H)$ into $G$ is an injective map $\iota\colon V(H)\to V(G)$ such that for all  roots $v$ and non-roots $w$ of $(R,H)$, $v$ and $w$ are adjacent in $H$ iff $\iota(v)$ and $\iota(w)$ are adjacent in $G$; such a weak embedding is called an \emph{embedding} if also 
 any two non-roots $v$ and $w$ of $(R,H)$ are adjacent in $H$ iff $\iota(v)$ and $\iota(w)$ are adjacent in $G$.
Note that there is no requirement about edges between roots. (This terminology does not appear in \cite{Spencer} which talks about ``$(R,H)$-extensions'' instead.)

Let $(R,H)$ be a rooted graph. The \emph{average degree} of $(R,H)$ is $\operatorname{adeg}(R,H):=2e/v$ where  $v=v(R,H)>0$ is the number of vertices of $H$ which are not roots and $e=e(R,H)$ is the number of edges of $H$ which do not have both ends in $R$. The \emph{maximum average degree $\operatorname{mdeg}(R,H)$} of $(R,H)$ is defined as the maximum of $\operatorname{adeg}(R,H')$ where $(R,H')$ is a rooted subgraph of $(R,H)$. If $\operatorname{adeg}(R,H)>2/\alpha$ then $(R,H)$ is called \emph{dense,} and \emph{sparse} otherwise (i.e., if $\operatorname{adeg}(R,H)<2/\alpha$). If
$\operatorname{mdeg}(R,H)<2/\alpha$ then $(R,H)$ is called \emph{safe,} and \emph{unsafe} otherwise.

Now if $H$ is dense then $G$ does not contain a copy of $H$, whereas
if $H$ is safe then $G$ contains a copy (indeed, an induced copy) of $H$. 
More generally, if $(R,H)$ is unsafe then there is no weak embedding of $(R,H)$ into $G$ \cite[p.~69]{Spencer}, and if
$(R,H)$ is safe then every injective map $R\to V(G)$ extends to an embedding of $(R,H)$ into $G$ \cite[Theorem~5.2.1]{Spencer}.

Let now $\mathcal S$ be a non-empty set system on $[t]=\{1,\dots,t\}$, where $t>0$. We associate a rooted graph $(R,H)=(R_{\mathcal S},H_{\mathcal S})$ to $\mathcal S$ as follows: the set of vertices of $H$ is the disjoint union of $[t]$ and $\mathcal S$,   the set of roots is $R=[t]$, there are no edges between two roots and no edges between two non-roots, and a root $i\in [t]$ and a non-root $S\in\mathcal S$ are related by an edge iff $i\in S$. (Cf.~Figure~\ref{fig:rooted graph}.) Note that this rooted graph has average degree $\operatorname{adeg}(R,H)=\frac{2}{\abs{\mathcal S}}\sum_{S\in\mathcal S}\abs{S}$ and maximum average degree
$$\operatorname{mdeg}(R,H)=\max_{\emptyset\neq\mathcal S'\subseteq\mathcal S} \frac{2}{\abs{\mathcal S'}}\sum_{S\in\mathcal S'}\abs{S}.$$
So if $\mathcal S\subseteq { [t] \choose k }$ where $k\in \{0,\dots,t\}$ then
$\operatorname{adeg}(R,H)=\operatorname{mdeg}(R,H)=2k$; hence if in addition $k<1/\alpha$ then $(R,H)$ is safe (so there exists an embedding of $(R,H)$ into $G$) whereas if $k>1/\alpha$ then $(R,H)$ is dense (and so there is no weak embedding of $(R,H)$ into~$G$).

\begin{figure}
\vskip1em

\begin{tikzpicture}[scale=1.3, back line/.style={solid}, cross line/.style={preaction={draw=white, -, line width=0.3em}}]

\path (7,1) node {$R=[t]$};

\path (7.5,2.7) node {$\mathcal S=\{S,S',\dots\}$};

\path (1.5, 2.8) node (bX) {$S$};
\path (3.5, 2.8) node (bX) {$S'$};
\path (5.5, 2.7) node (bX) {$\dots$};


\draw (1,1) edge [back line] (3.5,2.5);
\draw (3,1) edge [back line] (3.5,2.5);
\draw (5,1) edge [back line] (3.5,2.5);

\draw (1,1) edge (1.5, 2.5);
\draw (3,1) edge [cross line] (1.5, 2.5);
\draw (4,1) edge [cross line] (1.5, 2.5);

\fill [black] (1,1) circle (2pt);
\fill [black] (2,1) circle (2pt);
\fill [black] (3,1) circle (2pt);
\fill [black] (4,1) circle (2pt);
\fill [black] (5,1) circle (2pt);

\fill [black] (1.5,2.5) circle (2pt);
\fill [black] (3.5,2.5) circle (2pt);

\end{tikzpicture}
\caption{The rooted graph associated to a set system}
 \label{fig:rooted graph}

\end{figure}
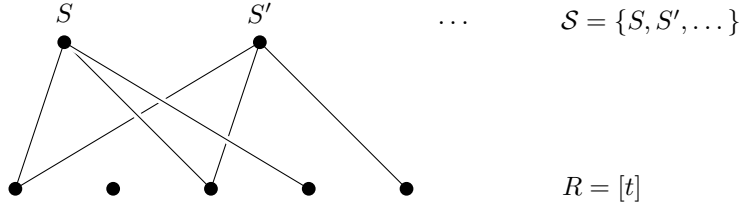

\begin{proof}[Proof of Lemma~\ref{lem:vc random graphs}]
Applying the remarks above to $\mathcal S={ [t] \choose \lfloor 1/\alpha \rfloor }$, where $t>0$, we see that, as $ \lfloor 1/\alpha \rfloor<1/\alpha$,   there are pairwise distinct vertices $a_1,\dots,a_t$ and $b_S$ ($S\in\mathcal S$) of $G$ such that 
$a_i$ and $b_S$ are adjacent iff $i\in S$, i.e., writing $A=\{a_1,\dots,a_t\}$ we have
$A\cap\varphi(G;b_S)=\{a_i:i\in S\}$ and hence $\abs{A\cap\mathcal S_\varphi}\geq\abs{\mathcal S}$. 
So $\pi_\varphi(t)\geq {t\choose \lfloor 1/\alpha \rfloor }$ for each~$t$, 
therefore $\vc(\varphi)\geq \lfloor 1/\alpha \rfloor$.

To show the reverse inequality suppose for a contradiction that $\vc(\varphi) > \lfloor 1/\alpha \rfloor$. Let $\varrho$ be a real number with $\vc(\varphi)>\varrho>\lfloor 1/\alpha \rfloor$.
Note that for every set $A$ of vertices of $G$ the set system $A\cap \mathcal S_\varphi$ is the union of the set system
\begin{equation}\label{eq:system 1}
\big\{ \{b\} \cup \{a\in A:(a,b)\in E(G)\} : b\in A \big\}
\end{equation}
consisting of at most $\abs{A}$ sets, and
\begin{equation}\label{eq:system 2}
\big\{ \{a\in A:(a,b)\in E(G)\} : b \in V(G)\setminus A \big\}.
\end{equation}
Since $\vc(\varphi)>\varrho$,  for every $C\geq 1$ there are arbitrarily large $t>0$ and $A\subseteq V(G)$ with $\abs{A}=t$ such that $\abs{A\cap\mathcal S_\varphi}\geq 2Ct^\varrho$.  The set system \eqref{eq:system 1} has at most $t$ elements; hence  \eqref{eq:system 2} contains at least $2Ct^\varrho-t$ sets and so, since $\varrho>\lfloor 1/\alpha \rfloor\geq 1$, contains at least $Ct^\varrho$ sets.
Identifying $A$ with $[t]$, the set system \eqref{eq:system 2} on $A$ thus gives rise to a set system $\mathcal S$ on $[t]$ with $\abs{\mathcal S}\geq Ct^\varrho$ whose associated rooted graph $(R_{\mathcal S},H_{\mathcal S})$ weakly embeds into $G$.

On the other hand,
let $\mathcal S$ be any set system on $[t]$ such that $(R_{\mathcal S},H_{\mathcal S})$ weakly embeds into $G$, and
for $k=0,\dots,t$ consider the set system $\mathcal S_k:=\mathcal S\cap {[t]\choose k}$ on $[t]$. If $\mathcal S_k\neq\emptyset$ then $(R_{\mathcal S_k},H_{\mathcal S_k})$ is a rooted subgraph of $(R_{\mathcal S},H_{\mathcal S})$ and hence $$2k=\operatorname{adeg}(R_{\mathcal S_k},H_{\mathcal S_k})\leq\operatorname{mdeg}(R_{\mathcal S},H_{\mathcal S})<2/\alpha.$$
Therefore $\mathcal S$ does not contain a $k$-element subset of $[t]$ with $k>1/\alpha$, i.e.,
$\mathcal S\subseteq {[t]\choose  \leq \lfloor 1/\alpha \rfloor}$. Hence $\abs{\mathcal S}\leq Ct^{\lfloor 1/\alpha \rfloor}$ where $C=C_\alpha$ is a constant only depending on $\alpha$. This contradicts the previous paragraph.
\end{proof}

We remark that a similar analysis shows that the simpler formula $E(x,y)$ also has VC~density $\lfloor 1/\alpha \rfloor$ in $T_\alpha$. We chose $\varphi$ as above because it allows us to compare Lemma~\ref{lem:vc random graphs} with the main result of \cite{ABC}, where the precise value of the VC~dimension of $\varphi$ (as it depends on $\alpha$) is computed.
In particular, \cite[Corollary~8]{ABC} shows that if $0<\alpha<\frac{93}{650}$~then 
$$\lfloor 1/\alpha \rfloor+3 \leq \VC(\varphi) \leq \lfloor 1/\alpha +3 (\alpha+1)\rfloor.$$

\subsection{Shatter functions not growing like a power}\label{sec:not growing like a power}
We finish this section with two examples of VC~classes definable in NIP theories
whose shatter function is not asymptotic to a real power function. 

\subsubsection{The hypercube}\label{sec:hypercube}

Let $Q$ be the ``infinitary hypercube'', i.e., the (symmetric, loopless) graph whose vertex set is the set of all sequences $s=(s_n)$ in $\{0,1\}^{\mathbb N}$ with finite support, with two sequences related by an edge iff they differ in only one component. (Alternatively, $Q$ can be represented as the set of all finite sets of natural numbers, with an edge between them iff their symmetric difference is a singleton.) Note that $Q$ is the increasing union $Q=\bigcup_{d>0} Q_d$ of its induced subgraphs $Q_d$ with vertex set 
$$V(Q_d)=\big\{s=(s_n)\in Q: \text{$s_n=0$ for $n\geq d$}\big\},$$
which we may identify with $\{0,1\}^d$ in the natural way.
(So $Q_d$ is the $d$-dimensional hypercube.) We construe $Q$ as an $\mathcal L_{\operatorname{gr}}$-structure.  
In \cite{PZ} a  condition sufficient for a (symmetric) graph to be stable is introduced, called superflatness: a graph $G$ is superflat if for every $m$ there is some $n$ such that no subdivision of the complete graph $K_n$ on $n$ vertices, obtained by placing at most $m$ additional vertices on each edge, embeds into~$G$.
Note that the graph $Q$ is not superflat: in fact, for every $d$ there is an embedding of a subdivision of $K_{d+1}$, obtained by placing at most one additional vertex on each edge, into $Q_d$, cf.~\cite{Hartman}.
However, we do have:

\begin{proposition}\label{prop:Q}
$Q$ is $\omega$-stable.
\end{proposition}

Towards a proof of this proposition, we first introduce some notation and terminology:
Given a language $\mathcal L$ let $\mathcal L(P)=\mathcal L\cup\{P\}$ where $P$ is a new predicate symbol, and given an $\mathcal L$-structure $\mathbf M$ and a subset $A$ of its domain let $(\mathbf M,A)$ be the expansion of $\mathbf M$ to an $\mathcal L(P)$-structure obtained by interpreting $P$ by $A$. The \emph{induced} structure $A_{\operatorname{ind}}$ on $A$ is the structure whose language consists of an $m$-ary relation symbol $R_\varphi$ for every $\mathcal L$-formula $\varphi(x)$, where $m=\abs{x}$, interpreted in $A_{\operatorname{ind}}$ by $\varphi^{\mathbf M}(M^m)\cap A^m$.
An $\mathcal L(P)$-formula $\varphi(x)$, where $x=(x_1,\dots,x_m)$, is said to be \emph{bounded} if it has the following form (slightly abusing syntax):
$$\varphi(x) = \Diamond_1 y_1 \in P \cdots \Diamond_n y_n \in P\, \psi(x_1,\dots,x_m,y_1,\dots,y_n),$$
where $\Diamond_i\in \{\forall,\exists\}$ and $\psi$ is an $\mathcal L$-formula.
Casanovas and Ziegler have shown:

\begin{theorem}
Let $\mathbf M$ be a structure in a language $\mathcal L$ and let $A\subseteq M$.
\begin{enumerate}
\item Suppose $\mathbf M$ is strongly minimal. Then  in $(\mathbf M,A)$, every $\mathcal L(P)$-formula is equivalent to a bounded formula.
\item Suppose that in $(\mathbf M,A)$, every $\mathcal L(P)$-formula is equivalent to a bounded formula, and let $\lambda\geq\abs{\mathcal L}$ be a cardinal. If both $\mathbf M$ and $A_{\operatorname{ind}}$ are $\lambda$-stable then $(\mathbf M,A)$ is $\lambda$-stable.
\end{enumerate}
\end{theorem}

(See \cite[Corollary~5.4 and Proposition~3.1]{CZ}; part (1) had actually first been shown by Pillay~\cite{Pillay}.)

\medskip

Consider now, slightly more general than necessary, an arbitrary field $K$, and let $\mathcal L_K=\{ {0},{+},(\lambda\cdot\,)_{\lambda\in K}\}$ be the language of $K$-vector spaces. 
Let $M$ be an infinite-dimensional $K$-vector space. Then $M$, construed as an $\mathcal L_K$-structure, has quantifier elimination and is $\omega$-stable. Let $A$ be a set of linearly independent elements of $M$.
Then the induced structure on $A$ is trivial: every subset of $A^m$ definable in $A_{\operatorname{ind}}$ is definable in the empty language. Hence by the theorem above,
the $\mathcal L_K(P)$-structure $(M,A)$ is $\omega$-stable. 

\medskip

For the proof of Proposition~\ref{prop:Q}, it now suffices to note that the $\mathcal L_{\operatorname{gr}}$-structure $Q$ is definable in $(M,A)$, for suitable choice of $K$, $M$ and $A$: Take $K=\mathbb F_2$ and let $M=\bigoplus_n \mathbb F_2\, a_n$ be a countably infinite $\mathbb F_2$-vector space with distinguished basis $A=\{a_n:n\geq 0\}$. Then  $Q$ is definable in $(M,A)$: identifying $V(Q)$ with $M$ in the natural way, we have, for all vertices $s$, $t$ of $Q$: $(s,t)\in E(Q)$ iff $s-t\in A$. 
Since $(M,A)$ is $\omega$-stable, so is $Q$. \qed

\medskip

Now
consider the $\mathcal L_{\operatorname{gr}}$-formula $\varphi(x;y):=E(x,y)$. Then
$$\mathcal S_{\widehat\varphi} = \big\{ \{s,s'\}:s,s'\in Q,\ (s,s')\in E(Q)\big\}$$
is the collection of undirected edges of $Q$. In the following $A$ denotes a subset of $Q$ (unlike in the proof of Proposition~\ref{prop:Q}).
Note that if $A\subseteq Q_d$ then
$$A\cap \mathcal S_{\widehat\varphi} = \textstyle {A\choose 1} \cup E[A]\quad\text{where }
E[A]:= \big\{ \{a,a'\}:a,a'\in A,\ (a,a')\in E(Q)\big\}.$$
Since $Q_d$ has $2^d$ vertices and $\frac{1}{2}d2^d$ undirected edges, we thus see that
$\pi_{\widehat\varphi}(t)\geq t+\frac{1}{2}t\log t$ for infinitely many $t$; in fact:

\begin{proposition}
$\pi_{\widehat\varphi}(t)=\frac{1}{2}t\log t\,(1+o(1))$ as $t\to\infty$.
\end{proposition}
\begin{proof}
Set
$$E_d(t) := \max\big\{\abs{E[A]}:A\subseteq Q_d,\ \abs{A}=t\big\} \qquad\text{for $d>0$ and $t\leq 2^d$.}$$
Then $\pi_{\widehat\varphi}(t)=t+\max_{d\geq \lceil\log t\rceil} E_d(t)$.
It is known (see, e.g., \cite{A-G}) that there is some function $g$ with $g(t)=\frac{1}{2}t\log t\,(1+o(1))$ as $t\to\infty$ such that $E_d(t)=g(t)$ for all $d$ and $t\leq 2^d$. This yields the claim.
\end{proof}

\subsubsection{An example in $\mathbf R=(\mathbb R,0,1,{+},{-},{\times},{<})$}
For this we use another one of the rare examples (besides the Sz\'eme\-re\-di-Trotter Theorem) where tight bounds on the number of incidences are known:

\begin{theorem}[Pach and Sharir \cite{PS-angles}] \label{thm:PS}
Let $\alpha$ be a real number with $0<\alpha<\pi$. The maximum number of times that $\alpha$ occurs as an angle among the ordered triples of $t$ points in the plane is $O(t^2\log t)$. Furthermore, suppose $\tan(\alpha)\in\bQ\sqrt{d}$ where $d\in\bN$ is not a square. Then there exists a constant $C=C_\alpha>0$ and, for every $t>3$, a $t$-element set $S_t\subseteq\bR^2$ with the property that at least $Ct^2\log t$ ordered triples of points from $S_t$ determine the angle $\alpha$. 
\end{theorem}

Let $x=(x_1,x_2)$, $y=(y_1,y_2)$, $z=(z_1,z_2)$ and consider the formula
$$\varphi(x,y,z) := x\neq y \wedge x\neq z \wedge
2\langle y-x,z-x\rangle = ||y-x||\, ||z-x||$$
in the language of the ordered field of real numbers $\mathbf R$, where $\langle\ ,\ \rangle$ denotes the usual inner product on $\bR^2$ and $||\ ||$ the associated norm. 
Then for  $a,b,c\in\bR^2$ we have $\mathbf R\models\varphi(a,b,c)$ iff the vectors $b-a$ and $c-a$ are non-zero and the angle $\angle(b,a,c)$ between them is $\frac{\pi}{3}$.
Let
$\widehat\varphi(v;x,y,z)$ be the partitioned formula with object variables $v=(v_1,v_2)$ and parameter variables $(x,y,z)$  given by
$$\widehat{\varphi}(v; x,y,z) := 
\varphi(x,y,z) \wedge ( v=x\vee v=y\vee v=z ),$$
so
$\mathcal S_{\widehat\varphi}$ consists of all $\{a,b,c\}\in{\bR^2\choose 3}$ with $\mathbf R\models\varphi(a,b,c)$.
We now have:

\begin{corollary}\label{cor:PS}
There exist constants $C_1,C_2>0$ such that 
$$C_1t^2\log t<\pi_{\widehat\varphi}(t)<C_2t^2\log t\qquad\text{for every $t>0$.}$$
That is,
$\pi_{\widehat\varphi}(t)=\Theta( t^2\log t )$ as $t\to\infty$.
\end{corollary}
\begin{proof}
Let $A\subseteq\bR^2$ be finite. Then
$$A\cap\mathcal S_{\widehat\varphi} = \textstyle {A\choose \leq 2}\cup 
\left\{ \{a,b,c\}\in {A\choose 3} : \mathbf R\models\varphi(a,b,c) \right\}.$$
Since $\tan(\pi/3)=\sqrt{3}$, the second part of Theorem~\ref{thm:PS} applies to $\alpha=\frac{\pi}{3}$.
The upper bound in Corollary~\ref{cor:PS} now follows from the first assertion in Theorem~\ref{thm:PS}, and the lower bound from the second assertion in the same theorem.
\end{proof}

\section{Theories with the $\VC{}d$ Property}\label{sec:VCm property}

\noindent
After defining the $\VC{}d$ property we prove that if a theory has this property then the dual VC density of any finite set of partitioned formulas in the tuple of object variables $x$ is at most $d\abs{x}$.
In Section~\ref{sec:examples of VCm theories} we will then show that various theories have the $\VC{}d$ property.
In the following $\mathbf M$ is a structure in a language $\mathcal L$, $\Delta(x;y)$ is a finite non-empty set of partitioned $\mathcal L$-formulas, and $m=\abs{x}$. 

\subsection{Uniform definability of types over finite sets}
Given a $\Delta(x;B)$-type $q\in S^\Delta(B)$, where $B\subseteq M^{\abs{y}}$, a family $\mathcal F=(\varphi_\#)_{\varphi\in\Delta}$ of $\mathcal L(M)$-formulas $\varphi_\#(y)$ is said to \emph{define $q$}\/ if for all $\varphi\in\Delta$ and $b\in B$ we have
$$\varphi(x;b)\in q \qquad\Longleftrightarrow\qquad \mathbf M\models \varphi_\#(b).$$
We also say that \emph{$\mathcal F$ is a definition of $q$.}\/
For a family $\mathcal F=\mathcal F(y;v)$ of partitioned $\mathcal L$-formulas $\psi(y;v)$, we denote by $\mathcal F(y;c)$ the family $(\psi(y;c))_{\psi\in\mathcal F}$ of $\mathcal L(M)$-formulas obtained by substituting a given tuple $c\in M^{\abs{v}}$ for the tuple of variables $v$.
The following generalizes a definition due to Guingona \cite{Guingona}:

\begin{definition}\label{def:UDTFS}
We say that $\Delta$ has \emph{uniform definability of types over finite sets}\/ (abbreviated as \emph{UDTFS}\/) in $\mathbf M$ if there are finitely many families 
$$\mathcal F_i=\big(\varphi_i(y;y_1,\ldots,y_d)\big)_{\ph\in\Delta}\qquad (i\in I)$$ 
of $\mathcal L$-formulas (where $\abs{y_j}=\abs{y}$ for every $j=1,\dots,d$) such that for every finite  set $B\subseteq M^{\abs{y}}$
and $q\in S^\Delta(B)$ there are $b_1,\ldots,b_d \in B$ and some $i\in I$ such that $\mathcal F_i(y;b_1,\dots,b_d)$
defines $q$. We call the family $\mathcal F=(\mathcal F_i)_{i\in I}$ a \emph{uniform definition of $\Delta(x;B)$-types over finite sets}\/ in $\mathbf M$ with $d$~parameters. If $\Delta=\{\varphi\}$ is a singleton, we also speak of $\varphi$ having UDTFS.
\end{definition}

The following observation shows in particular that every finite set $\Delta(x;y)$ of partitioned $\mathcal L$-formulas which is directed (see Example~\ref{ex:VC-minimal}) has UDTFS in $T$ with a single parameter:

\begin{lemma}\label{lem:breadth and UDTFS}
Let $\Delta(x;y)$ be a finite set of partitioned $\mathcal L$-formulas, and 
suppose the set system $\mathcal S_\Delta=\{\varphi(M^{\abs{x}};b):b\in M^{\abs{y}}\}$ has breadth~$d$.
Then $\Delta$ has UDTFS with $d$ parameters.
\end{lemma}
\begin{proof}
To see this set $\mathcal F_0(y):=(\exists z (z\neq z))_{\varphi\in\Delta}$ and,
for each $d$-tuple ${\mathbf\psi}=(\psi_1,\dots,\psi_d)\in\Delta^d$, define
$$\mathcal F_{\mathbf\psi}(y;y_1,\dots,y_d):=(\varphi_{\mathbf \psi}(y;y_1,\dots,y_d))_{\varphi\in\Delta}$$ where
$$\varphi_{\mathbf\psi}(y;y_1,\dots,y_d):=\forall x\big(\psi_1(x;y_1)\wedge\cdots\wedge\psi_d(x;y_d)\to\varphi(x;y)\big).$$
Then $\mathcal F=(\mathcal F_{\mathbf\psi})_{\mathbf\psi\in\{0\}\cup\Delta^d}$ is a uniform definition of $\Delta(x;B)$-types over finite sets. For suppose $q\in S^\Delta(B)$ where $B\subseteq M^{\abs{y}}$ is finite. If $\varphi(x;b)\notin q$ for all $\varphi\in\Delta$, $b\in B$, then $\mathcal F_0(y)$ defines $q$. Otherwise, by assumption 
we can pick $\psi_1(x;b_1),\dots,\psi_d(x;b_d)\in q$ such that 
$$\bigcap_{\varphi(x;b)\in q} \varphi(M^{\abs{x}};b) = \psi_1(M^{\abs{x}};b_1)\cap\cdots\cap \psi_d(M^{\abs{x}};b_d).$$
Then $\mathcal F_\psi(y;b_1,\dots,b_d)$, where ${\mathbf\psi}=(\psi_1,\dots,\psi_d)$, defines $q$.
\end{proof}

A uniform definition of $\Delta(x;B)$-types over finite sets in $\mathbf M$ remains a uniform definition of $\Delta(x;B)$-types over finite sets in any elementarily equivalent structure, and so it makes sense to speak of uniform definability of types over finite sets in a complete theory.
If we do not care about the number of parameters, we can always do with a single defining scheme (at least for non-trivial parameter sets):

\begin{lemma}
Let $\mathcal F=(\mathcal F_i)_{i\in I}$ be a uniform definition of $\Delta(x;B)$-types over finite sets with $d$ parameters as above, and let $n=\abs{I}$. Then there exists a family $$\mathcal F_\#=\big(\varphi_\#(y;y_1,\dots,y_d,v,w_1,\dots,w_n)\big)_{\varphi\in\Delta}$$ such that for every finite  set $B\subseteq M^{\abs{y}}$ with $\abs{B}\geq 2$
and every $q\in S^\Delta(B)$ there are $b_1,\ldots,b_d,c,c_1,\dots,c_n \in B$  such that $\mathcal F_\#(u;b_1,\dots,b_d,c,c_1,\dots,c_n)$ defines $q$. 
\end{lemma}
\begin{proof}
This is a simple coding trick due to Shelah (proof of Theorem~II.2.12~(1) in \cite{Shelah-book}, cf.~also \cite[Lemma~2.5]{Guingona}). For every $\varphi\in\Delta$ define
$$\varphi_\#(y;y_1,\dots,y_d,v,w_1,\dots,w_n):=\bigwedge_{i=1}^n \big(v=w_i \rightarrow \varphi_i(y;y_1,\dots,y_d)\big)$$
and let $\mathcal F_\#=(\varphi_\#)_{\varphi\in\Delta}$.
Let $B\subseteq M^{\abs{y}}$, $\abs{B}\geq 2$, and $q\in S^\Delta(B)$. By hypothesis, there are $b_1,\dots,b_d\in B$ and $i\in I$ such that $\mathcal F_i(y;b_1,\dots,b_d)$ defines $q$. Pick $c,c'\in B$ with $c\neq c'$, and put $c_i:=c$, $c_j:=c'$ for $j\neq i$. Then $\mathcal F_\#(y;b_1,\dots,b_d,c,c_1,\dots,c_n)$ defines~$q$.
\end{proof}

Similarly, the proof of Lemma~\ref{lem:encoding finite sets of formulas},~(1) shows that if 
the $\mathcal L$-formula $\psi_\Delta$ which we associated there to the finite set of $\mathcal L$-formulas $\Delta$
admits a uniform definition of $\psi_\Delta(x;B')$-types over finite sets with $d$ parameters, then $\Delta$ itself admits a uniform definition of $\Delta(x;B)$-types over finite parameter sets $B$ (with at least $2$ elements) having  $d+2$ parameters.

\medskip

On the other hand, if we have tight control over the number $d$ of parameters in our defining schemes, then we can bound the sizes of the $\Delta(x;B)$-type spaces over finite sets by polynomial functions (in the size of the parameter set) of degree $d$:  more precisely, if $\Delta$ allows a  uniform definition $\mathcal F=(\mathcal F_i)_{i\in I}$ of $\Delta(x;B)$-types over finite sets in $\mathbf M$ with $d$~parameters, then 
$\abs{S^\Delta(B)}\leq \abs{I}\,\abs{B}^d$  for every finite $B\subseteq M^{\abs{y}}$, hence $\pi^*_\Delta(t)\leq \abs{I}\, t^d$ for each $t$, and so $\vc^*(\Delta)\leq d$.

\subsection{The $\VC{}d$ property} 
We say that $\mathbf M$ has the \emph{$\VC{}d$ property} if  any $\Delta(x;y)$ with $\abs{x}=1$ has a uniform definition of $\Delta(x;B)$-types over finite sets with $d$~parameters.
Clearly, if $\mathbf M$ has the $\VC{}d$~property, then so does every elementarily equivalent $\mathcal L$-structure.
We say that a theory $T$ has the \emph{$\VC{}d$~property} if every model of $T$ has the $\VC{}d$ property.

\medskip

We point out that the $\VC{}0$ property  only holds in a very special situation; recall that a structure is called rigid if it has no automorphisms besides the identity.

\begin{lemma}
The following are equivalent:
\begin{enumerate}
\item $\mathbf M$ has the $\VC{}0$ property;
\item every model of $\Th(\mathbf M)$ is rigid;
\item $\mathbf M$ is finite and rigid.
\end{enumerate}
\end{lemma}
\begin{proof}
Suppose  $\mathbf M$ has the $\VC{}0$ property; to see (2), it suffices to show that $\mathbf M$ is rigid. 
Let $\varphi(x;y)$ be the $\mathcal L$-formula $x=y$, and let $\mathcal F=(\varphi_i(y))_{i\in I}$ be a uniform definition of $\varphi(x;B)$-types over finite sets. Suppose $\sigma\in\operatorname{Aut}(\mathbf M)$ and $b\in M$ satisfy $b\neq\sigma(b)$. Let $p=\tp^\varphi(b/B)$ where $B=\{b,\sigma(b)\}$, and choose $i\in I$ such that $\varphi_i$ defines $p$. Then
$$b=b \quad\Longleftrightarrow\quad \mathbf M\models\varphi_i(b) \quad\Longleftrightarrow\quad 
\mathbf M\models\varphi_i(\sigma(b)) \quad\Longleftrightarrow\quad b=\sigma(b),$$
a contradiction. This shows (1)~$\Rightarrow$~(2), and (2)~$\Rightarrow$~(3) is obvious. Suppose now that $M$ is finite, and let $\Delta(x;y)$, where $\abs{x}=1$, be a finite set of partitioned $\mathcal L$-formulas. It is easy to see that then for each $p\in S^\Delta(M^{\abs{y}})$ there is a family $\mathcal F_p=\{\varphi_p(y):\varphi\in\Delta\}$ of $\mathcal L$-formulas such that for all $b\in M^{\abs{y}}$ we have $\mathbf M\models\varphi_p(b)$ iff there is some automorphism $\sigma$ of $\mathbf M$ such that $\varphi(x;b)\in\sigma(p)$. Hence if in addition $\mathbf M$ is rigid then $\mathcal F=(\mathcal F_p)_{p\in S^\Delta(M^{\abs{y}})}$ is a uniform definition of $\Delta(x;B)$-types over finite sets in $\mathbf M$. This shows (3)~$\Rightarrow$~(1).
\end{proof}

The following result and its Corollary~\ref{cor:qe and UDTFS} are useful if we have some kind of quantifier elimination result at hand:

\begin{lemma}\label{lem:qe and UDTFS}
Suppose $\Delta=\Delta(x;y)$ and $\Psi=\Psi(x;y)$ are finite sets of partitioned $\mathcal L$-formulas such that every formula in $\Delta$ is equivalent in $T$ to a Boolean combination of formulas in $\Psi$, and $\Psi$ has UDTFS in $T$ with $d$ parameters. Then $\Delta$ has UDTFS  in $T$ with $d$ parameters.
\end{lemma}
\begin{proof}
We may assume that each $\varphi\in\Delta$ has the form
$$\varphi = \bigwedge_{r\in R_\varphi} \bigvee_{s\in S_{r,\varphi}} \Box_{r,s,\varphi}\,\psi_{r,s,\varphi}$$
where $R_\varphi$, $S_{r,\varphi}$ are finite index sets, $\Box_{r,s,\varphi}$ is $\neg$ or no condition, and $\psi_{r,s,\varphi}\in\Psi$. 
Suppose $\mathcal G=(\mathcal G_i)_{i\in I}$ is a uniform definition of $\Psi(x;B)$-types over finite sets in $T$, where
$$\mathcal G_i = \big( \psi_i(y;y_1,\dots,y_d) \big)_{\psi\in\Psi}\qquad\text{for each $i\in I$.}$$ 
Let $\mathcal F=(\mathcal F_i)_{i\in I}$ where 
$$\mathcal F_i = \big( \varphi_i(y;y_1,\dots,y_d) \big)_{\varphi\in\Delta}\qquad\text{where
$\varphi_i=\bigwedge_{r\in R_{\varphi}} \bigvee_{s\in S_{r,\varphi}} \Box_{r,s,\varphi}\,(\psi_{r,s,\varphi})_i$.}$$
Let $q\in S^\Delta(B)$ where $B\subseteq M^{\abs{y}}$ is finite. Let $a\in M^{\abs{x}}$ realize $q$, and put $p:=\tp^\Psi(a/B)$. Take $i\in I$ and $b_1,\dots,b_d\in B$ such that $\mathcal G_i(y;b_1,\dots,b_d)$ defines~$p$. One now easily verifies that $\mathcal F_i(y;b_1,\dots,b_d)$ defines $q$.
\end{proof}

\begin{corollary}\label{cor:qe and UDTFS}
Suppose $\Phi$ is a family of partitioned $\mathcal L$-formulas in the single object variable $x$ such that 
\begin{enumerate}
\item every partitioned $\mathcal L$-formula in the object variable $x$ is equivalent in $T$ to a Boolean combination of formulas from $\Phi$, and
\item every finite set of  $\mathcal L$-formulas from $\Phi$ has UDTFS in $T$ with $d$ parameters.
\end{enumerate}
Then $T$ has the $\VC{}d$~property.
\end{corollary}

The following theorem is at the root of the proof of Theorem~\ref{thm:weakly o-min} from the introduction; it shows that having UDTFS with a constant number of parameters for all sets of formulas in a single object variable entails UDTFS with a linearly bounded number of parameters for sets of formulas in an arbitrary number of object variables:

\begin{theorem}\label{VCdensity, 2}
Suppose that $\mathbf M$ has the $\VC{}d$ property. Then every $\Delta(x;y)$ has a uniform definition of $\Delta(x;B)$-types over finite sets in $\mathbf M$ with  $d\abs{x}$ parameters. 
\end{theorem}

Before we embark on the proof, we introduce some notation:
for a sequence $a\in M^m$ and a set $B\subseteq M^n$ we write $aB:=\{(a,b):b\in B\}\subseteq M^{m+n}$ and
$Ba:=\{(b,a):b\in B\}\subseteq M^{n+m}$.

\begin{proof}
We proceed by induction on $m=\abs{x}$. The base case $m=1$ holds by hypothesis. For the inductive step write $x=(x_0,x')$ where $x'=(x_1,\dots,x_m)$, and let $\Delta(x;y)$ be given. 
Let
$$
   \Delta_0(x_0;x',y) = \{\ph(x_0;x',y) : \ph(x;y)\in \Delta \}.
$$
By the $\VC{}d$ property applied to $\Delta_0$, we take finitely many families
$$\mathcal F_{i} = \big(\ph_i(x',y;y_1,\ldots,y_d)\big)_{\varphi\in\Delta}\qquad (i\in I)$$ of 
$\mathcal L$-formulas with the following property:
for any $a'\in M^m$, any finite set $B\subseteq M^{\abs{y}}$ and
any $q \in S^{\Delta_0}(a'B)$, there are $b_1,\ldots,b_d \in B$ and $i\in I$ 
such that $\mathcal F_{i}(a',y;b_1,\dots,b_d)$ defines $q$, i.e., for all $\ph\in\Delta$, $b\in B$:
\begin{equation}\label{eq:VCm, 1}
\varphi(x_0;a',b)\in q	\qquad\Longleftrightarrow\qquad \mathbf M\models \varphi_{i}(a',b;b_1,\dots,b_d).
\end{equation}
In the rest of this proof let $\varphi$ range over $\Delta$ and $i$ over $I$.
For each $i$, let 
$$
	\Delta_i(x';y,y_1,\ldots,y_d) =\big\{\ph_i(x';y,y_1,\ldots,y_d) : \ph(x;y) \in \Delta\big\}
$$
and apply the inductive hypothesis to each $\Delta_i$. Thus for each $i$ there are finite families 
$$\mathcal F_{ij} := \big(\ph_{ij}(y,y_1,\dots,y_d;v_1,\ldots,v_n)\big)_{\varphi\in\Delta}\qquad (j\in J_i)$$ of 
$\mathcal L$-formulas, where $n=md$, such that for all  finite subsets $B\subseteq M^{\abs{y}}$, all $\bb=(b_1,\dots,b_d)\in (M^{\abs{y}})^{d}$, and every  $p\in S^{\Delta_i}(B\bb)$, there exists some $j\in J_i$ and $c_1,\dots,c_n\in B$ such that for each $\varphi$ and each $b\in B$ we have
\begin{equation}\label{eq:VCm, 2}
\varphi_i(x';b,\bb)\in p \qquad\Longleftrightarrow\qquad \mathbf M\models \varphi_{ij}(b,\bb;c_1,\dots,c_n).
\end{equation}
Partition the variable tuple of the $\mathcal L$-formulas $\varphi_{ij}$ as $(y;y_1,\dots,y_d,v_1,\dots,v_n)$, and set
$\mathcal F := (\mathcal F_{ij})_{i\in I, j\in J_i}$.
We claim that $\mathcal F$ is a uniform definition of $\Delta(x;B)$-types over finite sets in $\mathbf M$; since $\mathcal F$ has $d+n=d(m+1)$ parameters, this will then finish the inductive step.
To see this, let a finite  $B\subseteq M^{\abs{y}}$ and some $a=(a_0,a')\in M^{1+m}$ be given. We let $b$ range over $B$.
We need to show that there are some $i\in I$, $j\in J_i$ and $\bb\in B^d$, $c_1,\dots,c_n\in B$ such that for all $\varphi$ and $b$ we have
\begin{equation}\label{eq:VCm, 3}
\mathbf M\models \varphi(a;b)\qquad\Longleftrightarrow\qquad \mathbf M\models \varphi_{ij}(b;\bb,c_1,\dots,c_n).
\end{equation}
Let  $q=\tp^{\Delta_0}(a_0/a'B)$ be the type in $S^{\Delta_0}(a'B)$ realized by $a_0$. Take  $i$ 
and $\bb=(b_1,\ldots,b_d)\in B^d$ such that \eqref{eq:VCm, 1} holds for all $\varphi$ and all $b$.
Set $p=\tp^{\Delta_i}(a'/B\bb)$. Then we may take $j\in J_i$ and $c_1,\dots,c_n\in B$ such that for each $\varphi$ and each $b$, the equivalence \eqref{eq:VCm, 2} holds. This yields \eqref{eq:VCm, 3}, for all $\varphi$ and $b$.
\end{proof}

In the next corollary we assume that $M$ is infinite (so we can meaningfully talk about VC~density). We already remarked that if $\Delta$ admits a uniform definition of $\Delta(x;B)$-types with $d$ parameters, then $\vc^*(\Delta)\leq d$; in particular, if $\mathbf M$ has the $\VC{}d$~property then this conclusion holds for all $\Delta(x;y)$ with $\abs{x}=1$. The previous theorem generalizes this observation to the upper bound $\vc^*(\Delta)\leq d\abs{x}$ for all $\Delta$. Hence:

\begin{corollary}\label{cor:VCm}
If $T=\Th(\mathbf M)$ has the $\VC{}d$~property then $m\leq \vc^T(m)\leq d\cdot m$ for every $m$. In particular, if $T$ has  the $\VC{}1$~property then $\vc^T(m) = m$ for every $m$.
\end{corollary}

In a multi-sorted setting, the $\VC{}d$ property can be naturally localized.
Suppose $\mathbf M$ is a multi-sorted structure, and let $S$ be one of the sorts of $\mathbf M$. We say that 
$\mathbf M$ has the $\VC{}d$ property in the sort $S$ if any $\Delta(x;y)$ with $x$ a single variable of sort $S$ has a uniform definition of $\Delta(x;B)$-types over finite sets with $d$ parameters.
With this definition, the following analogue of Theorem~\ref{VCdensity, 2} holds (with the same proof): 

\begin{corollary}\label{cor:VCdensity, localized}
If $\mathbf M$ has the $\VC{}d$ property in the sort $S$, then every $\Delta(x;y)$ with each variable $x_i$ of sort $S$ has a uniform definition of $\Delta(x;B)$-types over finite sets with $d\abs{x}$ parameters.
\end{corollary}

So for example, if $\mathbf M$ is an infinite single-sorted structure and some expansion of $\mathbf M^\eq$  has the $\VC{}d$ property in the home sort (where the uniform defining formulae are in the expanded language), then $\vc^T(m)\leq dm$ for each $m$.

\subsection{Relationship to other notions}\label{sec:relationship}
We now want to put the $\VC{}d$~property into perspective and compare it to two other strengthenings of the NIP concept, namely, {\it uniform definability of types over finite sets}\/ and {\it dp-minimality.}\/
The following notion was introduced and studied in \cite{Guingona}:

\begin{definition}
The structure $\mathbf M$ is said to have {\it uniform definability of types over finite sets}\/ ({\it UDTFS}\/) if every partitioned $\mathcal L$-formula has UDTFS in $\mathbf M$. Clearly UDTFS is an invariant of the elementary theory of $\mathbf M$, and so we say that an $\mathcal L$-theory $T$ has 
{\it uniform definability of types over finite sets}\/ ({\it UDTFS}\/) if every model of $T$ does.
\end{definition}

By Theorem~\ref{VCdensity, 2}, if $\mathbf M$ has the $\VC{}d$~property, for some $d$, then $\mathbf M$ has UDTFS.  More generally, \cite[Lemma~2.6]{Guingona} shows that if every partitioned $\mathcal L$-formula $\varphi(x;y)$ with $\abs{x}=1$ has UDTFS in $\mathbf M$, then $\mathbf M$ has UDTFS.
The proof of this lemma as given in \cite{Guingona} in fact shows that if every  partitioned $\mathcal L$-formula with a single object variable admits a uniform definition of $\varphi(x;B)$-types with $d$ parameters,
then every  partitioned $\mathcal L$-formula in the object variables $x$ admits a uniform definition of $\varphi(x;B)$-types with at most $(d+1)^{\abs{x}}-1$ parameters. (In contrast, our bound in Theorem~\ref{VCdensity, 2} is linear in $\abs{x}$.)

Every stable formula has UDTFS; in particular, every stable theory has UDTFS. In fact, Laskowski \cite{chrisdef} has shown that if $T$ is stable then every  partitioned $\mathcal L$-formula $\varphi(x;y)$  has UDTFS in $T$ with $R^m(x=x,\varphi,2)$ parameters.
(See \cite[Definition~II.1.1]{Shelah-book} for the definition of the rank $R^m(-,-,2)$. Stability of $T$ is equivalent to  $R^m(x=x,\varphi,2)<\omega$ for all $\varphi(x;y)$, cf.~\cite[Theorem~II.2.2]{Shelah-book}.)

\medskip

An \emph{ICT pattern}\/ in $\mathbf M$ consists of a pair $\alpha(x;y)$, $\beta(x;y)$ of partitioned $\mathcal L$-formulas, where $\abs{x}=1$,  and sequences $(a_i)_{i\in\bN}$, $(b_j)_{j\in\bN}$ in $M^{\abs{y}}$ such that for all $i$ and $j$ the set of $\mathcal L(M)$-formulas
$$\big\{\alpha(x;a_i),\beta(x;b_j)\big\}\cup\big\{ \neg\alpha(x;a_k) : k\neq i\big\} \cup \big\{ \neg\beta(x;b_l):l\neq j\big\}$$
is consistent (with $\mathbf M$). This notion and the following definition originate in \cite{S-strongly dependent}:

\begin{definition}
An $\mathcal L$-theory $T$ is said to be \emph{dp-minimal} if in no model of $T$ there is an ICT pattern, and  $\mathbf M$ is dp-mininmal if $\Th(\mathbf M)$ is dp-minimal \textup{(}equivalently, if there is no ICT pattern in an elementary extension of $\mathbf M$\textup{)}.
\end{definition}

The following proposition (which shows that in the definition of dp-minimality we could have restricted ourselves to ICT patterns given by identical formulas $\alpha$, $\beta$) and its Corollary~\ref{cor:dp-min} are due to Dolich, Goodrick and Lippel \cite{dl}; for convenience of the reader we indicate their proofs:

\begin{proposition}\label{prop:dp-min}
Suppose $\mathbf M$ is a monster model of the complete $\mathcal L$-theory $T$.
Then $T$ is dp-minimal iff there are no $\mathcal L$-formula $\varphi(x;y)$ with $\abs{x}=1$ and sequences $(c_i)_{i\in\bN}$, $(d_j)_{j\in\bN}$ in $M^{\abs{y}}$ such that for all $i$,~$j$,
$$\big\{\varphi(x;c_i),\varphi(x;d_j)\big\}\cup\big\{ \neg\varphi(x;c_k) : k\neq i\big\} \cup \big\{ \neg\varphi(x;d_l):l\neq j\big\}$$
is consistent.
\end{proposition}
\begin{proof}
Suppose $\alpha(x;y)$, $\beta(x;y)$ (where $\abs{x}=1$)  and the sequences $(a_i)_{i\in\bN}$, $(b_j)_{j\in\bN}$ are an ICT pattern in $\mathbf M$. Let $\varphi(x;y,z):=\alpha(x;y)\vee\beta(x;z)$, and for every $i$, $j$ let $c_i:=(a_{2i},b_{2i})$ and $d_j:=(a_{2j+1},b_{2j+1})$. Let $a\in M$ realize the type
$$\big\{\alpha(x;a_{2i}),\beta(x;b_{2j+1})\big\}\cup\big\{ \neg\alpha(x;a_k) : k\neq 2i\big\} \cup \big\{ \neg\beta(x;b_l):l\neq 2j+1\big\}.$$
Then $a$ satisfies $\varphi(x;c_i)$ and 
$\varphi(x;d_j)$. If $k\neq i$ then $a$ satisfies $\neg\alpha(x;a_{2k})$ (since $2k\neq 2i$) and 
$\neg\beta(x;b_{2k})$ (since $2k\neq 2j+1$) and hence also $\neg\varphi(x;c_k)$. 
Similarly we see that $a$ satisfies $\neg\varphi(x;d_l)$ for $l\neq j$.
\end{proof}

Let us tentatively say that  a complete $\mathcal L$-theory $T$ is \emph{vc-minimal} if $\vc^*(\varphi)<2$ for every $\mathcal L$-formula $\varphi(x;y)$ with $\abs{x}=1$. So if $\vc^T(1)<2$ (in particular, if $T$ is VC-minimal), then $T$ is vc-minimal.
An example of a theory $T$ which is not VC-minimal yet satisfies $\vc^T(1)=1$ (and thus is vc-minimal) was given in \cite[Proposition~3.7]{dl}.

\begin{corollary}\label{cor:dp-min}
Every vc-minimal theory is dp-mini\-mal.
\end{corollary}
\begin{proof}
Suppose $\mathbf M$ is a monster model of $T=\Th(\mathbf M)$, and $T$ is not dp-minimal. Take $\varphi$ and $(c_i)$, $(d_j)$ with the properties in the previous proposition. For every $n$ let $B_n:=\{ c_i : i<n \} \cup \{ d_j : j<n\}$.
Then $\abs{S^\varphi(B_n)}\geq n^2\geq \frac{1}{4}\abs{B_n}^2$, hence $\vc^*(\varphi)\geq 2$.
\end{proof}

In particular, each of the theories in Examples~\ref{ex:strongly minimal}--\ref{ex:C-minimal} (being VC-minimal) is dp-minimal.
Other proofs of the dp-minimality of weakly o-minimal theories can be found in \cite{Adler-VCmin, dl}.
See also \cite{KOU} for a generalization of Corollary~\ref{cor:dp-min} to a bound on ``dp-rank'' in terms of VC~density.

\medskip

A characterization of dp-minimal theories among stable theories was given in \cite{OU}.
The main result of \cite{Guingona} is that every dp-minimal theory $T$ has UDTFS. In particular, by Corollary~\ref{cor:dp-min}, every vc-minimal $T$ has UDTFS.  
(Actually, \cite[Theorem~3.14]{Guingona} gives a more precise result: if $\varphi(x;y)$ is an $\mathcal L$-formula such that $\pi_\varphi(t)\leq {t+1\choose 2}$ for some $t>0$, then $\varphi$ has UDTFS.)

\medskip

We summarize the implications between the properties of a theory $T$ discussed above in the following diagram:
$$\xymatrixcolsep{1.5pc}\xymatrixrowsep{1.7pc}
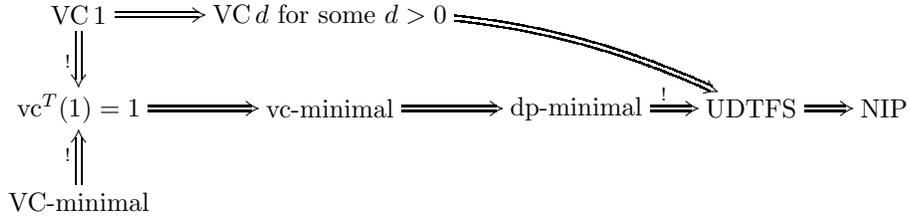
\xymatrix{
\text{$\VC{}1$} \ar@2[d]_{!}	\ar@2[r] & \text{$\VC{d}$ for some $d>0$} \ar@2@/^1.1em/[drr] \\
	 \vc^T(1)=1 \ar@2[r] & \text{vc-minimal} \ar@2[r] & \text{dp-minimal}\ar@2[r]^{!} & \text{UDTFS}\ar@2[r] & \text{NIP} \\
\text{VC-minimal} \ar@2[u]^{!}
}
$$
Here the arrows marked with an exclamation mark are known not to be reversible. (For an example showing that $\vc(1)=1 \not\Rightarrow \VC{}1$ see \cite[Example~3.15]{ADHMS}.) We do not know which of the other arrows are reversible;
whether the converse of the implication $\text{UDTFS}\Rightarrow \text{NIP}$ holds was first asked by Laskowski \cite[Open~Question~4.1]{Guingona}.

\medskip

Recall from Corollary~\ref{cor:Shelah expansion} that the Shelah expansion $\mathbf M^{\Sh}$ of $\mathbf M$ has the same VC~density function as $\mathbf M$.
In \cite{OU} it is observed that the Shelah expansion of a dp-minimal structure is again dp-minimal. We finish this section by showing that the analogous statement also holds for the $\VC{}d$ property; in fact, we have more generally:

\begin{proposition}
Suppose every finite set of partitioned $\mathcal L$-formulas in $m$ object variables has UDTFS in $T$ with $d$ parameters. Then 
every finite set of partitioned $\mathcal L^{\Sh}$-formulas in $m$ object variables has UDTFS in $T^{\Sh}$ with $d$ parameters.
\end{proposition}
\begin{proof}
As in the definition of the Shelah expansion (cf.~Section~\ref{sec:coding}) let $\mathbf M^*$ be a very saturated elementary extension of $\mathbf M$.
Let $\Delta=\Delta(x;y)$ be a finite set of partitioned $\mathcal L^{\Sh}$-formulas where $m=\abs{x}$; below $\varphi$ ranges over $\Delta$. We need to show that $\Delta$ has UDTFS in $T^{\Sh}$ with $d$ parameters. For this, by Lemma~\ref{lem:qe and UDTFS} and since $T^{\Sh}$ admits quantifier elimination, we may assume that each of the $\mathcal L^{\Sh}$-formulas in $\Delta$ is atomic; that is, there exist $\mathcal L$-formulas $\psi_\varphi(x,y;z)$, one for each $\varphi$, and a tuple $c\in (M^*)^{\abs{z}}$ such that each $\varphi$ has the form $\varphi(x;y)=R_{\psi_\varphi,c}(x;y)$. Let now $\Psi(x;y,z):=\{\psi_\varphi(x;y,z):\varphi\in\Delta\}$ and take a uniform definition  $\mathcal G=(\mathcal G)_{i\in I}$ of $\Psi(x;B^*)$-types over finite sets in $T^{\Sh}$, where
$$\mathcal G_i=\big((\psi_{\varphi})_i(y,z;(y_1,z_1),\dots,(y_d,z_d))\big)_{\varphi\in\Delta}\qquad\text{for each $i\in I$.}$$
For each $i\in I$ set
\begin{multline*}
\mathcal F_i:=\big(R_{\varphi_i,c}(y;y_1,\dots,y_d)\big)_{\varphi\in\Delta}\\ \text{where $\varphi_i(y,y_1,\dots,y_d;z):=(\psi_\varphi)_i(y,z,(y_1,z),\dots,(y_d,z))$.}
\end{multline*}
We claim that $\mathcal F=(\mathcal F_i)_{i\in I}$ is a uniform definition of $\Delta(x;B)$-types over finite sets in $T^{\Sh}$. To see this let $p\in S^\Delta(B)$ where $B\subseteq M^{\abs{y}}$ is finite, and let $a\in M^{m}$ be a realization of $p$ in $\mathbf M^{\Sh}$. Put $B^*:=B\times\{c\}\subseteq M^{\abs{y}}\times (M^*)^{\abs{z}}$ and let $p^*:=\tp^\Psi(a/B^*)$ (in $\mathbf M^*$).
Take $b_1,\dots,b_d\in B$ and $i\in I$ such that $\mathcal G_i(y,z;(b_1,c),\dots,(b_d,c))$ defines~$p^*$; then for every $\varphi$ and $b\in B$ we have
\begin{align*}
\mathbf M^{\Sh}\models \varphi(a;b)	&\quad\Longleftrightarrow\quad \mathbf M^*\models\psi_\varphi(a;b,c) \\
									&\quad\Longleftrightarrow\quad \mathbf M^*\models (\psi_{\varphi})_i(b,c;(b_1,c),\dots,(b_d,c)) \\
									&\quad\Longleftrightarrow\quad \mathbf M^{\Sh}\models R_{\varphi_i,c}(b;b_1,\dots,b_d).
\end{align*}
That is, $\mathcal F_i(y;b_1,\dots,b_d)$ defines $p$ (in $\mathbf M^{\Sh}$).
\end{proof}

\section{Examples of $\VC{}d$: Weakly O-minimal Theories and Variants}\label{sec:examples of VCm theories}

\noindent
In this section we apply Theorem~\ref{VCdensity, 2} from the preceding section to give a proof of Theorem~\ref{thm:weakly o-min} on VC~density in weakly o-minimal theories from the introduction. We also observe that a similar technique allows us to treat all (weakly) quasi-o-minimal theories. 

Throughout this section $\calL$ is a language containing a binary relation symbol ``$<$'' and $T$ is a theory extending the theory of infinite linear orderings. 

\subsection{Weakly o-minimal theories}
We begin by introducing some terminology concerning ordered sets. Let $(X,{<})$ be a linearly ordered set, and let $S$ be a subset of $X$ which is
a union of finitely many non-empty convex subsets of $X$. We refer to the convex sets in the unique  minimal such presentation of $S$ as its {\em \textup{(}convex\textup{)} components}. 
Suppose $S$ has $N$ convex components, where $N>0$.
These components are ordered by~$<$, so for $i=1,\dots,N$ we can refer to the $i$th component of $S$; for $i>N$ we declare the $i$th component of $S$ to be equal to the $N$th.


Recall that $T$ is called {\em weakly o-minimal} if for any $\mathbf M \models T$,  any definable subset of $M$ is a finite union of convex subsets of~$M$. 

\begin{theorem}\label{weaklyomin}
Assume that $T$ is weakly o-minimal. Then $T$ has the $\VC{}1$ property, and hence any finite set $\Delta(x;y)$ of $\mathcal L$-formulas has dual VC density at most $\abs{x}$.
\end{theorem}
\begin{proof} Let $\mathbf M\models T$.
Fix a finite non-empty set of $\calL$-formulas $\Delta(x;y)$ with $\abs{x}=1$. We let $\ph$ range over $\Delta$ and $b$ over $M^{\abs{y}}$. By the weak o-minimality of $T$, there is an integer $N>0$ such that
for any $\ph$ and  any $b$, $\ph(M;b)$ has at most
$N$ components. For any $\ph$ and $i\in [N]$ there is an $\mathcal L$-formula $\varphi^i(x;y)$ such that for every $b$ with $\ph(M;b)\neq\emptyset$, the $i$th component of $\ph(M;b)$ equals $\ph^i(M;b)$, and such that for every $b$ with $\ph(M;b)=\emptyset$ we have $\ph^i(M;b)=\emptyset$. So 
\begin{equation}\label{eq:initial seg, 1}
\varphi(M;b)=\varphi^1(M;b)\cup\cdots\cup\varphi^N(M;b)\qquad\text{for every $b$.}
\end{equation}
Set
\begin{align*}
\ph_{\leq}^i(x;y)	&:= \exists x_0 (\ph^i(x_0;y) \wedge x\leq x_0), \\
\ph_{<}^i(x;y)		&:= \forall x_0 (\ph^i(x_0;y) \rightarrow x< x_0).
\end{align*}
Then clearly 
\begin{equation}\label{eq:initial seg, 2}
\ph^i(M;b) = \ph^i_{\leq}(M;b)\cap \big(M\setminus \ph^i_{<}(M;b)\big)\qquad\text{for all $\ph$, $b$ and $i\in [N]$.}
\end{equation}
Now set
$$\Psi(x;y) := \big\{ \ph_{\Box}^i(x;y):\ph\in\Delta,\ i\in [N],\ \Box\in\{{\leq},{<}\} \big\}.$$
For each $\psi\in\Psi$ and each $b$, the set $\psi(M;b)$ is an initial segment of $M$; hence $\mathcal S_\Psi$ is directed, so $\Psi$ has UDTFS with a single parameter, by Lemma~\ref{lem:breadth and UDTFS}.
Moreover, by \eqref{eq:initial seg, 1} and \eqref{eq:initial seg, 2}, every $\ph$ is equivalent to a Boolean combination of $2N$ formulas from $\Psi$. Hence $\Delta$ also has UDTFS with a single parameter, by Lemma~\ref{lem:qe and UDTFS}.
Thus $\mathbf M$ has the $\VC{}1$ property. By Corollary~\ref{cor:VCm} therefore $\vc^*(\Delta)\leq\abs{x}$ for every finite set $\Delta(x;y)$ of $\mathcal L$-formulas.
\end{proof}

\begin{remark*} 
In the previous theorem we assume that the {\it theory}\/ $T$ is weakly o-minimal (i.e., all models of $T$ are weakly o-minimal). 
Recall that a weakly o-minimal structure need not have weakly o-minimal theory. 
We do not know whether the conclusion of Theorem~\ref{weaklyomin} holds if $T$ is merely assumed to have {\it some}\/ weakly o-minimal model.
\end{remark*}

Put $\calL_{\divi,<}:=\calL_{\divi}\cup\{<\}$, where $\calL_{\divi}=\{0,1,{+},{-},{\times},{\,|\,}\}$ is the language of rings expanded by a divisibility predicate (see Example~\ref{ex:C-minimal} above) and
``$<$'' is a binary relation symbol.
Let $\RCVF$ denote the theory of real closed fields equipped with a proper convex valuation ring, parsed in the language $\calL_{\divi,<}$. The following corollary is now immediate, as by \cite{di}, $\RCVF$ is weakly o-minimal. 

\begin{corollary}\label{rcvf}
Let $K\models \RCVF$. Then any finite set $\Delta(x;y)$ of $\calL_{\divi,<}$-formulas has dual VC~density at most $\abs{x}$ in $K$.
\end{corollary}

This result in turn yields a  VC density bound for algebraically closed valued fields of residue characteristic $0$ (which is non-optimal by Example~\ref{ex:C-minimal}):

\begin{corollary}\label{acvf}
Let $\ACVF_{(0,0)}$ be the theory of non-trivially valued algebraically closed fields of residue characteristic $0$, in the language $\calL_{\divi}$. Let $\Delta(x;y)$ be a finite set of 
$\calL_{\divi}$-formulas. Then $\Delta$ has dual VC~density at most $2\abs{x}$ in $\ACVF_{(0,0)}$.
\end{corollary}

\begin{proof}
The theory $\ACVF_{(0,0)}$, which is complete,  is interpretable in $\RCVF$: if $K$ is a model of $\RCVF$, then its algebraic closure $K^{\operatorname{alg}}$ is a degree $2$ extension of $K$: $K^{\operatorname{alg}}=K(i)$, where $i^2=-1$. So $K^{\operatorname{alg}}$ can be identified with $K^2$, and the valuation $v$ of $K$ can be definably extended to one of $K^{\operatorname{alg}}$ by setting $v(a+bi)=\frac{1}{2}v(a^2+b^2)$ for $a,b\in K$.
Thus, by Lemma~\ref{interp} and Corollary~\ref{rcvf}, $\Delta(x;y)$ has dual VC density at most $2\abs{x}$.
\end{proof}

We have no results on VC density for $\ACVF$ in characteristics other than $(0,0)$.

\subsection{Quasi-o-minimal theories}
We now turn to quasi-o-minimal theories: $T$ is said to be {\em quasi-o-minimal}\/ if for any $\mathbf M \models T$,  any definable 
subset of $M$ is a finite Boolean combination of singletons, intervals in $M$,  and $\emptyset$-definable sets. 
(See \cite{bpw}.)

\begin{theorem}\label{quasiomin}
Assume that $T$ is quasi-o-minimal. Then $T$ has the  $\VC{}1$ property, and hence $\vc^T(n)=n$ for each $n$. 
\end{theorem}

\begin{proof} 
Let $\mathbf M\models T$.
Fix a finite set $\Delta(x;y)$ of $\calL$-formulas with $\abs{x}=1$; we let $\ph$ range over $\Delta$ and $b$ over $M^{\abs{y}}$. 
There is some positive integer $N$ and  
$\emptyset$-definable subsets $D_1, \dots, D_N$ of $M$ so that for 
any  $\ph$ and any choice of parameters $b$,  the set $\ph(M;b)$ of realizations of $\ph(x;b)$ is a Boolean combination of the $D_i$ and at most $N$ singletons and intervals in $M$ \cite[Theorem~3]{bpw}. 

Let $\D$ be the collection of sets of the form $\widetilde{D}_1\cap\dots\cap \widetilde{D}_N$, where $\widetilde{D}_i$ is either $D_i$ or its  complement in $M$ (so $\mathcal D$ is a partition of $M$ into at most $2^N$ sets). We let $D$ range over $\mathcal D$.
For every $D$, $\varphi$, and $b$, the set $D\cap\varphi(M;b)$ is then a finite union of at most $N$ convex subsets of the ordered set $D$.
For every $i\in [N]$ and every $D$  let $\varphi^{i,D}(x;u)$ be an $\mathcal L$-formula such that for every $b$, if the set $D\cap\varphi(M;b)$ is non-empty, then the $i$th convex component of  $D\cap\varphi(M;b)$ (viewed as a subset of the ordered set $D$) is given by $\varphi^{i,D}(M;b)$, and if $D\cap\varphi(M;b)=\emptyset$ then $\varphi^{i,D}(M;b)=\emptyset$. 
Hence for each $\ph$ and $b$ we have
$$\varphi(M;b) = \bigcup_{D\in\mathcal D,\ i\in [N]} \varphi^{i,D}(M;b).$$
Now let (slightly abusing syntax)
\begin{align*}
\ph_{\leq}^{i,D}(x;y)	&:= x\in D \wedge \exists x_0 (x_0\in D \wedge \ph^i(x_0;y) \wedge x\leq x_0), \\
\ph_{<}^{i,D}(x;y)		&:= x\in D \wedge \forall x_0 (x_0\in D \wedge \ph^i(x_0;y) \rightarrow x< x_0).
\end{align*}
Then 
$$\ph^{i,D}(M;b) = \ph^{i,D}_{\leq}(M;b)\cap \big(M\setminus \ph^{i,D}_{<}(M;b)\big)\qquad\text{for all $\ph$, $b$, $D$ and $i\in [N]$.}$$
Each set $\varphi^{i,D}_\Box(x;b)$, where $\Box\in\{{\leq},{<}\}$, is an initial segment of $D$, and any two distinct elements of $\mathcal D$ are disjoint. Thus the set system $\mathcal S_\Psi$, where
$$\Psi(x;y) = \big\{ \ph_{\Box}^{i,D}(x;y):\ph\in\Delta,\ i\in [N],\ \Box\in\{{\leq},{<}\},\ D\in\mathcal D \big\},$$
is directed. As in the proof of Theorem~\ref{weaklyomin} it now follows that $\Delta$ has UDTFS with a single parameter.
\end{proof}

\begin{corollary}\label{cor:presburger}
The following structures all have the  $\VC{}1$ property:
\begin{enumerate} 
\item $( \bR, {<}, \bQ )$ \textup{(}i.e., the ordered set of reals with a predicate for the rationals\textup{)};
\item $( \bZ^n, {<}, {+} )$ where $<$ is the lexicographic ordering on $\bZ^n$;
\item $( \bZ^n \times \bQ, {<}, {+} )$ where $<$ is the lexicographic ordering on $\bZ^n \times \bQ$.
\end{enumerate}
\end{corollary}
\begin{proof}
Each of the examples has quasi-o-minimal theory: For (1) this was noted in \cite[Section 1]{bpw}, and for (2) and (3) this is proved (based on a quantifier-elimination result from \cite{w}) in \cite[Theorem~15]{bvw}. 
The corollary now follows from Theorem~\ref{quasiomin}.
\end{proof}

\begin{remark*}
The ordered abelian groups in (2) and (3) of the previous corollary are typical for quasi-o-minimal groups.
Here and below, ``quasi-o-minimal group'' means ``quasi-o-minimal expansion of an ordered group.''
(A quasi-o-minimal group is necessarily abelian \cite[Theorem~11]{bpw}.)
An expansion $\mathbf G$ of an ordered group is called \emph{coset-minimal} if every  subset of $G$ definable in $\mathbf G$ is a finite union of cosets of definable subgroups, intersected with intervals. A theory expanding the theory of ordered groups is said to be coset-minimal if all its models are. (See \cite{bvw}.) Now by \cite[Theorem~5.3]{Point-Wagner}, the theory of $\mathbf G$ is coset-minimal iff the theory of the expansion of $\mathbf G$ by constant symbols for the elements of $G$  is quasi-o-minimal, and in this case $\mathbf G$ is an expansion of an ordered group elementarily equivalent to either $( \bZ^n, {<}, {+} )$ or $( \bZ^n \times \bQ, {<}, {+} )$, for some~$n$.
\end{remark*}

Part (2) of the previous corollary shows in particular that Presburger Arithmetic, i.e., the theory of the ordered group $(\mathbb Z,{<},{+})$ of integers, has the $\VC{}1$~property, and hence is dp-minimal.
By Theorem~10 in \cite{bpw}, $(\mathbb Z,{<},{+})$  has no proper quasi-o-minimal expansions. 
One can strengthen this statement:

\begin{proposition}\label{prop:no proper dp-min exp of Presburger}
No proper expansion of $(\mathbb Z,{<},{+})$ is dp-minimal.
\end{proposition}

The proof is the same as in \cite{bpw}, replacing the use of  \cite[Theorem~7]{bpw} by a result from \cite{Simon}; we state the latter employing some convenient terminology from \cite{bpw}: Let $(X,{<})$ be a linearly ordered set. We say that two subsets $S$, $T$ of $X$ are \emph{eventually equal} (in symbols: $S\approx T$) if there is some $a\in X$ such that $S\cap (a,+\infty)=T\cap (a,+\infty)$. Clearly $\approx$ is an equivalence relation on subsets of $X$. We say that a family of subsets of $X$ is \emph{eventually finite} if it is partitioned into finitely many classes by~$\approx$.

\begin{lemma}[Simon {\cite[Lemma~2.9]{Simon}}]\label{lem:Simon}
Suppose $T$ is dp-minimal. Let $\mathbf M\models T$ and let $\varphi(x;y)$ be a partitioned $\mathcal L(M)$-formula where $\abs{x}=1$. Then $\mathcal S_\varphi$ is eventually finite.
\end{lemma}

For the benefit of the reader we now indicate the details of the proof of Proposition~\ref{prop:no proper dp-min exp of Presburger}.
Let $\mathbf Z$ be a proper expansion of $(\mathbb Z,{<},{+})$. By a theorem of Michaux and Villemaire \cite{MV} (and an easy extra argument, given in the proof of \cite[Theorem~10]{bpw}), there is a subset $U$ of $\bZ$ which is definable in $\mathbf Z$ but not definable in $(\mathbb Z,{<},{+})$.
By Simon's lemma, the family $\{a+U:a\in\bZ\}$ is eventually finite; thus the subgroup $A$ of $\bZ$ consisting of all $a\in\bZ$ such that $a+U\approx U$ is non-zero, so $A=a\bZ$ for some positive integer $a$.
Let $V$ be the union of all cosets of $A$ which contain arbitrarily large elements of $U$; then $V$ is definable in $(\mathbb Z,{<},{+})$, hence it suffices to show that $U\approx V$.
As $a+U\approx U$, we can take $\alpha\in\bZ$ such that for every $u\in\bZ$ with $u\geq\alpha$ we have $a+u\in U \Longleftrightarrow u\in U$. One now proves easily that for  every $u\in\bZ$ with $u\geq\alpha$ we have $u\in U\Longleftrightarrow u\in V$. \qed

\medskip

So for example, the expansion of $(\mathbb Z,{<},{+})$ by the set $b^{\bN}=\{b^n:n\geq 0\}$ of powers of a natural number $b>1$
is not dp-minimal, as is the expansion of $(\mathbb Z,{<},{+})$ by the set of factorials or by the set of Fibonacci numbers.
In all these examples, the corresponding expansion of $(\mathbb Z,{<},{+})$ has quantifier elimination in a natural expansion of $\{<,{+},U\}$ (see \cite{Cherlin-Point, Point}) and is NIP (as will be shown elsewhere).

\subsection{Weakly quasi-o-minimal theories}
In \cite{Ku1}, $T$ is called \emph{weakly quasi-o-minimal} if for any $\mathbf M\models T$, any definable 
subset of $M$ is a finite Boolean combination of convex subsets of $M$ and $\emptyset$-definable sets. Every 
weakly quasi-o-minimal theory is NIP \cite[Theorem~2.3]{Ku1}. In fact, the proof of Theorem~\ref{quasiomin} (mutatis mutandis) also shows more generally:

\begin{theorem}\label{thm:weakly quasi-o-min}
All weakly quasi-o-minimal theories have the $\VC{}1$ property.
\end{theorem}

This observation can be used to strengthen \cite[Proposition~4.2]{Simon}, where it is shown that the complete theories of colored linearly ordered sets with monotone relations (shown to be NIP in \cite{Schmerl}) are dp-minimal. A binary relation $R$ on a set $X$ is said to be monotone with respect to a linear ordering $<$ of $X$ if 
$$x'\leq xRy\leq y'\Rightarrow x'Ry'\qquad\text{for all $x,x',y,y'\in X$.}$$
A \emph{colored linearly ordered set with monotone relations} is a
 structure of the form $\mathbf M=(M,{<},\{C_i\}_{i\in I},\{R_j\}_{j\in J})$ where $<$ is a linear ordering on $M$, the $C_i$ are unary predicates (``colors''), and the $R_j$ are binary relations which are monotone (with respect to $<$). It was shown by Simon \cite[Proposition~4.1]{Simon} that every such colored linearly ordered set with monotone relations has quantifier elimination provided that each $\emptyset$-definable subset of $M$ is given by one of the predicates $C_i$ and each monotone $\emptyset$-definable binary relation is given by one of the $R_j$.

\begin{proposition}\label{prop:simon}
Let $\mathbf M$ be a colored linearly ordered set with monotone relations as above. Then $T=\Th(\mathbf M)$ is weakly quasi-o-minimal, and hence has the $\VC{}1$ property.
\end{proposition}
\begin{proof}
By the result of Simon just quoted, we may assume that $\mathbf M$ admits quantifier elimination. Now for each $b\in M$ and $j\in J$ the set
$$\{x\in M: \mathbf M\models xR_jb \}$$
is an initial segment of $M$, and 
$$\{x\in M: \mathbf M\models bR_jx \}$$
is a final segment of $M$ (i.e., its complement is an initial segment of $M$). Hence any definable subset of $M$ is a finite Boolean combination of initial segments of $M$ and $\emptyset$-definable sets, so $T$ is weakly quasi-o-minimal.
\end{proof}

As in \cite{Schmerl, Simon} this leads to a result for (partially) ordered sets of finite width. 
Let $\mathbf P=(P,{<})$ be an ordered set, i.e., a set $P$ equipped with an irreflexive, asymmetric and transitive binary relation $<$ on $P$. A subset $A$ of $P$ is an antichain if for all $a\neq a'$ in $A$, neither $a<a'$ nor $a'<a$ holds, and $C\subseteq P$ is a chain if for all $c\neq c'$ in $C$, either $c<c'$ or $c'<c$.
The \emph{width} of $\mathbf P$ is defined to be the supremum of the cardinalities of antichains in $P$, and denoted by $\width(\mathbf P)$. (Dually, the \emph{height} of $\mathbf P$ is defined to be the supremum of the cardinalities of a chain in $P$, denoted by $\height(\mathbf P)$.) A \emph{colored ordered set} is a structure  $\mathbf P=(P,{<},(C_i)_{i\in I})$ where $(P,{<})$ is an ordered set and each $C_i$ is a unary predicate.
 
\begin{corollary}
Let $\mathbf P=(P,{<},(C_i)_{i\in I})$ be an infinite colored ordered set of finite width. Then 
$\vc^{\Th(\mathbf P)}(m)=m$ for every $m$.
\end{corollary}
\begin{proof}
Let $n=\width(\mathbf P)$ and let $i$, $j$ range over $[n]$.
By Dilworth's Theorem there is a partition $P=P_1\cup\cdots\cup P_n$ of $P$ into disjoint chains $P_i$. 
Define a linear ordering~$\prec$ on~$P$ by setting $a\preceq b$ iff either $a,b\in P_i$ for some $i$ and $a\leq b$, or $a\in P_i$, $b\in P_j$ with $i<j$. For all $i$, $j$ the binary relation
$$R_{ij} := \big\{ (a,b)\in P: \exists a'\in P_i, b'\in P_j: a\preceq a'\leq b'\preceq b \big\}$$
is monotone with respect to $\prec$, and the original ordering $<$ is $\emptyset$-definable in the  linearly ordered set with monotone relations $(P,{\prec},(R_{ij})_{i,j})$, noting that $a\leq b$ iff $a R_{ii} a R_{ij} b R_{jj} b$ for some $i$ and $j$. The claim now follows immediately from Proposition~\ref{prop:simon}.
\end{proof}

\begin{question}
Is every ordered set of finite width $\VC{}1$?
\end{question}

By an interpretability argument, the previous corollary also leads to a (perhaps non-optimal) bound on the VC~density for those distributive lattices with NIP theory. By \cite[Theorem~6]{Schmerl} these are exactly the distributive lattices of finite breadth. From Section~\ref{sec:breadth} recall that a semilattice $(L,{\wedge})$ has breadth at most $d$ if for all $b_1,\dots,b_{d+1}\in L$ there is some $i\in [d+1]$ such that $b_1\wedge \cdots\wedge b_{d+1} = b_1\wedge\cdots\widehat{b_i}\cdots\wedge b_{d+1}$, and the smallest such $d$ (if it exists) is called the breadth of $(L,{\wedge})$.

\begin{corollary}
Let $\mathbf L=(L,{\wedge},{\vee})$ be an infinite distributive lattice of  breadth~$d$. Then 
$\vc^{\Th(\mathbf L)}(m)\leq dm$ for every $m$.
\end{corollary}
\begin{proof}
Let $\mathbf P=(P,{<})$ be an ordered set of width $d$, and let $\mathcal A(\mathbf P)$ be the set of antichains of $\mathbf P$ (so each element of $\mathcal A(\mathbf P)$ is a subset of $P$ of size $\leq d$). For $A,A'\in \mathcal A(\mathbf P)$ let $A\wedge A'$ denote the set of minimal elements of $A\cup A'$ and $A\vee A'$ the set of maximal elements of $A\cup A'$; then $A\wedge A', A\vee A'\in \mathcal A(\mathbf P)$, and $(\mathcal A(\mathbf P),{\wedge},{\vee})$ is a distributive lattice of breadth~$d$. Moreover, one can choose $\mathbf P$ such that the given distributive lattice $\mathbf L$ is isomorphic to~$\mathcal A(\mathbf P)$ \cite[Theorem~3]{Schmerl-distributive-lattices}. Since $\mathcal A(\mathbf P)$ is interpretable in $\mathbf P$ on a definable subset of $P^d$, we have $\vc^{\Th(\mathbf L)}(m)\leq \vc^{\Th(\mathbf P)}(dm)$ by Corollary~\ref{cor:interp} and hence $\vc^{\Th(\mathbf L)}(m)\leq dm$ by the previous corollary.
\end{proof}

In \cite{Simon} it is shown that the complete theory of each infinite tree $\mathbf T$ (viewed as an ordered set) is dp-minimal. Here, a tree is an ordered set $\mathbf T=(T,{<})$ with the property that for each $t\in T$ the set $\{t'\in T:t'< t\}$ is linearly ordered (by the restriction of $<$).

\begin{problem*}
Determine the VC~density function of each (infinite) tree.
\end{problem*}

(It is known \cite{Parigot} that a tree $\mathbf T$ is stable iff $\mathbf T$ has finite height, and then $\mathbf T$ is superstable of $\URk$-rank $\leq \height(\mathbf T)$, so conceivably, the methods of \cite{ADHMS} could  be applied.)

\section{A Strengthening of $\VC{}d$, and $P$-adic Examples}\label{sec:strong VCm}

\noindent
In this section we first introduce a strengthening of the $\VC{}d$ property defined and studied in Section~\ref{sec:VCm property}, and we prove a more precise version of Theorem~\ref{VCdensity, 2} for strong $\VC{}d$ structures.
The extra precision afforded by this theorem is useful in situations where $\vc(1)=1$, yet we can only prove the $\VC{}d$ property for some $d>1$. This is the case for $P$-minimal theories, which are discussed in the last subsection, where we prove Theorem~\ref{thm:p-adic} from the introduction.

\subsection{The strong $\VC{}d$ property}
In the following $\mathbf M$ is a structure in a language~$\mathcal L$ and $\Delta=\Delta(x;y)$ is a finite non-empty set of partitioned $\mathcal L$-formulas.
Let $\mathcal F=(\mathcal F_i)_{i\in I}$ be a uniform definition of $\Delta(x;B)$-types over finite sets, where
$$\mathcal F_i=\big(\varphi_i(y;y_1,\ldots,y_d)\big)_{\ph\in\Delta}\qquad (i\in I).$$ 
Recall from  Definition~\ref{def:UDTFS} above that this means the following: for every finite set $B\subseteq M^{\abs{y}}$
and $q\in S^\Delta(B)$ there are $b_1,\ldots,b_d \in B$ and some $i\in I$ such that $\mathcal F_i(y;b_1,\dots,b_d)$
defines $q$. 
If in addition for {\it every}\/ choice of $b_1,\ldots,b_d \in M^{\abs{y}}$ and $i\in I$, the set 
$$p_i(x;b_1,\dots,b_d):=\big\{\varphi(x;b) : \ph\in\Delta,\ b\in M^{\abs{y}},\ \mathbf M\models \varphi_{i}(b;b_1,\dots,b_d)\big\}$$
of $\mathcal L(M)$-formulas is consistent \textup{(}with $\mathbf M$\textup{)}, then we say that  $\mathcal F$ is a 
\emph{coherent}  definition of $\Delta(x;B)$-types over finite sets. (In this case, every restriction of $p_i(x;b_1,\dots,b_d)$ to a finite parameter set $B\subseteq M^{\abs{y}}$ extends to a complete $\Delta(x;B)$-type, but $\mathcal F_i(y;b_1,\dots,b_d)$ does not in general define such an extension.)

\begin{remark*}
Often all our defining formulas $\varphi_i$ have the syntactic form
$$\varphi_i(y;y_1,\ldots,y_d) = \forall x \big(\chi_{i}(x;y_1,\dots,y_d)\rightarrow \varphi(x;y)\big)$$
where $\chi_{i}$ is an $\mathcal L$-formula. In this case, the coherency condition for $\mathcal F$ is automatically satisfied provided  $\chi_i^{\mathbf M}(x;\overline{b})\neq\emptyset$ for all $i\in I$ and $\overline{b}\in (M^{\abs{y}})^d$.
For example if $\mathcal S_\Delta$ has breadth $d$ and is $d$-consistent, then $\Delta$ has a coherent definition of $\Delta(x;B)$-types over finite sets with $d$ parameters. (See Lemma~\ref{lem:breadth and UDTFS}.)
\end{remark*}

We say that $\mathbf M$ has the {\em strong $\VC{}d$ property} if any $\Delta(x;y)$ with $\abs{x}=1$ has  a coherent definition of $\Delta(x;B)$-types over finite sets with $d$~parameters.
Clearly, the strong $\VC{}d$ property is a property of the elementary theory of $\mathbf M$. We say that a theory $T$ has the {\em strong $\VC{}d$ property} if every model of $T$ has the strong $\VC{}d$ property.

\begin{remark*}
Suppose $\abs{x}=1$ and $\mathcal F=(\mathcal F_i)_{i\in I}$ is a coherent definition of $\Delta(x;B)$-types over finite sets in $\mathbf M$, where
$\mathcal F_i=(\varphi_i)_{\varphi\in\Delta}$. Then for every  $\Delta'\subseteq\Delta$, the families $\mathcal F_i':=(\varphi_i)_{\varphi\in\Delta'}$ form a 
coherent definition of $\Delta'(x;B)$-types over finite sets in $\mathbf M$.
This shows in particular that in order to check that $\mathbf M$ has the strong~$\VC{}d$~property, one may restrict oneself to sets of $\mathcal L$-formulas $\Delta$ which are closed under negation.
\end{remark*}

In the rest of this subsection we assume that $M$ is infinite. We have the following result on counting types in structures with the strong $\VC{}d$ property:

\begin{theorem}\label{VCdensity}
Suppose that $\mathbf M$ has the strong $\VC{}d$ property, and let
$r\in\bR$ such that
$$\pi_\Delta^*(t)=O(t^r)\qquad\text{for every  $\Delta(x;y)$ with $\abs{x}=1$.}$$ 
Then we have
$$\pi_\Delta^*(t)=O(t^{d(\abs{x}-1)+r})\qquad\text{for  every $\Delta(x;y)$.}$$
\end{theorem}

The proof of Theorem~\ref{VCdensity} proceeds by counting types via sets of representatives (and with an induction supported by the strong~VC{}$d$~property):
given a finite set $B\subseteq M^{\abs{y}}$, we say that 
$R\subseteq M^{\abs{x}}$ is a {\em set of representatives for $S^\Delta(B)$} if for every 
$q \in S^\Delta(B)$ there is $\alpha \in R$ realizing~$q$.
Equivalently, for every $a\in M^{\abs{x}}$ there is $\alpha \in R$ such that for every $\ph\in \Delta$ and every $b\in B$,
$\mathbf M \models \ph(a;b)$ if and only if $\mathbf M\models \ph(\alpha;b)$. Thus $\abs{S^\Delta(B)}\leq K$ iff there is a set of representatives for $S^\Delta(B)$ of size at most $K$.  

\begin{proof}
The proof is similar to that of Theorem~\ref{VCdensity, 2}.
We again induct on $m=\abs{x}$, with the case $m=1$ holding by hypothesis. For the inductive step write $x=(x_0,x')$ where $x'=(x_1,\dots,x_m)$, and let $\Delta(x;y)$ be given. We may assume that $\Delta$ is closed under negation.
As in the proof of Theorem~\ref{VCdensity, 2} let
$$
   \Delta_0(x_0;x',y) = \{\ph(x_0;x',y) : \ph(x;y)\in \Delta \}.
$$
By the strong $\VC{}d$ property applied to $\Delta_0$, we can take finitely many families
$$\mathcal F_{i} = \big(\ph_i(x',y;y_1,\ldots,y_d)\big)_{\varphi\in\Delta}\qquad (i\in I)$$ of 
$\mathcal L$-formulas with the following two properties:
for any $a'\in M^m$, any finite $B\subseteq M^{\abs{y}}$ and
any $q \in S^{\Delta_0}(a'B)$, there are $b_1,\ldots,b_d \in B$ and $i\in I$ 
such that $\mathcal F_{i}(a',y;b_1,\dots,b_d)$ defines $q$, i.e., for all $\ph\in\Delta$, $b\in B$:
\begin{equation}\label{eq:VCm, 4}
\varphi(x_0;a',b)\in q	\qquad\Longleftrightarrow\qquad \mathbf M\models \varphi_{i}(a',b;b_1,\dots,b_d);
\end{equation}
and for all $i\in I$, $a'\in M^m$ and $\bb\in (M^{\abs{y}})^d$,
the set
$$p_i(x_0;a',\bb):=\big\{\varphi(x_0;a',b) : \ph\in\Delta,\ b\in M^{\abs{y}},\ \mathbf M\models \varphi_{i}(a',b;\bb)\big\}$$
of $\mathcal L(M)$-formulas is consistent \textup{(}with $\mathbf M$\textup{)}.
In the rest of this proof let $\varphi$ range over $\Delta$ and $i$ over $I$.
For each $i$, let 
$$
	\Delta_i(x';y,y_1,\ldots,y_d) =\big\{\ph_i(x';y,y_1,\ldots,y_d) : \ph(x;y) \in \Delta\big\}
$$
and apply the inductive hypothesis to each $\Delta_i$. Thus there are constants $K_i$ such that
for any finite $C\subseteq (M^{\abs{y}})^{(d+1)}$ there is a set of representatives for
$S^{\Delta_i}(C)$ of size at most $K_i \abs{C}^{d(m-1)+r}$.

Now let a finite  $B\subseteq M^{\abs{y}}$ be given. We let $b$ range over $B$ and $\bb=(b_1,\ldots,b_d)$ over $B^d$. For each $\bb$ and each $i$, let
$R_i(B\bb)$ be a set of representatives for $S^{\Delta_i}(B\bb)$. Thus for any $a'\in M^m$ and $i$ there is some $\alpha\in R_i(B\bb)$ such that for any $\ph$ and $b$,
\begin{equation}\label{eq:VCm, 5}
	\mathbf M\models \ph_i(a';b,\bb) \qquad\Longleftrightarrow\qquad \mathbf M\models \ph_i(\alpha;b,\bb).
\end{equation}
Notice that there are $\abs{B}^d$ sequences $\bb$, and $\abs{B\bb}=\abs{B}$ for each $\bb$. 
As above, we may suppose  $|R_i(B\bb)|\leq K_i|B\bb|^{d(m-1)+r}=K_i|B|^{d(m-1)+r}$.

For each $i$, given $\alpha\in M^m$ and $\bb$, let $\delta_{i, \alpha,\bb}\in M$ realize the restriction of the type $p_i(x_0;\alpha,\bb)$ to the (finite) parameter set $\alpha B$. Let 
$$
	R_\Delta = \big\{ (\delta_{i,\alpha,\bb},\alpha) :i\in I,\ 
	\bb\in B^d,\ \alpha\in R_i(B\bb)\big\}
$$
and observe that 
$$
	|R_\Delta | \le \sum_{i} |B|^d\, K_i|B|^{d(m-1)+r} =  \left( \sum_{i} K_i\right) |B|^{md+r}.
$$
Thus we are finished once we have shown that $R_{\Delta}$ is a set of representatives for $S^\Delta(B)$.
For this, let $a=(a_0,a')\in M^{1+m}$ be given; we need to show that there is an element $(\delta_{i,\alpha,\bb},\alpha)\in R_{\Delta}$ such that for every $\ph$ and every $b$,
\begin{equation}\label{eq:VCm, 6}
\mathbf M\models \ph(a;b)\qquad\Longleftrightarrow\qquad\mathbf M\models \ph(\delta_{i,\alpha,\bb},\alpha;b).
\end{equation}
Let  $q=\tp^{\Delta_0}(a_0/a'B)\in S^{\Delta_0}(a'B)$. Take  $i$ 
and $\bb$ such that \eqref{eq:VCm, 4} holds for all $\varphi$ and all $b$, and then take a representative $\alpha\in R_i(B\bb)$ for $\tp^{\Delta_i}(a'/B\bb)$. 
Note that by \eqref{eq:VCm, 4} and since $\Delta$ is assumed to be closed under negation, for each $\ph$ and $b$, either $\mathbf M\models \varphi_i(a',b;\bb)$ or 
$\mathbf M\models \psi_i(a',b;\bb)$, where $\psi\in\Delta$ is equivalent to $\neg\varphi$;
hence also either $\mathbf M\models \varphi_i(\alpha;b,\bb)$ or 
$\mathbf M\models \psi_i(\alpha;b,\bb)$, by \eqref{eq:VCm, 5}, and thus
\begin{equation}\label{eq:VCm, 7}
\mathbf M\models\ph_i(\alpha;b,\bb)  \qquad\Longleftrightarrow\qquad \mathbf M\models \varphi(\delta_{i,\alpha,\bb};\alpha,b),
\end{equation}
by choice of $\delta_{i, \alpha,\bb}$. Combining \eqref{eq:VCm, 4}, \eqref{eq:VCm, 5} and \eqref{eq:VCm, 7} now yields \eqref{eq:VCm, 6} as required.
\end{proof}

In each of the cases treated in Theorems~\ref{weaklyomin} and \ref{quasiomin} one can show that the theory in question has, indeed, the {\it strong}\/ $\VC{}1$ property. Since this is of limited interest for computing VC~density (the $\VC{}1$ property already gives the optimal result $\vc(m)=m$), we do not give the details, and instead now turn to an application of Theorem~\ref{VCdensity} to $p$-adic examples. 

\subsection{$P$-minimal theories}\label{sec:Pmin}
Let $\calL_{{\rings}}$ be the language of rings, let 
$\calL_p=\calL_{{\rings}} \cup \{P_n:n>1\}$ where the $P_n$ are unary predicates, and let
$\calL$ be a language containing $\calL_p$. Here and below, $p$ is a fixed prime number.
Let $\pCF$ denote the $\calL_p$-theory  of $\bQ_p$, where each $P_n$ is interpreted as the set of $n$th powers in $\bQ_p$:
$$\bQ_p\models \forall x\big(P_n(x) \leftrightarrow \exists y(y^n=x)\big).$$
By a theorem of Macintyre \cite{mac}, $\pCF$ has elimination of quantifiers.
Following \cite{hm}, an $\calL$-theory $T$ containing $\pCF$ is called {\em $P$-minimal} if, in every 
model of $T$, every definable subset in one variable is quantifier-free definable just using the language~$\calL_p$. (In fact, the setting of \cite{hm} also allowed for $p$-adically closed fields of arbitrary fixed $p$-rank, and our methods here could be adjusted to that.)

By \cite{dhm}, a motivating example of a $P$-minimal theory is the theory
$\pCF_{\an}$ first investigated by Denef and van den Dries \cite{dd}. This is the theory of the $p$-adic numbers equipped with, for every $n>0$ and power series
$\sum_\nu a_\nu X^\nu\in\bQ_p[[X]]$ such that $\abs{a_\nu}\to 0$ as $\abs{\nu}\to \infty$, a function symbol
 $f$ of arity $n$ taking value identically zero off ${\mathbb Z}_p^n$, and such that
$f(x)=\sum_\nu a_\nu x^\nu$ for all $x \in \bZ_p^n$. (Here,
$X=(X_1,\dots,X_n)$,
$\abs{a}$ denotes the $p$-adic norm of $a\in\bQ_p$, 
$\nu=(\nu_1,\ldots,\nu_n)\in {\mathbb N}^n$ is a multi-index,  $\abs{\nu}=\nu_1+\dots+\nu_n$, and $x^\nu=x_1^{\nu_1}\cdots x_n^{\nu_n}$.)

\medskip

Our main result about VC density in $P$-minimal theories is:

\begin{theorem}\label{Pmin}
Let $T$ be a $P$-minimal $\calL$-theory with definable Skolem functions. Then $T$ has the strong~$\VC{}2$~property, and  
any finite set $\Delta(x;y)$ of $\mathcal L$-formulas has dual VC density at most $2\abs{x}-1$.
\end{theorem}

Before proving this theorem, we introduce some notation and establish some auxiliary facts.  We fix a model $K$ of $\pCF$, with valuation $v\colon K\to\Gamma_\infty$. We view $\mathbb Z$ as a convex subgroup of $\Gamma$, by identifying $1$ with $v(p)$. In the following, by a {\em ball}\/ in $K$ we always mean a closed ball, i.e.,
a set of the form 
$$B=B_\rho(a) = \big\{ x\in K : v(x-c)\ge \rho \big\}\qquad\text{where $c\in K$ and $\rho \in \Gamma$.}$$ 
Its {\em radius}, denoted $\rad(B)$, is $\rho$. By convention $\rad(K):=-\infty$.
Let $\B$ denote the set of all balls in $K$. There is a natural semilinear partial order on  ${\B}$, with $B\leq B'$ if and only
if $B \supseteq B'$. 
A ball $B=B_\rho(a)$ as above has a unique immediate predecessor, namely $B_{\rho-1}(a)$, and $p$ immediate successors, namely 
$B_{\rho+1}(a_i)$ where $a_i=a+ir$ for $i=0,\dots,p-1$; here $r$ is an arbitrary element of $K$ with $v(r)=\rho$.
Thus,
if we form a graph with vertex set ${\B}$, with 
vertices $B$, $B'$ adjacent if and only if one of $B$, $B'$ is an immediate successor of the other 
in the partial order, each of its connected components is an unrooted tree of valency $p+1$. For $B,B' \in {\B}$ we write $\dist(B,B')=d$ if $B$ and $B'$ are at distance $d$ in this graph; for each $B\in\B$, there are $(p+1)^d$ balls at distance $d$ to $B$, and $\beta_d:=\sum_{i=0}^d (p+1)^i=\frac{1}{p}((p+1)^{d+1}-1)$ balls at distance at most $d$ to $B$.
Note that $\dist$ is a metric on each connected component of $\B$.  

\begin{lemma}\label{lem:number of balls}
Let $A\subseteq K$ be finite, $A\neq\emptyset$. Then there at most $\abs{A}-1$ distinct balls of the form $B_{v(a-b)}(a)$ where $a,b\in A$, $a\neq b$.
\end{lemma}
\begin{proof}
We may assume $\abs{A}>1$.
For $a\in A$ set $$\mu_a:=1+\max\big\{v(a-b):b\in A,\ a\neq b\big\}.$$
Let $\mathcal B_A$ be the smallest connected subgraph of $\mathcal B$ containing all $B_{\mu_a}(a)$, $a\in A$. Note that the balls $B_{\mu_a}(a)$, $a\in A$, are pairwise disjoint; in particular, they are the leaves of the tree $\mathcal B_A$. The balls  $B_{v(a-b)}(a)$, where $a,b\in A$, $a\neq b$ are vertices of $\mathcal B_A$, and all but  one of them has degree greater than $2$.
Now use the fact that any (undirected) tree with finite vertex set $V$ with $\abs{V}>1$ has $2+\sum_{v\in V, \deg(v)>2} (\deg(v)-2)$ leaves. (This follows immediately from the well-known formula $2(\abs{V}-1)=\sum_{v\in V} \deg(v)$.)
\end{proof}

We also recall the following basic fact  (a consequence of the Newton formulation of Hensel's Lemma, see \cite[Lemma~2.3]{hm}) about the  subgroups $P_n^\times=P_n\setminus\{0\}$ of the multiplicative group  $K^\times=K\setminus\{0\}$ of~$K$:

\begin{lemma}\label{2.3}
Suppose $n>1$, and let $x,y,a\in K$ with $v(y-x)>2v(n)+v(y-a)$. Then $(x-a)(y-a)^{-1} \in P_n^\times$.
\end{lemma}

Suppose now that $T$ is an $\calL$-theory satisfying the hypothesis of Theorem~\ref{Pmin}, and $K\models T$. Employing definability of Skolem functions and the explicit description of immediate predecessors and successors in the partial order of $\B$ given above, one easily shows, by induction on~$d$:

\begin{lemma}\label{lem:defining balls}
Let $(B_b)_{b\in K^m}$ be a $\emptyset$-definable family of subsets of $K$. 
Then there exist $\emptyset$-definable functions $c_i,r_i\colon K^m\to K$, $i\in\mathbb N$, with the following property: if $b\in K^m$ is such that $B=B_b$ is a ball in $K$, then the balls $B_{v(r_i(b))}(c_i(b))$, where $i=1,\dots,\beta_d$, are exactly the balls of distance at most $d$ to $B$.
\end{lemma}

The assumption of definable Skolem functions also guarantees that any model of the theory $T$ has cell decomposition
\cite{m}. 
Let $\Delta(x;y)$ be a finite set of $\mathcal L$-formulas, where $\abs{x}=1$, closed under negation.
Then there are integers $N,n>0$, and for each $i=1,\dots,N$ there are
$\emptyset$-definable functions $f_i, g_i, c_i\colon K^{\abs{y}}\to K$ and
elements $\lambda_i$ of a fixed set of representatives of the cosets of the subgroup $P_n^\times$ of $K^\times$
with the following properties:
for any $\varphi\in\Delta$ and $b\in K^{\abs{y}}$, the set $\varphi(K;b)$ of realizations of $\varphi(x;b)$ is
a finite union of some of the \emph{cells} $U_1(b),\dots,U_N(b)$ defined by the data given above, i.e., sets of the form 
\begin{equation}\label{cells}
U_i(b)=	\big\{ x\in K : v(f_i(b)) \Box_{i1} v(x-c_i(b))\Box_{i2} v(g_i(b)) \,\&\, P_n(\lambda_i(x-c_i(b))) \big\}
\end{equation}
where each symbol $\Box_{ij}$ is $\leq$, $<$, or no condition. Note that this includes the cases where $U_i(b)=\{c_i(b)\}$ is a singleton, or where
\begin{equation}\label{eq:degenerate cells}
U_i(b)=\big\{ x\in K:  P_n(\lambda_i(x-c_i(b))) \big\}.
\end{equation}
The \emph{center}\/ of the cell $U_i(b)$ is given by $c_i(b)$. 
Since the value group of $K$ has smallest positive element $v(p)$, using the equivalences
\begin{align*}
v(a)<v(a')		&\quad\Longleftrightarrow\quad v(pa)\leq v(a')\\ 
v(a)\leq v(a')	&\quad\Longleftrightarrow\quad v(a)<v(pa'),
\end{align*}
valid for all $a,a'\in K$, not both zero, one sees that we may assume that in \eqref{cells}, $\Box_{i1}$ is $\leq$ or no condition, and $\Box_{i2}$ is $<$ or no condition. From now on, we assume for convenience that all our cells have this particular form.

\medskip

Based on the above data for the cell decomposition, we now describe a uniform definition of $\Delta(x;B)$-types over finite sets in $K$, in terms of the graph of balls $\B$. A \emph{special ball} is a ball having one of the following forms:
$$B_{v(c_i(b)-c_j(b'))}(c_i(b)),\quad B_{v(f_i(b))}(c_i(b)), \quad B_{v(g_i(b))}(c_i(b))\qquad (b,b'\in K^{\abs{y}}).$$
Note that each special ball can be defined by using at most two parameter tuples. 
We say that $B_{v(c_i(b)-c_j(b'))}(c_i(b))$,
$B_{v(f_i(b))}(c_i(b))$ and $B_{v(g_i(b))}(c_i(b))$ are special balls defined over $\{b,b'\}$.
We also say that a special ball is defined over a subset $B$ of $K$ if it is defined over $\{b,b'\}$ where $b,b'\in B$. By Lemma~\ref{lem:number of balls}, given a finite $B\subseteq K$, there are no more than $3N\cdot\abs{B}-1$ special balls defined over $B$.

Let us say that a ball $B'$ is \emph{near} a ball $B$ if $\dist(B,B')\leq n+4v(n)+2$; each ball has $\beta:=\beta_{n+4v(n)+2}$ balls near it. In particular, given $b_1,b_2\in K^{\abs{y}}$ there are at most $M:=(6N-1)\cdot\beta$ balls near special balls defined over $\{b_1,b_2\}$. Set $I^{(1)}:=[M]=\{1,\dots,M\}$ and $I^{(2)}=I^{(3)}:=[N]$.
By Lemma~\ref{lem:defining balls} there are
$\mathcal L$-formulas $\chi_i(x;y_1,y_2)$, $i\in I^{(1)}$, such that
for each $b_1,b_2\in K^{\abs{y}}$, the formulas $\chi_i(x;b_1,b_2)$ define exactly the balls near special balls defined over $\{b_1,b_2\}$. 
For each $i\in I^{(1)}$ and $\varphi\in\Delta$ define
$$\varphi^{(1)}_i(y;y_1,y_2) := \forall x \big(\chi_i(x;y_1,y_2) \rightarrow \varphi(x;y)\big).$$
So for $b,b_1,b_2\in K^{\abs{y}}$ we have:
$$K\models \varphi^{(1)}_i(b;b_1,b_2) \quad\Longleftrightarrow\quad \varphi(K;b)\supseteq \chi_i(K;b_1,b_2).$$
For each $i\in I^{(2)}=I^{(3)}$ and $\varphi\in\Delta$ set
$$\varphi_{i}^{(2)}(y;y_1) := \varphi(c_i(y_1);y)$$
and
$$\varphi_{i}^{(3)}(y;y_1) := \forall x\big( P_N(\lambda_i(x-c_i(y_1))) \to \varphi(x;y)\big).$$
Now set $\mathcal F^{(j)}_i:=(\varphi^{(j)}_i(y;y_1,y_2))_{\varphi\in\Delta}$ for each $i\in I^{(j)}$, $j=1,2,3$. The first part of Theorem~\ref{Pmin} will be proved once we show the following:

\begin{claim*}
$\mathcal F=(\mathcal F^{(j)}_i)$ is a coherent definition of $\Delta(x;B)$-types over finite sets.
\end{claim*}

Since the coherency condition is obviously satisfied, it is enough to show that $\mathcal F$ is a uniform definition of $\Delta(x;B)$-types over finite sets.
For this let $B\subseteq K^{\abs{y}}$ be finite and non-empty, and let $q\in S^\Delta(B)$. Let
$$c(B) := \big\{ c_i(b) : b\in B,\ i=1,\dots,N\big\}$$
be the set of centers of the cells $U_i(b)$. 
In the following we let $i$ range over $[N]=\{1,\dots,N\}$ and $b$ (possibly with decorations) over $B$. By a ``special ball'' we always mean a special ball defined over $B$, and similarly a ``near ball'' is a ball near a special ball (defined over $B$).

We first eliminate two special cases (which are taken care of by the families $(\mathcal F^{(2)}_i)$ and $(\mathcal F^{(3)}_i)$):
Suppose first that $q^K\cap c(B)\neq\emptyset$, say $c_i(b_1)\in q^K$ for some $i$ and $b_1\in B$. In this case, with such choice of $i$ and $b_1$, $\mathcal F^{(2)}_{i}(y;b_1)$ defines $q$. Similarly, if $\abs{c(B)}=1$, and for all  $i$ and $b$ with $q^K\subseteq U_i(b)$, the condition $\Box_{i1}$ is vacuous and $g_i(b)=0$, then all such cells $U_i(b)$ have the form as in \eqref{eq:degenerate cells}, and for suitable $i$ and $b_1$, 
$\mathcal F^{(3)}_i(y;b_1)$ defines $q$.

So from now on we may assume that:
\begin{itemize}
\item[(a)] $q^K$ is disjoint from $c(B)$; and 
\item[(b)] if $c(B)$ is a singleton, then for some $i$ and $b$ with $q^K\subseteq U_i(b)$ the condition $\Box_{i1}$ is ${\leq}$ or $g_i(b)\neq 0$.
\end{itemize}
Under these assumptions, it is enough to show: {\it there is a near ball $D$ such that $D\subseteq q^K$.}\/
We first note:

\begin{lemma}\label{eq:balls in cells}
Let $a\in K\setminus c(B)$, and let $B_1$ be a ball containing $a$ which is maximal subject to the condition $B_1\cap c(B)=\emptyset$; that is, $B_1=B_\delta(a)$ where
$$\delta=1+\max\big\{ v(a-c): c\in c(B)\big\}.$$
Let also $B_0:=B_{\delta+2v(n)}(a)$. Then for all $x\in K$ and $c\in c(B)$ we have:
\begin{enumerate}
\item $x\in B_1\Rightarrow v(x-c)=v(a-c)$;
\item $x\in B_0\Rightarrow (x-c)(a-c)^{-1}\in P_n$.
\end{enumerate}
In particular, if a cell $U_i(b)$ contains $a$ then it contains $B_0$.
\end{lemma}

The proof of (1) is obvious, and to deduce (2) from (1) use Lemma~\ref{2.3}.

\medskip

Let $a\in K$ realize $q$, and define $\delta$, $B_0$ and $B_1$ as in the previous lemma. Also, take $c\in c(B)$ such that $\delta=1+v(a-c)$. 

\begin{lemma}\label{lem:B2}
Let $B_2$ be a ball. Then
\begin{enumerate}
\item $B_2$ properly contains $B_1$ iff it contains both $a$ and $c$;
\item if $B_2$ contains $c$ but not $a$, then $\dist(B_1,B_2) = \rad(B_2) - \delta +2 $.
\end{enumerate}
\end{lemma}

The proof of (1) is clear, and for (2) note that if $c\in B_2$ and $a\notin B_2$, then $B_1,B_2\subseteq B_{\delta-1}(a)$.

\medskip

We first assume that there is a special ball $E$  such that $\dist(B_1,E)\leq 2v(n)+n+1$. 
Then by Lemma~\ref{eq:balls in cells}, $D := B_0 = B_{\delta+2v(n)}(a)$ is contained in those cells $U_i(b)$ which contain $a$; hence $D\subseteq q^K$. Also, $\dist(D,E)\leq n+4v(n)+2$, so $D$ is near $E$.
Hence the ball $D$ has the required properties.

\medskip

So from now on, we may suppose that for any special ball $E$ we have $\dist(B_1,E) > 2v(n)+n+1$.
We distinguish two cases: 

\subsubsection*{Case 1: there is a special ball which contains $B_1$.} We let $C$ be the smallest such special ball, with radius $\rho$.
We have $c\in C$ (since $C$ properly contains $B_1$) and hence $C=B_\rho(c)$. The idea now is to replace $a$ and the ball $D=B_0$ by another realization $a'\in C$ of $q$ and a ball $D'$ which is contained in and near the special ball $C$.
As $\Gamma$ is a $\bZ$-group, there is a unique $\delta'\in\Gamma$ such that
$$\rho+2v(n)+1<\delta'\leq \rho+2v(n)+1+n\qquad\text{and}\qquad \delta'\equiv\delta\bmod n.$$
By assumption we have 
$$2v(n)+n+1<\dist(B_1,C)=\delta-\rho$$ 
and hence $\delta>\delta'>\rho+1$. Now choose $d\in P_n$ with $v(d)=\delta'-\delta$. (From now on until the end of the proof of Theorem~\ref{Pmin} we temporarily suspend our promise of $d$ always denoting a natural number.) Put
$$D' := B_{\delta'+2v(n)+1}(a')\qquad\text{where $a':=d(a-c)+c$.}$$
Then $D'$ is contained in the special ball $C$, and $D'$ is near $C$.
Indeed, $D'\subseteq B_{\delta'-1}(a')$, and these two balls are at distance $2v(n)+2$. The latter ball also contains $c$, so 
$$\dist\big(B_{\delta'-1}(a'),C\big)=\dist\big(B_{\delta'-1}(c),B_\rho(c)\big) = \delta'-1-\rho \leq 2v(n)+n,$$
hence $\dist(D',C)\leq 4v(n)+2+n$.
Thus $D'$ has the right properties, provided we manage to show:

\claim[1]{Let $U_i(b)$ be a cell as in \eqref{cells} which contains $a$. Then $D'\subseteq U_i(b)$.}

\medskip

Towards the proof of this claim, we first show two auxiliary claims:

\claim[2]{Let $c'\in c(B)$. Then}
$$v(a'-c')=\begin{cases}
\delta'-1 & \text{if $v(c-c')\geq v(a-c)$,}\\
v(c-c')\leq\rho<\delta'-1 & \text{otherwise.}
\end{cases}$$
\begin{proof}
If $v(c-c')\geq v(a-c)$ then
$$v(c-c')\geq v(a-c)=\delta-1>\delta'-1$$
and hence
$$v(a'-c')=v(d(a-c)+(c-c'))=\delta'-1.$$
So suppose $v(c-c') < v(a-c)$. We have $v(a-c)>v(c-c')$, so
$v(a-c')=v(c-c')$ and hence $a$ is contained in the special ball $E:=B_{v(c-c')}(c)$. In fact, for every $x\in B_1$ we have
$$v(x-a)\geq\delta=v(a-c)+1>v(a-c)>v(c-c')$$
and hence $B_1\subseteq E$. By minimality of $C$ thus $C\subseteq E$. This yields
$\delta'-1>\rho\geq v(c-c')$
and thus $v(a'-c')=v(c-c')<\delta'-1$.
\end{proof}

\claim[3]{For every $c'\in c(B)$ we have $v(a-c')\geq v(a'-c')$.}

\begin{proof}
Certainly, $v(a-c')\geq\min\{v(a-a'),v(a'-c')\}$. But the minimum is always achieved by $v(a'-c')$, as 
$$v(a-a')=v((d-1)(a-c))=\delta'-1\geq v(a'-c')$$
by Claim~2.
\end{proof}


By Claim~2 and Lemma~\ref{eq:balls in cells} (applied to $a'$ in place of $a$), in order to show Claim~1, it is enough to prove that $a'\in U_i(b)$. We abbreviate $c'=c_i(b)$.
Suppose that $\Box_{i1}$ is ${\leq}$. Then $f_i(b)\neq 0$, and by Lemma~\ref{eq:balls in cells},~(1), all elements $x$ of $B_1$ satisfy the condition $v(f_i(b)) \leq v(x-c')$;  hence
$C$ is contained in the special ball $B_{v(f_i(b))}(c')$, by the minimality of $C$. Since $D'\subseteq C$, all elements $x$ of $D'$ also satisfy $v(f_i(b)) \leq v(x-c')$; in particular, of course, $v(f_i(b)) \leq v(a'-c')$. If $\Box_{i2}$ is ${<}$, then by Claim~3, $v(a'-c')\leq v(a-c')<v(g_i(b))$, as required.
It remains to check that $a-c'$ and $a'-c'$ lie in the same coset of $P_n^\times$. We distinguish two cases. 
If $v(c-c')\leq\rho$ then $v(a-c)=\delta-1>\rho\geq v(c-c')$, hence $v(a-c')=v(c-c')$
and thus
$$v(a-a')=\delta'-1>2v(n)+\rho\geq 2v(n)+v(a-c');$$
therefore $a-c'$ and $a'-c'$ are in the same $P_n^\times$-coset, by Lemma~\ref{2.3}.
Suppose $v(c-c')>\rho$. Then by Claim~2 we have $v(c-c')\geq v(a-c)=\delta-1$ and $v(a'-c')=\delta'-1$.
Now consider the special ball $E:=B_{v(c-c')}(c)$. Note that $a\notin E$: otherwise $v(c-c')=v(a-c)=\delta-1$ and hence $B_1=B_\delta(a)\subseteq E$ with $\dist(B_1,E)=1$, contrary to our initial assumption (made before Case~1). Thus, by Lemma~\ref{lem:B2},~(2) and said assumption, we obtain $2v(n)+n+1 < v(c-c')-\delta+2$.
Hence 
$$v(c-c')>2v(n)+\delta-1=2v(n)+v(a-c)  \geq 2v(n)+v(a'-c'),$$
with the last inequality by Claim~3.
So by Lemma~\ref{2.3}, $a-c'$ and $a-c$ are in the same $P_n^\times$-coset, as are $a'-c$ and $a'-c'$.
Certainly, as $a'-c=d(a-c)$ and $d\in P_n^\times$, the elements $a-c$ and $a'-c$ lie in the same coset of $P_n^\times$. 
Hence $a-c'$ and $a'-c'$ also lie in the same coset of $P_n^\times$.
This finishes the proof of Claim~1, and hence of Case~1. \qed

\subsubsection*{Case 2: no special ball contains~$B_1$.} In this case, for every $c'\in c(B)$, the special ball $C=B_{v(c-c')}(c)$ does not contain $a$, so $v(c'-c)>v(a-c)=\delta-1$ and hence
$$v(a-c')=\min\{v(a-c),v(c-c')\}=v(a-c)=\delta-1.$$ 
Since $C$ is of distance greater than $2v(n)+n+1$ from $B_1$,  by part~(2) of Lemma~\ref{lem:B2} we also obtain
\begin{equation}\label{eq:v(c-c')}
v(c-c') > \delta+2v(n)+n-1.
\end{equation}
Similarly, since each special ball $B_{v(g_i(b))}(c_i(b))$ does not contain $a$, we get
$$v(g_i(b)) > \delta+2v(n)+n-1,$$
and since $a\notin B_{v(f_i(b))}(c_i(b))$, the condition $\Box_{i1}$ is vacuous for each $i$ and $b$ with $a\in U_i(b)$.

Fix a special ball $E$ of the form $B_{v(g_i(b))}(c_i(b))$ with minimal radius $\gamma=v(g_i(b))$, if there is such a special ball; otherwise let $\gamma=\infty$.
Also, if $\abs{c(B)}>1$,  let $C$ be a special ball of the form
$B_{v(c-c')}(c)$, where $c'\in c(B)$, with minimal radius $\rho=v(c-c')$; we set $\rho=\infty$ if $\abs{c(B)}=1$. 
Note that by our general assumption (made before Lemma~\ref{eq:balls in cells}), not both of $\gamma$ and $\rho$ are $\infty$. We now distinguish two subcases:

\subsubsection*{Case 2a: $\rho-2v(n)\leq\gamma$}
Let $\delta'\in\Gamma$ such that 
$$\rho-2v(n)-n<\delta'\leq \rho-2v(n), \qquad \delta'\equiv\delta\mod n,$$
choose $d\in P_n$ with $v(d)=\delta'-\delta$, and set
$$D' := B_{\delta'+2v(n)}(a')\qquad\text{where $a':=d(a-c)+c$.}$$
By \eqref{eq:v(c-c')} we have $\delta'>\delta$.
Moreover, for each $c''\in c(B)$ we have $v(c-c'')\geq\rho > \delta'-1=v(d(a-c))$ and hence $v(a'-c'')=\delta'-1$.
Note that the ball $B_{\delta'-1}(a')$ contains $D'$ and is of distance $2v(n)+1$ to $D'$. The ball $B_{\delta'-1}(a')$ contains $c$, hence
$$\dist(B_{\delta'-1}(a'),C) \leq \rho-(\delta'-1) < 2v(n)+n+1$$
and thus $\dist(D',C)<4v(n)+n+2$, so $D'$ is near $C$.
Let $U_{i}(b)$ be a cell containing $a$; it remains to show that then $a'\in U_{i}(b)$. We already noted that condition $\Box_{i1}$ is vacuous. As to $\Box_{i2}$, suppose that condition is $<$. Writing $c'=c_i(b)$ we then have
$$v(a'-c')=\delta'-1<\rho-2v(n) \leq \gamma\leq v(g_i(b))$$
as required. Finally, by \eqref{eq:v(c-c')} and Lemma~\ref{2.3}, $a-c'$ and $a-c$ are in the same $P_n^\times$-coset, and since
$$v(c-c')\geq \rho > 2v(n)+\delta'-1=2v(n)+v(a'-c'),$$
the elements $a'-c$ and $a'-c'$ are also in the same $P_n^\times$-coset. As $a'-c=d(a-c)$ and $a-c$ are in the same $P_n^\times$-coset, finally $a'-c'$ and $a-c'$ are in the same $P_n^\times$-coset, as required. \qed

\subsubsection*{Case 2b: $\rho-2v(n) > \gamma$}
In this case we let $\delta'\in\Gamma$ be such that 
$$\gamma-2v(n)-n<\delta'\leq \gamma-2v(n), \qquad \delta'\equiv\delta\mod n,$$
and with this choice of $\delta'$ define $d$, $a'$ and $D'$ as in Case~2a. 
Note that $\rho > \gamma$, so for each $c''\in c(B)$ we have $v(c-c'') > \delta'-1=v(d(a-c))$ and hence $v(a'-c'')=\delta'-1$.
Since 
$$\dist(B_{\delta'-1}(a'),E) \leq \gamma-(\delta'-1) < 2v(n)+n+1$$
we see, similarly as in Case~2a, that $\dist(D',E)<4v(n)+n+2$, so $D'$ is near $E$.
Let $U_{i}(b)$ be a cell containing $a$, and suppose $\Box_{i2}$ is ${<}$.
Then $v(a'-c')=\delta'-1<\gamma\leq v(g_i(b))$, and as at the end of Case~2a one sees that $a'-c'$ and $a-c'$ are in the same coset of~$P_n^\times$. \qed

\medskip

To complete the proof of the theorem,  we apply Theorem~\ref{VCdensity} with $r=1$. 
By what we have shown above, for every finite non-empty $B\subseteq M^{\abs{y}}$, each type in $S^\Delta(B)$ is uniquely determined by either a center $c_i(b)$, where $b\in B$, or a near ball. However,
there are at most $N\abs{B}=O(\abs{B})$ centers, and at most $(3N\abs{B}-1)\cdot\beta=O(\abs{B})$ near balls; thus $\abs{S^\Delta(B)}=O(\abs{B})$ as required. \qed

\medskip

From Theorem~\ref{Pmin} and Corollary~\ref{cor:dp-min} we obtain:

\begin{corollary}\label{cor:Pmin}
Every $P$-minimal theory with definable Skolem functions is dp-minimal.
\end{corollary}

\begin{remark} \label{analytic}
In \cite[Section~6]{dl}, Dolich, Goodrick and Lippel already showed that $\pCF=\Th(\bQ_p)$ is dp-minimal.
By 3.6 of \cite{dd}, the $P$-minimal theory $\pCF_{\an}$ has definable Skolem functions.
(Formally, \cite{dd} handles the corresponding subanalytic structure on ${\mathbb Z}_p$, but the translation is straightforward.) 
Note that the proof of \cite[3.6]{dd} takes place in the ground model ${\mathbb Z}_p$, where all elements are named by constants,
so `definable' means  `$\emptyset$-definable', and curve selection, as stated there, gives definable Skolem functions.
Hence the conclusion of Theorem~\ref{Pmin} and Corollary~\ref{cor:Pmin} apply to it and its reducts. (See also Lemma~\ref{lem:expansions by constants}.) Cell decomposition in $\pCF_{\an}$ is also proved in \cite{cluckers}.
\end{remark}

\end{document}